\documentclass[11pt]{amsart}

\allowbreak
\usepackage{amsfonts,amssymb,amscd,amsmath,latexsym,amsbsy}
\usepackage{graphicx}
\usepackage{amssymb}
\usepackage{amsfonts}
\usepackage{latexsym}
\usepackage{mathrsfs}
\usepackage{mathtools}
\usepackage{amsthm}
\usepackage{listings}
\usepackage{comment}
\usepackage{tikz-cd}
\usepackage{bookmark}
\usepackage{enumitem}
\usepackage{float}

\allowdisplaybreaks
\newtheorem{thm}{Theorem}[section]
\newtheorem{defn}[thm]{Definition}
\newtheorem{lmm}[thm]{Lemma}
\newtheorem{prop}[thm]{Proposition} 
\newtheorem*{rmk*}{Remark} 
\newtheorem{rmk}[thm]{Remark}
\newtheorem{coro}[thm]{Corollary}

\usepackage{mathtools}

\numberwithin{equation}{section}

\newcommand{\R}{\ensuremath{\mathbb{R}}}
\newcommand{\Z}{\ensuremath{\mathbb{Z}}}
\newcommand{\N}{\ensuremath{\mathbb{N}}}

\newcommand{\WF}{\ensuremath{\mathrm{WF}}}
\newcommand{\Ell}{\ensuremath{\mathrm{Ell}}}
\newcommand{\supp}{\ensuremath{\mathrm{supp}}}
\newcommand{\cu}{\ensuremath{\mathrm{cu}}}

\newcommand{\mk}{\ensuremath{\mathfrak}}
\newcommand{\tb}{\ensuremath{\textbf}}

\newcommand{\la}{\ensuremath{\langle}}
\newcommand{\ra}{\ensuremath{\rangle}}

\newcommand{\Op}{\ensuremath{\mathrm{Op}}}
\newcommand{\sgn}{\ensuremath{\mathrm{sgn}}}
\newcommand{\Id}{\ensuremath{\mathrm{Id}}}
\newcommand{\Sp}{\ensuremath{\mathrm{Sp}}}

\newcommand{\cualpha}{\ensuremath{\mathrm{cu},\alpha}}

\begin{document}
\setlength{\emergencystretch}{3em}
\title{Propagation of Singularities with Normally Hyperbolic Trapping}

\author{Qiuye Jia}

\begin{abstract}
We prove a new microlocal estimate with normally hyperbolic trapping, which can be applied to Kerr and Kerr-de Sitter spacetimes. 
We use a new type of symbol class, and corresponding operator class, which is constructed by blowing up the intersection of the unstable manifold and fiber infinity. For scalar wave equations on Kerr and Kerr-de Sitter spacetimes, the extra loss of the microlocal estimates compared with the standard propagation of singularities without trapping is arbitrarily small.
\end{abstract}

\maketitle

\tableofcontents

\section{Introduction}
\subsection{Background} \label{Background}
The Kerr(-de Sitter) spacetime (before compactifying at time infinity) is a manifold $M^\circ$ equipped with a Lorentzian metric $g_{\mk{m},\mk{a}}$ such that
\begin{align}
M^\circ=\R_t \times X , \quad X=(r_e,r_c) \times \mathbb{S}^2.
\label{kds_manifold}
\end{align}
For detailed definitions of $r_e,r_c,g_{\mk{m},\mk{a}}$, please refer to Sections \ref{main_results} and \ref{kds_intro}. $(M^\circ,g_{\mk{m},\mk{a}})$ describes a rotating black hole. It is determined by two parameters: the angular momentum $\mk{a}$ and the mass $\mk{m}$. Our paper concerns the regularity of solutions to wave equations on Kerr(-de Sitter) spacetimes by proving propagation of singularities type estimates. Propagation of singularities builds a connection between the regularity of a solution to a partial differential equation at different locations on null-bicharacteristics, which are integral curves of the Hamilton vector field associated to that equation. The propagation with finite time, i.e. the parameter of null-bicharacteristics, was considered by H\"ormander \cite{hormander_propagation} and his work joint with Duistermaat \cite{duistermaat1972fourier}. The propagation with infinite time bifurcates into radial point estimates and trapping estimates. An important form of the former was proved by Vasy \cite{vasy2013microlocal}. Trapping estimates, in particular with normally hyperbolic trapping, are developed by a series of works.

In \cite{wunsch2011resolvent}, Wunsch and Zworski proved high energy resolvent estimates (microlocally near the trapped set) with a logarithmic order loss of the regularity when the spectral parameter is on the real line and a polynomial order loss when the spectral parameter is in a thin band below the real line. 
And they combined this with standard theory on asymptotically Euclidean manifolds (playing the role of the spatial slice in a stationary spacetime) to obtain exponential energy decay for certain wave equations on stationary spacetimes with asymptotically Euclidean spatial slices.
Hintz \cite{hintz2021normally} proved propagation estimates with one extra Sobolev order loss (extra compared with non-trapping estimates) for weighted Sobolev spaces on Kerr(-de Sitter) spacetimes and its time-dependent perturbations, which is an important motivation for our work.

In this paper we prove propagation estimates with arbitrarily small extra loss of regularity compared with the classical non-trapping propagation estimates when the operator is symmetric up to subleading order.  
And this loss of regularity is necessary according to \cite{sbierski2015characterisation}.
One major application of estimates of this type is to wave equations on Kerr(-de Sitter) spacetimes. The novelty of this paper is that we associate different orders to different boundary faces in the blown-up compactified cotangent bundle. This further decomposes the positive commutator argument in \cite{hintz2021normally}.

\subsection{The main result}
\label{main_results}
The metric $g_{\mk{m},\mk{a}}$ on the Kerr(-de Sitter) spacetime $M^\circ$ with angular momentum $\mk{a}$ and mass $\mk{m}$ is given by 
\begin{align}
\begin{split}
g_{\mk{m},\mk{a}} =& -(r^2+\mk{a}^2\cos^2\theta)
\big(\frac{dr^2}{\Delta(r)}+\frac{d\theta^2}{\Delta_\theta}\big)
-\frac{\Delta_\theta \sin^2\theta}{\Delta_0^2(r^2+\mk{a}^2\cos^2\theta)}
\big(\mk{a}dt-(r^2+\mk{a}^2)d\varphi\big)^2\\
&+\frac{\Delta(r)}{\Delta_0^2(r^2+\mk{a}^2\cos^2\theta)}
\big(dt-\mk{a}\sin^2\theta d\varphi\big)^2,
\end{split}
\label{kds_metric}
\end{align}
where $\Lambda$ is the cosmological constant and
\begin{align}
\begin{split}
&\Delta(r) = (r^2+\mk{a}^2)(1-\frac{\Lambda r^2}{3})-2\mk{m}r, \quad \Delta_\theta=1+\frac{\Lambda \mk{a}^2}{3}\cos^2\theta,\\
&\Delta_0 = 1+\frac{\Lambda \mk{a}^2}{3}, \quad \Lambda \geq 0.
\end{split}
\label{delta_intro}
\end{align}
Specifically, the case $\Lambda>0$ is called the Kerr-de Sitter spacetime, while the case $\Lambda=0$ is called the Kerr spacetime.

The wave operator $\Box_{g_{\mk{m},\mk{a}}}$ on $(M^\circ,g_{\mk{m},\mk{a}})$ is a second order differential operator and its principal symbol is the dual metric function $G_{\mk{m},\mk{a}}(z,\zeta):=|\zeta|^2_{g_{\mk{m},\mk{a}}^{-1}(z)}$, where $\zeta \in T^*_zM^\circ$. The characteristic set $\Sigma_{\mk{m},\mk{a}}$ is defined to be 
\begin{align*}
\Sigma_{\mk{m},\mk{a}}:= \{ (z,\zeta) \in T^*M^\circ\backslash o : G_{\mk{m},\mk{a}}(z,\zeta) = 0   \},
\end{align*}
where $o$ is the zero section in $T^*M$. Denote momentum variables dual to $t,r,\varphi,\theta$ by $\xi_t,\xi_r,\xi_\varphi,\xi_\theta$ respectively, then $\Sigma_{\mk{m},\mk{a}}$ has two components
\begin{align*}
\Sigma_{\pm} := \Sigma_{\mk{m},\mk{a}} \cap \{\pm g_{\mk{m},\mk{a}}^{-1}(\zeta,dt) > 0\}.
\end{align*}
Since all objects concerned here are homogeneous with respect to the dilation of fiber variables, we can pass to the cosphere bundle which is obtained by identifying orbits of the $\R^+$ action given by dilations in the fibers: $S^*M^\circ=(T^*M^\circ\backslash o)/\R^+$. 
The rescaled vector field $\textbf{H}_{G_{\mk{m},\mk{a}}}:=(g_{\mk{m},\mk{a}}^{-1}(\zeta,dt))^{-1}H_{G_{\mk{m},\mk{a}}}$ is a homogeneous (with respect to the fiber dilation) vector field of degree 0, and thus can be viewed as a vector field on $S^*M^\circ$. The main feature we use is that there exists a trapped set for $\textbf{H}_{G_{\mk{m},\mk{a}}}$ in $\Sigma_{\mk{m},\mk{a}}$ where the $\textbf{H}_{G_{\mk{m},\mk{a}}}$-flow is $r$-normally hyperbolic for every $r$ in the sense of \cite{wunsch2011resolvent}. See Section \ref{application} for more discussion.
The trapped set is given by
\begin{align*}
\Gamma_0:=\{ (z,\zeta) \in \Sigma_{\mk{m},\mk{a}}: \xi_r=r-r_{\xi_t,\xi_\varphi}=G_{\mk{m},\mk{a}}=0 \}.
\end{align*}
The definition of $r_{\xi_t,\xi_\varphi}$ is given in Proposition \ref{prop_trap}. Null-geodesics starting from $\Gamma_0$ never escape through the event horizon $\{r=r_e\}$ or to `spatial' infinity. Instead, when projected to $X$, they stay in the compact set $\{r=r_{\xi_t,\xi_\varphi}\}$.  In addition, this trapped set has the form $\Gamma_0=\Gamma_0^u \cap \Gamma_0^s$. $\Gamma_0^{u/s}$ are unstable/stable manifolds respectively, consisting of $(z,\zeta)\in \Sigma_{\mk{m},\mk{a}}$ such that the backward/forward integral curve starting at $(z,\zeta)$ tends to $\Gamma_0$. Both $\Gamma^{u/s}_0$ are conic codimension 1 submanifolds of $\Sigma_{\mk{m},\mk{a}}$ given by
\begin{align}
\Gamma^{u/s}_0:= \{{\varphi}^{u/s} =0 \} \cap \Sigma_{\mk{m},\mk{a}},  \label{defn_Gammau/s}
\end{align}
where ${\varphi}^{u/s} := \xi_r \mp  \sgn(r-r_{\xi_t,\xi_\varphi})(1+\hat{\alpha})\sqrt{\frac{F_{\xi_t,\xi_\varphi}(r)-F_{\xi_t,\xi_\varphi}(r_{\xi_t,\xi_\varphi})}{\Delta(r)}}$, $r_{\xi_t,\xi_\varphi}$ is defined in Proposition \ref{prop_trap}, $F_{\xi_t,\xi_\varphi}(r):= \frac{1}{\Delta(r)}( (r^2+\mk{a}^2)\xi_t+\mk{a}\xi_\varphi )^2$, $\Delta(r)$ is defined in (\ref{delta_intro}), and $\hat{\alpha}=\frac{\Lambda\mk{a}^2}{3}$.

Our key analytic tool is the cusp pseudodifferential algebra, which we recall below. We recommend \cite{hintz2021normally}\cite{mazzeo1998pseudodifferential}\cite{vasy2018minicourse} as further references on this topic. Our construction is on Kerr(-de Sitter) spacetimes, but the same construction applies to other manifolds with boundary. First we compactify $M^\circ$ to $M= (M^\circ \sqcup ([0,\infty)_\tau \times X))/\sim$, where $\sim$ identifies $(t,x) \in \R_t \times X$ and $(\tau=t^{-1},x)$ for $0<t<\infty$. The cusp vector fields are
\begin{align*}
\mathcal{V}_{\mathrm{cu}}(M):=\{ V \in \mathcal{V}(M): V\tau \in \tau^2 \mathcal{C}^\infty(M) \}.
\end{align*}
Away from $\{\tau=0\}$, $\mathcal{V}_{\mathrm{cu}}(M)$ is the same as $\mathcal{V}(M)$. Let $x=(x_1,...,x_{n-1})$ be coordinates on $X$, then near $\{\tau=0\}$, $\mathcal{V}_{\mathrm{cu}}(M)$ is locally spanned by 
\begin{align}
\tau^2\partial_\tau,\partial_{x_1},\partial_{x_2},...\partial_{x_{n-1}} \label{cusp_frame}
\end{align}
as a $\mathcal{C}^\infty(M)$-module. And vector fields in $\mathcal{V}_{\mathrm{cu}}(M)$ can be viewed as smooth sections of a vector bundle ${}^{\mathrm{cu}}TM$ called the cusp tangent bundle. (\ref{cusp_frame}) also represents a local frame of ${}^{\mathrm{cu}}TM$.
The class of $m$-th order cusp differential operators $\mathrm{Diff}^m_{\mathrm{cu}}(M)$ then consists of sums of products of up to $m$ cusp vector fields. As an element of $\mathcal{V}_{\mathrm{cu}}(M)$, $\tau^2\partial_\tau$ is non-vanishing even down to $\tau=0$, which is similar to $\tau\partial_\tau$ being non-vanishing in Melrose's b-calculus. 

The cusp cotangent bundle ${}^{\mathrm{cu}}T^*M$ is the dual bundle of ${}^{\mathrm{cu}}TM$. Locally it is spanned by
\begin{align}
\frac{d\tau}{\tau^2},dx^1,dx^2,...,dx^{n-1}.                     \label{cusp_co_frame}
\end{align}
 Writing covectors in ${}^{\mathrm{cu}}T^*M$ as
\begin{align*}
-\sigma \frac{d\tau}{\tau^2} + \sum_{i=1}^{n-1}\xi_i dx^i = \sigma dt + \sum_{i=1}^{n-1}\xi_i dx^i,
\end{align*}
then to a differential operator
\begin{align*}
P = \sum_{j+|\alpha|}a_{j\alpha}(\tau,x)(-\tau^2D_\tau)^jD_x^\alpha \in \mathrm{Diff}^m_{\mathrm{cu}}(M), \quad a_{j\alpha} \in \mathcal{C}^\infty(M),
\end{align*}
we associate a function called its principal symbol:
\begin{align*}
\sigma_{\mathrm{cu}}^m(P):=\sum_{j+|\alpha|=m} a_{j\alpha}\sigma^j\xi^\alpha.
\end{align*}
We write $\sigma(P)$ when there is no confusion about the order and in which pseudodifferential algebra we are discussing this operator. More importantly, $\mathrm{Diff}^m_{\mathrm{cu}}(M)$ can be generalized by allowing symbols to be more general functions other than polynomials in fiber variables. The class of $m$-th order symbols $S^m_{\mathrm{cu}}(M)$ consists of functions $a\in \mathcal{C}^\infty({}^{\mathrm{cu}}T^*M)$ such that
\begin{align} \label{defn_cu_symbol}
|\partial_\tau^j \partial_x^\alpha\partial_\sigma^k\partial_\xi^\beta a(\tau,x,\sigma,\xi)| \leq C_{j\alpha k \beta} (1+|\sigma|+|\xi|)^{m-k-|\beta|}
\end{align}
in terms of local coordinates. $a$ is said to be a symbol of order $m$ on a cone if $a$ satisfies this estimate on this cone. $a$ is said to be of order $-\infty$ if it satisfies this estimate all $m$ (with constants depending also on $m$). Away from $\{\tau=0\}$, since $\tau^2\partial_\tau=-\partial_t$, this condition remains the same when we replace $\partial_\tau$ by $\partial_t$, which is how we define the uniform symbol class. The quantization $\Op(a)$ of $a \in S_{\mathrm{cu}}^m(M)$ is a pseudodifferential operator acting on smooth functions $u \in \dot{\mathcal{C}}^\infty(M)$ supported in a single coordinate patch near $\{\tau=0\}$ by
\begin{align*}
\Op(a)u(t,x) = (2\pi)^{-n}\int e^{i((t-t')\sigma+(x-x')\xi)}a(t^{-1},x,\sigma,\xi)u(t',x')dt'dx'd\sigma d\xi.
\end{align*}
For general $u$, we define this action using a partition of unity. We denote the collection of all such $\Op(a)$ by $\Psi_{\mathrm{cu}}^m(M)$. Then we define the cusp wavefront set $\WF'_{\mathrm{cu}}(A)$ of an operator $A=\Op(a)$ by defining its complement.
\begin{defn}
For $\mk{z} \in {}^{\mathrm{cu}}T^*M$, we say that $\mk{z} \notin \WF'_{\mathrm{cu}}(A)$ when there exists a cone containing $\mk{z}$ on which $a$ is of order $-\infty$ (i.e. satisfies (\ref{defn_cu_symbol}) for all $m \in \R$).
\label{defn_cusp_wf}
\end{defn}
This is a conic set and we identify it with its image in the quotient space ${}^{\mathrm{cu}}S^*M$. Next we define cusp $L^2$-spaces and Sobolev spaces. The cusp cotangent bundle is locally spanned by $\frac{d\tau}{\tau^2},dx_1,dx_2,...,dx_{n-1}$, whose wedge product gives the cusp density $\nu_{\mathrm{cu}}$. $L_{\mathrm{cu}}^2(M)$ consists of functions that are supported on $\{ \tau \leq 1 \}$ and square integrable with respect to density $\nu_{\mathrm{cu}}$, and it is equipped with the norm $||u||_{L_{\mathrm{cu}}^2(M)}:=(\int_M |u|^2d\nu_{\mathrm{cu}})^{\frac{1}{2}}$. For $r\in \R$, the $r$-weighted cusp $L^2$-space is defined by
\begin{align*}
L_{\mathrm{cu}}^{2,r}(M):=\{  u \in L_{\mathrm{loc}}^2(M):
\supp u \subset \{ \tau \leq 1 \}, \tau^{-r}u \in L_{\mathrm{cu}}^2(M) \},
\end{align*}
where $L_{\mathrm{loc}}^2(M)$ is the space of locally $L^2$-integrable function class on $M$. For a non-negative integer $s$, we define the weighted cusp Sobolev space as
\begin{align*}
H_{\mathrm{cu}}^{s,r}(M):=\{ u \in L_{\mathrm{cu}}^{2,r}: Au \in L_{\mathrm{cu}}^{2,r} \quad \text{for all } A \in \mathrm{Diff}_{\mathrm{cu}}^s(M)  \}.
\end{align*}
For general $s>0$, $H_{\mathrm{cu}}^{s,r}(M)$ is defined by interpolation. For $s<0$, $H_{\mathrm{cu}}^{s,r}(M)$ is defined to be the dual space of $H_{\mathrm{cu}}^{-s,-r}(M)$.

The operator of our major concern is $P \in \Psi_{\mathrm{cu}}^{2}(M)$ with real principal symbol $p= \sigma^2_{\mathrm{cu}}(P)$ and characteristic set $\Sigma:=p^{-1}(0) \subset \prescript{\mathrm{cu}}{}{T}^*M \backslash o$, where $o$ is the zero section of $\prescript{\mathrm{cu}}{}{T}^*M$. The order $2$  can be replaced by other numbers, and we make this choice in our statement of the theorem because of its most important application: scalar wave equations on asymptotically Kerr and Kerr-de Sitter spacetimes, which has differential order 2.

Let $\hat{\rho}$ be a defining function of fiber infinity in ${}^{\cu}T^*M$ and $\tb{H}_p=\hat{\rho}H_p$ be the rescaled Hamilton vector of $p$. 
The key property we need $P$ to satisfy is that the trapped set for the $\tb{H}_p$-flow is \emph{normally hyperbolic trapping} in the sense of Section \ref{sec_assumptions}. 
Let $\Gamma^u,\Gamma^s$ be the unstable and stable manifolds of the $\tb{H}_p$-flow respectively and $\Gamma = \Gamma^u \cap \Gamma^s \cap \{ \tau=0 \}$ is the trapped set at time infinity. See Section \ref{sec_assumptions} for detailed definitions, and see Proposition \ref{prop: trappedset} for their explicit characterization in exact Kerr(-de Sitter) spacetimes, Theorem \ref{thm: structure of perturbed dynamics} for the setting of asymptotically Kerr-(de Sitter) spacetimes.  

Our main result, the propagation of singularities says that for $Pv=f$,  when one has control of weighted Sobolev norms of $f$ near $\Gamma$ and weighted Sobolev norms of $v$ near the part of $\Gamma^u$ except $\Gamma$, then one can propagate this control of $v$ into $\Gamma$, with a loss of differential order $(1-\lambda\alpha)$. More precisely, we have: 
\begin{thm}
Let $P \in \Psi_{\mathrm{cu}}^{2}(M)$ with $\tb{H}_p$-flow being normally hyperbolic in the sense of Section~\ref{sec_assumptions}, and its sub-principal symbol ${\bf p}_1 = \hat{\rho} \sigma_{\cu}^1(\frac{1}{2i}(P-P^*))$ satisfies
\begin{align} \label{eq: p1 condition, in main thm}
\sup_{\Gamma} {\bf p}_1 < \frac{1}{2} \nu_{\min},
\end{align}
where $\nu_{\min}$ is the minimal expansion and contraction rate of the $\tb{H}_p$-flow on normal directions on $\Gamma$ (see \eqref{eq: nu min definition} for the precise definition).
Then for $\alpha \in (0,1)$ and $\lambda \in (0,1)$ satisfying (\ref{eq: lambda condition}), and $v \in H_{\cu}^{-N,\mu}(M)$, such that $Pv=f$, suppose
$\WF^{s+1-\lambda\alpha}(v) \cap \Gamma^u = \emptyset$, 
$\WF_{\mathrm{cu}}^{s+1-\lambda\alpha,\mu}(v)\cap(\bar{\Gamma}^u\backslash\Gamma)=\emptyset$, and $\WF_{\mathrm{cu}}^{s-\alpha,\mu}(f)\cap\Gamma=\emptyset$, then $\WF_{\mathrm{cu}}^{s,\mu}(v)\cap \Gamma=\emptyset$.
\label{thm_propagation_M}
\end{thm}

\begin{rmk} \label{rmk: lambda explanation}
$\lambda$ is introduced when we include factors like $(x_1^2+\hat{\rho}^{2\alpha})^{-\lambda}$ in our commutant (see Section \ref{sec_step1}).
Since the absence of wavefront sets means microlocal regularity of corresponding order there, this theorem can be roughly interpreted as: 
suppose $u$ has $H^{s+1-\lambda\alpha,\mu}_{\cu}$-regularity at somewhere outside $\Gamma$, we can propagate this regularity into $\Gamma$, with a loss of $(1-\lambda\alpha)$ order, concluding $H^{s,\mu}_{\cu}$-regularity at $\Gamma$. 
In fact, we prove a sharper estimate on refined function spaces defined in Section \ref{sec: Sobolev spaces and operator classes, weighted}.
$\lambda\alpha$ can be thought of as the order on the front face, which corresponds to the differential order on the unstable manifold, that we can improve when we take advantage of our new pseudodifferential algebra. 
When $P-P^* \in \Psi_{\cu}^{m-2}$, we can make the loss $(1-\lambda\alpha)$ arbitrarily close to $0$. 
See Theorem \ref{thm: mirolocal estimate} and Remark \ref{rmk: lambda, loss explanation}.
\end{rmk}

\begin{rmk}  \label{rmk: p1 condition}
Our $\lambda$ satisfying \eqref{eq: lambda condition} exists only if $\sup_{\Gamma} {\bf p}_1 < \nu_{\min}$, which holds by \eqref{eq: p1 condition, in main thm}, which is stronger than that and needed in our positive commutator argument in Section \ref{sec_step2.3}. 
This is satisfied for scalar wave operator $\Box_g$ when we choose the density defining $L^2$-norm (this does not change the space, just changing the norm to an equivalent one) and the adjoint to be the one adapted to $g$, which makes $\Box_g$ symmetric and the subprincipal symbol vanishes. Thus this condition also holds for perturbations of $\Box_g$ that are small on the subprincipal level.

For tensor valued equations, which we are not treating directly here, condition \eqref{eq: p1 condition, in main thm} is verified for the linearized gauge-fixed Einstein equation on Schwarzschild-de Sitter in \cite[Section~9.1]{hintz2018global}, and for wave equations on tensors in \cite{hintz2017resonance}. Those computations applies to Schwarzschild spacetime as well. And by continuity, this also verifies the case of slowly rotating Kerr(-de Sitter) black holes.
\end{rmk}

\subsection{Previous works}
From the vast literature on the study of wave equations on Kerr and Kerr-de Sitter spacetimes, we only list a few examples.

The close connection between the normally hyperbolic trapping and Kerr black holes is observed in \cite{wunsch2011resolvent}. Then this property is extended to the range $\{|\mk{a}|<\mk{m}\}$ by Dyatlov \cite{dy15}. The Kerr-de Sitter case with small angular momentum is discussed by Vasy \cite{vasy2013microlocal}. 

Nonnenmacher and Zworski \cite{nonnenmacher2015decay} extended the results in \cite{wunsch2011resolvent} with weaker conditions. Before Hintz's estimates in \cite{hintz2021normally}, Dyatlov \cite{dyatlov2016spectral} obtained the width of the resonance free strip of the modified Laplacian under the same dynamical assumptions. Hintz's work can be viewed as a quantitative version of it. Vasy \cite{vasy2013microlocal} gave a systematic microlocal treatment of Kerr-de Sitter spacetimes with small angular momentum $\mk{a}$ and obtained an expansion of solutions to wave equations on Kerr-de Sitter spacetimes with terms corresponding to quasinormal modes. Recently, this is extended to the full subextremal range (see Section \ref{kds_intro} for its definition) by Petersen and Vasy \cite{petersen2024wave}.
For the Kerr case ($\Lambda=0$), it is known that results for small $|\mk{a}|$ apply to the full subextremal range. See \cite{whiting1989mode}\cite{dafermos2012black}\cite{shlapentokh2015quantitative} and references therein. 

Regarding the stability aspect, the stability of the Schwarzschild black holes was considered as early as in \cite{wald1979note}. 
The linear stability of slowly rotating Kerr black holes was considered in \cite{hafner2021linear}. The non-linear stability of various families of black holes in various regions was considered in \cite{zilhao2014testing}\cite{hintz2018global}\cite{dafermos2021non}\cite{klainerman2020global}\cite{klainerman2023kerr}\cite{shen2024kerr}\cite{giorgi2024kerrwave} and references therein. The mode stability was also investigated in \cite{casals2021hidden}\cite{shlapentokh2015quantitative}.

Local energy and decay estimates in the Schwarzschild case are proved in \cite{price1972}. For the Kerr case, see \cite{tataru2011local}\cite{tataru2013local}, in which the author proved $t^{-3}$ local uniform decay rate for linear waves. See also \cite{andersson2015hidden}\cite{dafermos2016decay} for decay estimates for wave equations, and sharp decay estimates are proven in \cite{angelopoulos2018late}\cite{hintz2022sharpprice}. The existence of solutions to semilinear equations with small initial conditions and an extra null condition was considered in \cite{luk2013null}. 

The method of blowing up the compactified cotangent bundle and the construction of our pseudodifferential algebra in this paper borrow the idea of the second microlocalization from \cite{jean1986second}\cite{vasy2021resolvent}\cite{vasysemiclassical}. The method of treating Schwartz kernels as paired Lagrangian distributions in \cite{de2015diffraction} also inspired us.

\subsection{Structure of this paper}
In Section \ref{sec_microlocal_analysis}, we recall some basic notions of microlocal analysis, and introduce our new symbol classes, operator classes and Sobolev spaces. Then we prove basic facts about them, including mapping properties, composition laws, elliptic estimates and G\r{a}rding's inequality. In Section \ref{setup}, we state assumptions needed for the proof of the main result and a microlocal quantitative version of Theorem \ref{thm_propagation_M}. Section \ref{sec_commutator} is our main analytic part, which is a multistage positive commutator argument. Finally, in Section \ref{application}, we apply the microlocal framework that we constructed to Kerr(-de Sitter) spacetimes.

\section{Microlocal Analysis}
\label{sec_microlocal_analysis}
In this section, we are going to develop basic facts about microlocal analysis for the symbol class we construct for our purpose. Throughout the rest of the paper, we assume that all functions on $M$ are supported in the region $\{\tau \leq 1\}$ (or equivalently, $\{t \geq 1\}$) and all functions on throughout our argument, including latter sections. And we don't make this support assumption on symbols and pseudodifferential operators so that proofs of properties of our new pseudodifferential algebra referring to proofs in the setting of the standard or cusp pseudodifferential algebras become more transparent. $S^m_{\mathrm{cu}}(M),\Psi_{\mathrm{cu}}^m(M)$ denote the cusp symbol class and pseudodifferential operator class, respectively. 

\subsection{The new symbol class}  \label{sec: new symbol class}
The spacetime we consider is $M^\circ=\R_t \times X$, where $X$ is an $n-1$ dimensional closed manifold; thus, $n=4$ in our application to Kerr(-de Sitter) spacetimes. 
We recall the basic setup of the cusp calculus in the introduction.
$M$ is obtained by compactifying $M^\circ$ in the $t$-direction 
\begin{align}
  M:= (M^\circ \sqcup ([0,\infty)_\tau \times X))/\sim,   \label{Mdef}
\end{align}
where $\sim$ is the identification: $(t,x) \sim (\tau=t^{-1},x),\,x\in X,t \in (0,\infty)$. Use $\mathcal{V}(M)$ to denote smooth vector fields on $M$. To facilitate our analysis, we introduce the cusp vector fields:
\begin{align}
\mathcal{V}_{\mathrm{cu}}(M):= \{ V\in \mathcal{V}(M): V\tau \in \tau^2 \mathcal{C}^{\infty}(M) \}.	
\end{align}
Suppose $x_1,x_2,...,x_{n-1}$ are local coordinates on $X$, then $\mathcal{V}_{\mathrm{cu}}(M)$ is locally spanned by
\begin{align}
\tau^2\partial_\tau,\partial_{x_1},\partial_{x_2}...,\partial_{x_{n-1}}.   \label{cuvec}
\end{align}
${}^{\mathrm{cu}}TM$ is the vector bundle with local frame (\ref{cuvec}). We point out that, as a cusp vector field, $\tau^2\partial_\tau$ is nonvanishing down to $\tau=0$. And its dual bundle, the cusp cotangent bundle, ${}^{\mathrm{cu}}T^*M$ is locally spanned by
\begin{align}
\frac{d\tau}{\tau^2},dx_1,dx_2,...,dx_{n-1}.  \label{cucovec}
\end{align}
We denote the radial compactification of ${}^{\mathrm{cu}}TM,{}^{\mathrm{cu}}T^*M$ by ${}^{\mathrm{cu}}\bar{T}M,{}^{\mathrm{cu}}\bar{T}^*M$ respectively, and
the defining function of fiber infinity by $\hat{\rho}$. Concretely, let $\bar{\R}^n$ denote the radial compactification of $\R^n$, and in coordinate patches that trivialize ${}^{\mathrm{cu}}T^*M$, we replace the fibers $\R^n$ by $\bar{\R}^n$. Concretely, suppose $\mathcal{U}$ is an open set in $M$ on which $T^*M$ is trivialized, then we define:
\begin{align*}
{}^{\mathrm{cu}}\bar{T}^*_{\mathcal{U}}M:=\mathcal{U} \times \bar{\R}^n.
\end{align*}
We further identify the quotient of ${}^{\mathrm{cu}}\bar{T}^*M$ under the fiber dilation with the cusp cosphere bundle ${}^{\mathrm{cu}}S^*M$. 

Next we consider a codimension one conic submanifold $Y$ of ${}^{\mathrm{cu}}\bar{T}^*M\setminus o$, locally defined by $\phi_Y$ which is homogeneous of degree 0 with respect to the fiber dilation and assume that its restriction to $S^*M^\circ$ is smooth. In addition we assume that 
\begin{align*}
\bar{Y}:=Y \cap{}^{\mathrm{cu}}\bar{T}_{\partial M}^*M
\end{align*}
is a conic submanifold of ${}^{\mathrm{cu}}\bar{T}_{\partial M}^*M$ with local defining function $\bar{\phi}_Y:=\phi_Y|_{\bar{T}_{\partial M}^*M}$ such that when identified as a function on ${}^{\mathrm{cu}}S^*_{\partial M}M$, $\bar{\phi}_Y \in \mathcal{C}^\infty({}^{\mathrm{cu}}S^*_{\partial M}M)$.
Then we set 
\begin{align*}
Y_0:=[0,\infty)_\tau \times \bar{Y}
\end{align*}
to be the stationary extension of $\bar{Y}$. And $\phi_{Y_0}$, the defining function of $Y_0$, is the stationary extension of $\bar{\phi}_Y$. 
In addition, we assume that
\begin{align} \label{eq: phi Y0 = x1}
\phi_{Y_0} = x_1,
\end{align}
which is the first `spatial' coordinate on $M$. This is not satisfied for the unstable manifold in Kerr(-de Sitter) spacetimes, which is what $Y_0$ is expected to model, but we will reduce to this model case through conjugating by a Fourier integral operator in Section \ref{application}.

Next we define our new symbol class $S_{\mathrm{cu},\alpha}^{m,\tilde{m}}(M,Y)$. In our setting, $M= (M^\circ \sqcup ([0,\infty)_\tau \times X))/\sim$ is the base manifold before blowing up. 
Let $Z$ be the manifold with corners obtained by blowing up boundary hypersurface of $Y_0$ given by the intersection of $Y_0$ and fiber infinity in ${}^{\rm cu}\bar T^*M$ in a manner that we describe below. 

For convenience, we define 
\begin{align} \label{eq: rho, tilde rho definition}
\rho:=\hat{\rho}^\alpha, \quad \tilde{\rho}:=(\phi_{Y_0}^2 + \rho^2)^{1/2} =(x_1^2+\hat{\rho}^{2\alpha})^{1/2},
\end{align}
where $\alpha$ is `the order of the blow up', and we will assume 
\begin{align} \label{eq: 0<alpha<1}
0 < \alpha < 1
\end{align}
throughout the rest of the paper. 
The first step of our blow up is to change the smooth structure of ${}^{\rm cu}\bar T^*M$ near fiber infinity as follows: suppose $(\mathsf{X},\hat{\rho},\hat{\xi})$ is a local coordinate system of ${}^{\rm cu}\bar T^*M$, with $\mathsf{X}$ being a coordinate system of $M$ and $\hat{\xi}$ being spherical coordinates of the fiber part, then our new smooth structure uses  $(\mathsf{X},\rho,\hat{\xi})$ as the local coordinate system over the same region. We denote the manifold with corners which topologically is identical with ${}^{\rm cu}\bar T^*M$ but equipped with this new smooth structure by ${}^{\cualpha}\bar T^*M$.
Then 
\begin{align}
Z:= [{}^{\cualpha}\bar T^*M;\{\rho=0,x_1=0\}]
\end{align}
is obtained by blowing up $\{\rho=0,x_1=0\}$ in ${}^{\cualpha}\bar T^*M$ homogeneously, which means replacing $\{\rho=0,x_1=0\}$ by its inward-pointing normal sphere bundle in ${}^{\cualpha}\bar T^*M$. 
More concretely, suppose
\begin{align}
\tau,x_1,...,x_{n-1}, \rho, \hat{\xi}_1,...,\hat{\xi}_{n-1},
\end{align}
are local coordinates on ${}^{\cualpha}\bar T^*M$ near $\{\rho=0,x_1=0\}$, then in the interior of the front face (i.e. the image of $\{\rho=0,x_1=0\}$ in $Z$), or equivalently near the intersection of the front face and the lift of $Y_0$, 
\begin{align}  \label{eq: Z coordinates near Y0}
\sigma = \frac{x_1}{\rho} , \rho, \tau, x_2,...,x_{n-1}, \hat{\xi}_1,...,\hat{\xi}_{n-1}, 
\end{align}
is a valid coordinate system on $Z$. Near the corner formed by the front face and the lift of fiber infinity,
\begin{align}  \label{eq: Z coordinates near fiber infinity}
\sigma' =  \frac{\rho}{x_1}, x_1,\tau, x_2,...,x_{n-1}, \hat{\xi}_1,...,\hat{\xi}_{n-1},
\end{align}
is a coordinate system on $Z$. In summary, $\tilde{\rho}$ is a defining function of the front face and $\rho$ is a defining function of the lift of fiber infinity away from the front face.

We refer readers to \cite[Chapter~5]{melrose1996MWC} for more details about blow up. Notice that the blow up we performed above is \emph{not} the quasihomogeneous blow up introduced by Melrose in \cite[Chapter~5]{melrose1996MWC}, which involves only homogeneity with integer indices. 


\begin{figure}[h] 	
	\begin{center} 
		\includegraphics[scale=1]{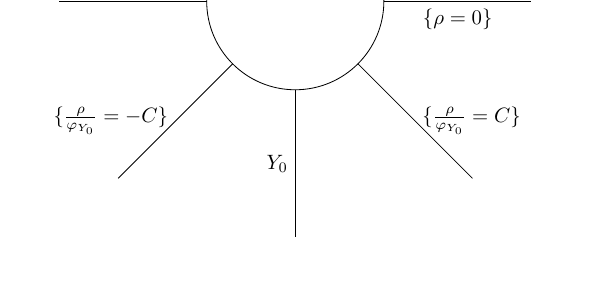}
	\end{center}
	\caption{The semicircle represents the front face. Two inclining lines are representing where $\frac{\rho}{\phi_{Y_0}}$ is a constant.}
	\label{fig:blow up, Y}
\end{figure}

\begin{defn}   \label{defn: smooth, new symbol class}
With $\tilde{\rho}$ in \eqref{eq: rho, tilde rho definition},
$S_{\mathrm{cu},\alpha}^{m,\tilde{m}}(M,Y)$ consists of functions on $Z$ satisfying:
\begin{align} \label{inv_1}
|W_1...W_ka| \leq C\hat{\rho}^{-m}\tilde{\rho}^{-\frac{\tilde{m}-m}{\alpha}},   
\end{align}
where $C$ may depends on $W_i$, which are are lifts of smooth vector fields on ${}^{\cualpha}\bar{T}^*M$ to $Z$ that are homogeneous of degree 0 with respect to both the fiber dilation near fiber infinity and tangent to the front face introduced by the blow up, and the lift of fiber infinity.  
\end{defn}
In the construction of $Z$ and the definition of $S_{\mathrm{cu},\alpha}^{m,\tilde{m}}(M,Y)$, we used $Y_0$ instead of $Y$. But $\bar{Y}$ and consequently $Y_0$ is uniquely determined by $Y$. Thus $S_{\mathrm{cu},\alpha}^{m,\tilde{m}}(M,Y)$ is well defined. In addition, for all $Y$  that settle down to the same $\bar{Y}$ in the manner described above as $\tau \rightarrow 0$, hence having the same stationary version $Y_0$, the class $S_{\mathrm{cu},\alpha}^{m,\tilde{m}}(M,Y)$ are the same. In particular, $S_{\mathrm{cu},\alpha}^{m,\tilde{m}}(M,Y) = S_{\mathrm{cu},\alpha}^{m,\tilde{m}}(M,Y_0)$. When $W_i$ are absent in \eqref{inv_1}, it imposes following pointwise bounds on $a$: 
on the region that is near the lift of fiber infinity in ${}^{\mathrm{cu}}\bar{T}^*M$ and away from the front face, we have $|a| \leq C \hat{\rho}^{-m}$; on the region that is near the interior of the front face, $a$ satisfies $|a| \leq C\hat\rho^{-\tilde m}$. Thus we say $m$ is the order (or index) associated with fiber infinity and $\tilde{m}$ is the order associated with the front face.

Next we give a description of $S_{\mathrm{cu},\alpha}^{m,\tilde{m}}(M,Y)$ in terms of local coordinates. 
\begin{lmm}
Suppose the defining functions of $Y_0$ and fiber infinity are $x_1$ and  $\hat{\rho}$, then $S_{\mathrm{cu},\alpha}^{m,\tilde{m}}(M,Y)$ consists of functions that are smooth on the interior of $Z$ which in terms of coordinates on $M$ satisfy:
\begin{align}
|\partial_\tau^{\mk{f}}\partial_{x_1}^\mk{l} \partial_{\hat{x}_1}^\beta   \partial_{\tilde{\xi}}^{\tilde{\gamma}} a(\tau,x_1,\hat{x}_1,\xi_{t},\xi)| \leq C \hat{\rho}^{-m+|{\tilde{\gamma}}|}\tilde{\rho}^{-\frac{\tilde{m}-m}{\alpha}-\mk{l}},
\label{diff_1}
\end{align}
where $C$ depends on $\mk{f},\mk{l},\beta,\tilde{\gamma}$. Here $\hat{x}_1$ means all $x_i$ other than $x_1$, $\tilde{\gamma}=(\gamma_\tau,\gamma_1,...,\gamma_{n-1})$, $\tilde{\xi}=(\xi_{t},\xi_1,...,\xi_{n-1})$ with $\xi_{t}$ being the variable dual to $\tau$ and $\xi_i$ being the variable dual to $x_i$, $1 \leq i \leq n-1$.
\label{lmm2}
\end{lmm}

\begin{proof}
Without loss of generality, we assume that $\xi_{n-1}$ is large relative to the other $\xi_i$ (including $\xi_t$), and we assume $\xi_{n-1}>0$ since the case where $\xi_{n-1}<0$ can be treated in the same manner.
Thus, we take $\hat{\rho}=\xi_{n-1}^{-1}$.
We still write $a \in S_{\mathrm{cu},\alpha}^{m,\tilde{m}}(M,Y)$ as a function of $\tau,x_i,\xi_{t},\xi_i$, but one should keep in mind that its the requirement on it is \eqref{inv_1}, i.e. the conormality with respect to the front face and fiber infinity on $Z$, which has coordinates \eqref{eq: Z coordinates near Y0} near the lift of $Y_0$ and \eqref{eq: Z coordinates near fiber infinity} near the lift of fiber infinity. 

First, we show that only the case $\mk{f}=0$ need to be considered. In (\ref{inv_1}), write each $W_i$ as
\begin{align}
W_i=c_{i\tau}\partial_\tau+\bar{W}_i, \label{wi}
\end{align} 
where $c_{i\tau} \in \mathcal{C}^{\infty}(Z)$ and $\bar{W}_i\tau=0$. Smoothness on the compactified space implies $c_{i\tau}$ and all of its derivatives are bounded. Consequently, we substitute (\ref{wi}) into the left hand side of (\ref{inv_1}) and expand. The final expression is a sum of terms of the form $\partial_\tau^{\mk{f}'}\bar{W}_{i1}\bar{W}_{i2}...a$ with a bounded function coefficient in front of it (possibly with different $\mk{f}'$ for different terms). Since the right hand sides of both (\ref{diff_1}) and (\ref{inv_1}) are independent of $\tau$, and all $\partial_\tau$ can be commuted to the front, we can consider the $\mk{f}=0$ case first and then add $\partial_\tau$ to (\ref{diff_1}) or each term of the expanded (\ref{inv_1}). In the proof below, we assume $\mk{f}=0$.

Consider the region $x_1 \leq C \rho=C \hat{\rho}^{\alpha}$, on which the local coordinates are:
\begin{align*}
\tau,\sigma:=\frac{x_1}{\xi_{n-1}^{-\alpha}},x_2,x_3,...,x_{n-1},\frac{\xi_{t}}{\xi_{n-1}},\frac{\xi_1}{\xi_{n-1}},\frac{\xi_2}{\xi_{n-1}},...,\rho=\hat{\rho}^\alpha=\xi_{n-1}^{-\alpha}.
\end{align*}
We first show that (\ref{diff_1}) implies (\ref{inv_1}).
We use the notation $\hat{\xi}_j:= \frac{\xi_j}{\xi_{n-1}}, 1\leq j \leq n-2$ or $j=t$. Recalling that $\hat{\rho}=\xi_{n-1}^{-1}$, we have 
\begin{align*}
\xi_{n-1} \partial_{\xi_{n-1}}=-\hat{\rho}\partial_{\hat{\rho}}.
\end{align*} 
Combining with $x_1=\rho\sigma, \hat{\rho}=\rho^{1/\alpha}$, we can choose vector fields in \eqref{inv_1} involving $\sigma$ and $\rho$ to be:
\begin{align}
&V_1:=\partial_\sigma =  \rho \partial_{x_1} = \xi_{n-1}^{-\alpha}\partial_{x_1}, \label{dsigma} \\
&V_2:=\rho \partial_\rho =
x_1\partial_{x_1} - \frac{1}{\alpha} (\xi_t\partial_{\xi_t}+\sum_{j=1}^{n-1} \xi_j \partial_{\xi_j} ).
\label{drho}
\end{align}
We first show that symbol class defined by (\ref{diff_1}) is preserved under application of $W_i$ defined after (\ref{inv_1}). We only need to verify this for $V_1$ and $V_2$ since being preserved under application of other $W_i$ is clear from (\ref{diff_1}). For $V_1$, we need to show that
\begin{align*}
|\partial_{x_1}^\mk{l} \partial_{\hat{x}_1}^\beta \partial_{\tilde{\xi}}^{\tilde{\gamma}} (\xi_{n-1}^{-\alpha}\partial_{x_1}a(\tau,x_1,\hat{x}_1,\xi_{t},\xi))| \leq C \hat{\rho}^{-m+|{\tilde{\gamma}}|}\tilde{\rho}^{-\frac{\tilde{m}-m}{\alpha}-\mk{l}}.
\end{align*}

Applying $\partial_{\xi_{n-1}}^{\gamma_{n-1}}$, it produces $\partial_{\xi_{n-1}}^{{\gamma}_{n-1}} (\xi_{n-1}^{-\alpha}\partial_{x_1}a(\tau,x_1,\hat{x}_1,\xi_{t},\xi))
= \sum_{k=0}^{{\gamma_{n-1}}} \binom{{\gamma_{n-1}}}{k} \partial_{\xi_{n-1}}^k (\xi_{n-1}^{-\alpha}) \partial_{\xi_{n-1}}^{{\gamma_{n-1}}-k}a.$ The $k$-th term is $C \xi_{n-1}^{-\alpha-k} \partial_{\xi_{n-1}}^{{\gamma_{n-1}}-k}a$. Applying the triangle inequality, we reduce the proof to the case of a single term, and notice that any power of $\xi_{n-1}$ will commute with $\partial_{x_i}$ in the front, so the inequality is equivalent to:
\begin{align*}
 |\hat{\rho}^{\alpha+k} \partial_{x_1}^{\mk{l}+1} \partial_{\hat{x}_1}^\beta \partial_{\tilde{\xi}}^{{\tilde{\gamma}}-(0,..,0,k)} a(\tau,x_1,\hat{x}_1,\xi_{t},\xi)| \leq C \hat{\rho}^{-m+|{\tilde{\gamma}}|}\tilde{\rho}^{-\frac{\tilde{m}-m}{\alpha}-\mk{l}}.
\end{align*}
Move $\hat{\rho}$ to the right hand side, and since $\hat{\rho}^\alpha$ is equivalent to $\tilde{\rho}$ on this patch, we have:
\begin{align}
 | \partial_{x_1}^{\mk{l}+1} \partial_{\hat{x}_1}^\beta \partial_{\tilde{\xi}}^{{\tilde{\gamma}}-(0,..,0,k)} a(\tau,x_1,\hat{x}_1,\xi_{t},\xi)| \leq C \hat{\rho}^{-m+(|{\tilde{\gamma}}|-k)} \tilde{\rho}^{-\frac{\tilde{m}-m}{\alpha}-(\mk{l}+1)},
\end{align}
which is (\ref{diff_1}), with ${\tilde{\gamma}},\mk{l}$ replaced by ${\tilde{\gamma}}-(0,..,0,k),\mk{l}+1$.

For $V_2$, we consider $x_1\partial_{x_1}$ and $\xi_{j}\partial_{\xi_{j}},j=t,1,...n-1$ individually. We prove:
\begin{align}\label{05}
|\partial_{x_1}^\mk{l} \partial_{\hat{x}}^\beta 
\partial_{\tilde{\xi}}^{\tilde{\gamma}} (x_1\partial_{x_1}a(\tau,x_1,\hat{x}_1,\xi_{t},\xi))| \leq C \hat{\rho}^{-m+|{\tilde{\gamma}}|}\tilde{\rho}^{-\frac{\tilde{m}-m}{\alpha}-\mk{l}},
\end{align}
and
\begin{align}
|\partial_{x_1}^\mk{l} \partial_{\hat{x}_1}^\beta \partial_{\tilde{\xi}}^{\tilde{\gamma}} (\xi_{j} \partial_{\xi_{j}}a(\tau,x_1,\hat{x}_1,\xi_{t},\xi))| \leq C \hat{\rho}^{-m+|{\tilde{\gamma}}|}\tilde{\rho}^{-\frac{\tilde{m}-m}{\alpha}-\mk{l}}. \label{06}
\end{align}
The proof is the same as $V_1$ case, except for that we use instead  
\begin{align*}
x_1 \leq C\xi_{n-1}^{-\alpha}, |\xi_j| \leq \hat{\rho}^{-1}
\end{align*}
to bound the extra $x_1,\xi_j$ factors in \eqref{05} and \eqref{06}. 
Thus we know that the function class defined by (\ref{diff_1}) is preserved under the application of $W_i$, and in order to show (\ref{inv_1}) holds, we only need to consider the case with a single vector. When this vector is one of $\partial_{x_i},\xi_k\partial_{\xi_j}$  $i \in \{2,...,n-1\}, \;  k,j \in \{1,2,...,n-1,t\}$, the inequality is straightforward from (\ref{diff_1}) with one of $|{\tilde{\gamma}}|$ or $|\beta|$ equal to one, and $\mk{l}$ and the other being 0. When this vector is $V_1$, the bound (\ref{inv_1}) follows from the case $\mk{l}=1$ and $\beta={\tilde{\gamma}}=0$. When this vector is $V_2$, we show the bound for two terms respectively. For the first term, we use $x_1 \leq C \rho$ and apply the result proved for $V_1$. For the second term, the bound follows from the case $\mk{l}=\beta=0$, ${\tilde{\gamma}}=(0,0,..,0,1)$. Combining all cases we conclude that (\ref{inv_1}) holds.

Conversely, we take (\ref{inv_1}) as the assumption to prove (\ref{diff_1}). We apply induction on $\mk{l}+|\beta|+|{\tilde{\gamma}}|$. 
The case when $\mk{l}=|\beta|=|{\tilde{\gamma}}|=0$ is straightforward.
We consider the case $\mk{l},|\beta|$ or ${\tilde{\gamma}}_i$ increases by 1 respectively. 

When ${\tilde{\gamma}}_i$ increases, we multiply $\xi_{n-1}$ on both sides of (\ref{diff_1}), and notice that:
\begin{align*}
&\xi_{n-1} \partial^{{\gamma_{n-1}}+1}_{\xi_{n-1}}a = \partial_{\xi_{n-1}}^{{\gamma_{n-1}}}(\xi_{n-1} \partial_{\xi_{n-1}} a)-\partial_{\xi_{n-1}}^{{\gamma_{n-1}}}a,\\
&\xi_{n-1} \partial^{{\tilde{\gamma}}_i+1}_{\xi_i}a = \partial_{\xi_i}^{{\tilde{\gamma}}_i}(\xi_{n-1}\partial_{\xi_i}a) .
\end{align*}
The symbol class defined by (\ref{inv_1}) is preserved under $\xi_{n-1} \partial_{\xi_i}$, so these two equations complete the induction step for the case in which ${\tilde{\gamma}}_i$ increases (other $x$-derivatives are not written here since they commute with everything here and do not affect the result. When $\mk{l}$ increases, we rewrite $\partial_{x_1}^{\mk{l}+1}= \rho^{-1} \rho\partial_{x_1}^{\mk{l}+1}$, and factor out $\rho^{-1} \approx \tilde{\rho}^{-1}$ ($\approx$ means quantities on each side can bound each other up to a constant factor) and apply the induction hypothesis to the rest part. Precisely, if ${\gamma_{n-1}}=0$, then $\rho$ commutes with all differential operators appear: $\rho\partial^{\mk{l}+1}_{x_1} \partial_{\hat{x}_1}^\beta \partial_{\tilde{\xi}}^{{\tilde{\gamma}}} a(\tau,x_1,\hat{x}_1,\xi_{t},\xi) = \partial^{\mk{l}}_{x_1} \partial_{\hat{x}_1}^\beta \partial_{\tilde{\xi}}^{{\tilde{\gamma}}}  (\rho\partial_{x_1}a(x_1,\hat{x}_1,\xi) )$. And then use the fact that $\rho \partial_{x_1} a$ is in the same symbol class since $\rho\partial_{x_1}$ can be taken as one of $W_i$ in \eqref{inv_1}. For ${\gamma_{n-1}} \geq 1$, we have
\begin{align*}
\rho \partial_{x_1}^{\mk{l}+1} \partial_{\hat{x}_1}^\beta  \partial_{\xi_{n-1}}^{{\gamma_{n-1}}}a &=  \partial_{x_1}^{\mk{l}} \partial_{\hat{x}_1}^\beta \rho \partial_{\xi_{n-1}}^{{\gamma_{n-1}}} \partial_{x_1} a  \\
= & \partial_{x_1}^{\mk{l}} \partial_{\hat{x}_1}^\beta \xi_{n-1}^{-\alpha}   \partial_{\xi_{n-1}}^{{\gamma_{n-1}}} \partial_{x_1} a\\
= & \partial_{x_1}^{\mk{l}} \partial_{\hat{x}_1}^\beta (\xi_{n-1}^{-\alpha}\partial_{\xi_{n-1}}^{{\gamma_{n-1}}}\xi_{n-1}^{\alpha})\xi_{n-1}^{-\alpha}\partial_{x_1} a
\\= &  \partial_{x_1}^{\mk{l}} \partial_{\hat{x}_1}^\beta (\partial_{\xi_{n-1}}+\alpha\xi_{n-1}^{-1})^{{\gamma_{n-1}}} \xi_{n-1}^{-\alpha}\partial_{x_1} a.
\end{align*}
Since $S_{\mathrm{cu},\alpha}^{m,\tilde{m}}(M,Y)$ is preserved under $\xi_{n-1}^{-\alpha}\partial_{x_1}$, using the induction hypothesis, each term after expanding $(\partial_{\xi_{n-1}}+\alpha\xi_{n-1}^{-1})^{{\gamma_{n-1}}}$ satisfies (\ref{diff_1}), which completes the induction step.
When $|\beta|$ increases, we simply notice that $\partial_{x_j}, 2\leq j \leq n-1$ are tangent to the frontface and the lift of fiber infinity in $Z$, hence the result is straightforward by (\ref{inv_1}). Combining all cases above, we have finished the proof in the patch $x_1 \leq C \rho$.

Next we consider the other patch $\rho \leq C x_1$, over which the coordinate system is
\begin{align*}
\tau,\sigma':=\frac{\xi_{n-1}^{-\alpha}}{x_1},x_2,x_3,...,x_{n-1},\frac{\xi_{t}}{\xi_{n-1}},\frac{\xi_1}{\xi_{n-1}},\frac{\xi_2}{\xi_{n-1}},...,\rho:=\hat{\rho}^\alpha.
\end{align*}
Thus, we can choose vector fields in \eqref{inv_1} involving $\sigma',\rho$ to be
\begin{align*}
& K_1 = \sigma'\partial_{\sigma'}  = - x_1 \partial_{x_1},\\
& K_2 = \rho\partial_{\rho} = x_1\partial_{x_1} - \frac{1}{\alpha} (\xi_t\partial_{\xi_t}+\sum_{j=1}^{n-1} \xi_j \partial_{\xi_j} ).
\end{align*}
The argument is similar to that in the first coordinate patch. 
We first show that (\ref{diff_1}) implies (\ref{inv_1}).
We verify that the symbol class defined by (\ref{diff_1}) is preserved under $K_1,K_2$. For $K_1$, the proof is the same as $V_1$ in the first patch, just notice that now $Cx_1 \geq \rho$, so $\tilde{\rho}$ is equivalent to $x_1$. For example: $|x_1\partial_{x_1}a |\leq \hat{\rho}^{-m}\tilde{\rho}^{-\frac{\tilde{m}-m}{\alpha}}$ is equivalent to $|\partial_{x_1}a |\leq \hat{\rho}^{-m}\tilde{\rho}^{-\frac{\tilde{m}-m}{\alpha}-1}$. For $K_2$, the proof is similar to $V_2$ case in the first patch, all terms are treated similarly. The difference to the proof in the first coordinate patch is that, we use $x_1 \leq \tilde{\rho}$ to bound $x_1$ in the front of the analogue of \eqref{05}.

Now we show that (\ref{inv_1}) implies (\ref{diff_1}) in this region. In order to verify (\ref{diff_1}), we again use induction, the case in which $|\beta|$ or $|{\tilde{\gamma}}|$ increase is the same as before. When $\mk{l}$ increases, we use $\partial_{x_1} = (x_1)^{-1}(x_1 \partial_{x_1})$, since now $(x_1)^{-1} \approx \tilde{\rho}^{-1}$, and then use:
\begin{align*}
x_1\partial_{x_1}^{\mk{l}+1} \partial_{\hat{x}_1}^\beta  \partial_{\tilde{\xi}}^{{\tilde{\gamma}}}a = \partial_{x_1}^{\mk{l}} \partial_{\hat{x}_1}^\beta  \partial_{\tilde{\xi}}^{{\tilde{\gamma}}} (x_1\partial_{x_1}a) - \partial_{x_1}^{\mk{l}} \partial_{\hat{x}_1}^\beta  \partial_{\tilde{\xi}}^{{\tilde{\gamma}}}a,
\end{align*}
when $\mk{l} \geq 1$, and
\begin{align*}
x_1\partial_{x_1}\partial_{\hat{x}_1}^\beta  \partial_{\tilde{\xi}}^{{\tilde{\gamma}}}a = \partial_{\hat{x}_1}^\beta  \partial_{\tilde{\xi}}^{{\tilde{\gamma}}} (x_1\partial_{x_1}a),
\end{align*}
when $\mk{l}=0$.
Since $x_1\partial_{x_1}$ satisfies conditions on $W_i$ in \eqref{inv_1}, $x_1\partial_{x_1}a$ is in the same symbol class, hence we can apply the induction hypothesis to the right hand side and use $x_1 \approx \tilde{\rho}$ to obtain:
$$|\partial_{x_1}^{\mk{l}+1} \partial_{\hat{x}_1}^\beta  \partial_{\tilde{\xi}}^{{\tilde{\gamma}}}a| \leq C \hat{\rho}^{-m+|{\tilde{\gamma}}|}\tilde{\rho}^{-\frac{\tilde{m}-m}{\alpha}-\mk{l}-1},$$
which completes the proof.
\end{proof}

\begin{rmk}  \label{remark: alternative-proof, symbol, local}
An alternative and conceptually more transparent way to prove the lemma above is to show that the space of vector fields in Definition \ref{defn: smooth, new symbol class} is spanned by 
\begin{align}
\partial_\tau, \tilde{\rho}\partial_{x_1}, \partial_{x_i}, \hat{\rho}^{-1}\partial_{\xi_j}, \hat{\rho}^{-1}\partial_{\xi_t}  \; 2 \leq i \leq n-1, 1 \leq j \leq n-1,
\end{align}
near the front face, and use this to show that the characterization of the symbol class above is equivalent to \eqref{inv_1} in local coordinates.
\end{rmk}

Recall that the cusp symbol class $S_{\mathrm{cu}}^m(M)$ consists of conormal functions on ${}^{\mathrm{cu}}\bar T^*M$ satisfying:
\begin{align} \label{eq: smooth cusp symbol}
|\partial_\tau^{\mk{f}}\partial_{x_1}^\mk{l} \partial_{\hat{x}_1}^\beta   \partial_{\tilde{\xi}}^{\tilde{\gamma}} a(\tau,x_1,\hat{x}_1,\xi_{t},\xi)| \leq C \hat{\rho}^{-m+|{\tilde{\gamma}}|}.          
\end{align}
Since functions on ${}^{\mathrm{cu}}\bar T^*M$ that are conormal to fiber infinity are the same as functions on ${}^{\cualpha}\bar{T}^*M$ that are conormal to fiber infinity and they lift to functions on $Z$ that are conormal to the front face and fiber infinity. Comparing with (\ref{diff_1}), we have the inclusion:
\begin{coro}  \label{coro: symbol inclusion}
For $m,\tilde{m}\in \R$, $\tilde{m} \geq m$, $M,Y$ as above, we have the following relationship between symbol classes:
\begin{align}
    S_{\mathrm{cu}}^m(M) \subset S_{\mathrm{cu},\alpha}^{m,\tilde{m}}(M,Y).     \label{inclusion}
\end{align}
\end{coro}

Next we discuss the quantization procedure. Let $\dot{\mathcal{C}}^\infty(M)$ be the class of smooth functions on $M$ which vanish to infinite order at $\partial M$. For $u \in \dot{\mathcal{C}}^\infty(M)$ supported in a coordinate chart near $\{\tau=0\}$, the action of $\mathrm{Op}(a)$, the left quantization (in short, we use `quantization' below) of $a \in S_{\mathrm{cu},\alpha}^{m,\tilde{m}}(M,Y)$ (which is assumed to be supported in the same coordinate chart) is defined by:
\begin{align}
\mathrm{Op}(a)u(t,x)  := (2\pi)^{-n} \int   e^{i( (t-t')\xi_t+(x-x')\cdot\xi )}a(t^{-1},x,\xi_t,\xi)u(t',x')dt'dx'd\xi_t d\xi,               \label{eq: cusp quantization}
\end{align}
where $\xi_t$ is the variable dual to $t$. And for general $u$ and $a$, the action of $\mathrm{Op}(a)$ is defined by decomposing $u$ and $a$ using partitions of unity.
Since we use $\frac{d\tau}{\tau^2}$ in the frame of ${}^{\mathrm{cu}}\bar{T}^*M$, which is also a frame of ${}^{\cualpha}\bar{T}^*M$ and
\begin{align}
\xi_t dt + \sum_{i=1}^{n-1}\xi_idx_i = -\xi_t\frac{d\tau}{\tau^2}+\sum_{i=1}^{n-1}\xi_idx_i,
\end{align}
 $\xi_t$ is also dual to $\tau$ up to a sign. 
 We use $\Psi_{\mathrm{cu},\alpha}^{m,\tilde{m}}(M,Y)$ to denote the operator class obtained by quantizing symbols in $S_{\mathrm{cu},\alpha}^{m,\tilde{m}}(M,Y)$. The symbol class defined by (\ref{inv_1}) is globally defined, hence by the equivalence shown by Lemma \ref{lmm2}, (\ref{diff_1}) also defines a symbol class on the manifold $Z$. Next we verify that this class is closed under composition:
\begin{prop}  \label{prop: composition}
Under assumptions above: $\alpha \in (0,1)$, $Y_0$ is defined by $x_1=0$, for $A \in \Psi_{\mathrm{cu},\alpha}^{m_1,\tilde{m}_1}(M,Y), B \in \Psi_{\mathrm{cu},\alpha}^{m_2,\tilde{m}_2}(M,Y)$ with symbols $a \in S_{\mathrm{cu},\alpha}^{m_1,\tilde{m}_1}(M,Y), b \in S_{\mathrm{cu},\alpha}^{m_2,\tilde{m}_2}(M,Y)$, we have
\begin{align}
A \circ B \in \Psi_{\mathrm{cu},\alpha}^{m_1+m_2,\tilde{m}_1+\tilde{m}_2}(M,Y), 
\end{align}
and its symbol $a \circ b \in S_{\mathrm{cu},\alpha}^{m_1+m_2,\tilde{m}_1+\tilde{m}_2}(M,Y)$ has asymptotic expansion
\begin{align} \label{sym_composition}
a \circ b(z,\zeta) = \sum_{l \in \N^n} \frac{(-i)^{|l|}}{l !}\partial^l_\zeta a \partial^l_z b,  
\end{align}
where $z=(t,x_1,\ldots,x_{n-1})$ and correspondingly $\zeta=(\xi_t,\xi_1,\ldots,\xi_{n-1})$. In addition, the term with $l = (l_t,l_1,\ldots,l_{n-1}) \in \N^n$ as derivative index belongs to $S_{\mathrm{cu},\alpha}^{m_1+m_2-|l|,\tilde{m}_1+\tilde{m}_2-|l|(1-\alpha)}(M,Y)$.  \label{prop_comp}
\end{prop}
\begin{proof}
Using the same method as in the cusp calculus, or more concretely the left reduction (see \cite[Proposition~5.1]{vasy2018minicourse} in the setting of the uniform symbol class), we know that the left symbol of $A \circ B$ is given by (\ref{sym_composition}). 
Notice that, the difference with the proof in \cite{vasy2018minicourse} is that, the symbol class there incur a $\delta$-order loss whenever the symbol is differentiated in both position and momentum variables, and this necessitate $\delta \in (0,\frac{1}{2})$ for further terms in the expansion to have lower order. However, our symbol class only incurs a loss when we differentiate in $x_1$, not in any momentum variables, thus the $(1-2\delta)|l|$ improvement in order is now $(1-\alpha)|l|$, which we will explain in detail below. 
In addition, \eqref{sym_composition} is not valid when we are not using a coordinate system such that $Y_0$ is defined by one of $x$-components
, see Remark \ref{rmk: reversed inclusion}.

Assuming that both of $a$ and $b$ satisfy (\ref{diff_1}), we verify that (\ref{sym_composition}) still satisfies (\ref{diff_1}). 
The only source of potential growth (as $\tilde{\rho} \rightarrow 0$) comes from $\partial_{x_1}^{l_1}b$, which gives $\tilde{\rho}^{-l_1}$ growth (compared with $\hat{\rho}^{-m}\tilde{\rho}^{-\frac{\tilde{m}-m}{\alpha}}$). On the other hand, this term is multiplied by $\partial_{\xi_1}^{l_1}a$, which gives $\hat{\rho}^{l_1}$ decay. Concretely, the $l$-th term has orders reduced by at least $|l|,|l|(1-\alpha)$ in the first and second indices respectively compared with the typical bound for $S_{\mathrm{cu},\alpha}^{m_1+m_2,\tilde{m}_1+\tilde{m}_2}(M,Y)$, i.e. $\hat{\rho}^{-(m_1+m_2)}\tilde{\rho}^{-\frac{\tilde{m}_1+\tilde{m}_2-m_1-m_2}{\alpha}}$. 

In order to verify the symbolic property of $a \circ b$, we can apply Leibniz's rule. Each partial derivative will fall on exactly one of $\partial_\xi^la$ and $\partial_x^lb$. The bound on the right hand side of (\ref{diff_1}) with $(m,r)$ being either one of $(m_1-|l|,\tilde{m}_1)$ and $(m_2,\tilde{m}_2+\alpha l_1)$ changes in the same manner under each differentiation. Notice that $l_1\leq |l|$, hence their product satisfies (\ref{diff_1}), but with $(m_1+m_2-|l|,\tilde{m}_1+\tilde{m}_2-|l|(1-\alpha))$ as the symbol class order. 
\end{proof}

\begin{rmk} \label{rmk: reversed inclusion}
When $\tilde{m} \geq m$, a weak version of the converse of Corollary \ref{coro: symbol inclusion} is true if one uses the cusp version of H\"ormander's $(\varrho,\delta)$-class of order $\tilde{m}$ with $\varrho=1,\delta=\alpha$, which we denote by $S^{\tilde{m}}_{\rm cu; 1,\alpha}(M)$. Concretely, this is the space of functions on ${}^{\mathrm{cu}}\bar T^*M$ satisfying:
\begin{align} \label{eq: 1,alpha symbol}
|\partial_\tau^{\mk{f}} \partial_{x_1}^\mk{l} \partial_{\hat{x}_1}^\beta   \partial_{\tilde{\xi}}^{\tilde{\gamma}} a(\tau,x_1,\hat{x}_1,\xi_{t},\xi)| \leq C \hat{\rho}^{-\tilde{m}+|{\tilde{\gamma}}|-\alpha(\mk{f}+\mk{l}+|\beta|)},         
\end{align}
where $C$ depends on $\mk{f},\mk{l},\beta,\tilde{\gamma}$. Comparing this requirement and \eqref{diff_1} and notice that $\hat{\rho}^{-1} \geq \tilde{\rho}^{-\frac{1}{\alpha}}$, we have
\begin{align}
S_{\mathrm{cu},\alpha}^{m,\tilde{m}}(M,Y) \subset S^{\tilde{m}}_{\rm cu; 1,\alpha}(M).
\end{align}
And the construction for the pseudodifferential algebra can be carried out in the same manner as the one on $\mathbb{R}^n$ with $(1,\alpha)-$type symbols. However, in the expansion \eqref{sym_composition}, the term with $|l|=2$ is only $2(1-\alpha)-$order lower than the leading term in general. Consequently, in order to run a positive commutator argument, which requires the that term to be negligible compared with the sub-principal symbol in the cusp pseudodifferential algebra, we need $2(1-\alpha)>1$, which is equivalent to $\alpha<\frac{1}{2}$. For $\alpha \in [\frac{1}{2},1)$, we need to exploit the second microlocal nature (concretely, the normally hyperbolicity near the trapped set) as in Section \ref{sec: improve orders of operators} to justify that terms with $|l| \geq 2$ in \eqref{sym_composition} is harmless in the commutator argument.

In addition, this symbol class $S^{\tilde{m}}_{\rm cu; 1,\alpha}(M)$ is not invariant under pull-backs of symplectic lifts of coordinate changes. When $\alpha \in [\frac{1}{2},1)$, terms in \eqref{sym_composition} with larger $|l|$ might not have lower order if we make a change of variables since the role of $\partial_x,\partial_{\xi}$ will be mixed after that change. Thus, the left reduction above should be understood as in a single coordinate patch. 
In fact, the characterization of $S_{\mathrm{cu},\alpha}^{m,\tilde{m}}(M,Y)$ in Lemma \ref{lmm2} is not invariant under such pull-backs as well (otherwise we would have run the whole argument in Kerr(-de Sitter) spacetimes directly). However, it is preserved under (symplectic lifts of) coordinate changes preserving \eqref{eq: phi Y0 = x1}. That is to say, the characterization in Lemma \ref{lmm2} does not depend on the specific choice of the defining function of $Y_0$ as long as the coordinate system takes it as $x_1$. One can see this from the proof of Lemma \ref{lmm2}, but this can be seen in local charts as well. For example, we verify that the class defined by \eqref{diff_1} in the new coordinates $(x',\xi')$ is contained in the class defined by \eqref{diff_1} in terms of $(x,\xi)$, and the reversed inclusion holds by symmetry.
For such a symplectomorphism, we have
\begin{align}
x_1 = \tilde{c}(x') x_1',
\end{align}
where $\tilde{c},\frac{1}{\tilde{c}}$ have uniformly bounded derivatives. 
Consequently, $\partial_{x_1}$ appearing when we express $\partial_{x'}$ in terms of the old coordiantes are accompanied by a $x_1$-factor, hence not incurring loss  as in the proof of Lemma \ref{lmm2}, except for when we are expressing $\partial_{x_1'}$, which will have a term $\tilde{c}(x')\partial_{x_1}$, which incurs $\alpha$-order loss but that is allowed for $\partial_{x_1'}$.
\end{rmk}

\subsection{Sobolev spaces and operator classes}  \label{sec: Sobolev spaces and operator classes, weighted}
Recall \eqref{cucovec}, the cusp cotangent bundle is locally spanned by $\frac{d\tau}{\tau^2},dx_1,dx_2,$ $...,dx_{n-1}$, whose wedge product gives the cusp density $\nu_{\mathrm{cu}}$. $L_{\mathrm{cu}}^2(M)$ consists of functions that are supported on $\{\tau \leq 1\}$ and square integrable with respect to density $\nu_{\mathrm{cu}}$ equipped with the norm 
\begin{align*}
||u||_{L_{\mathrm{cu}}^2(M)}:=(\int_M |u|^2d\nu_{\mathrm{cu}})^{\frac{1}{2}}. 
\end{align*}
Recall that for $r\in \R$, the $r$-weighted cusp $L^2$-space is defined by
\begin{align*}
L_{\mathrm{cu}}^{2,r}(M):=\{  u \in L_{\mathrm{loc}}^2(M): \tau^{-r}u \in L_{\mathrm{cu}}^2(M) \},
\end{align*}
where $L_{\mathrm{loc}}^2(M)$ is the space of locally $L^2$-integrable functions on $M$. We use the notation $\la \cdot \ra = (1+|\cdot|^2)^{\frac{1}{2}}$. 
Define 
\begin{align*}
&D:=\Op(\hat{\rho}^{-1}),\, \la D \ra:=\Op(\la \hat{\rho}^{-1} \ra),\\
&\tilde{D}:=\Op(\tilde{\rho}^{-1}),\, \la \tilde{D} \ra:=\Op(\la \tilde{\rho}^{-1} \ra),\\
& D_{m,\tilde{m}}:= \Op(\la \hat{\rho}^{-1} \ra^m \la \tilde{\rho}^{-1} \ra^{\frac{\tilde{m}-m}{\alpha}}).
\end{align*}
For $s,\tilde{s} \geq 0$, the weighted Sobolev space can be characterized as
\begin{align}
H_{\mathrm{cu},\alpha}^{s,\tilde{s},r}(M,Y)=\{ u \in L_{\mathrm{cu}}^{2,r}(M):  D_{s,\tilde{s}}u\in L_{\mathrm{cu}}^{2,r}(M)\}. \label{defn_wsob2}
\end{align}
For $s,\tilde{s} < 0$, $H_{\mathrm{cu},\alpha}^{s,\tilde{s},r}(M,Y)$ is defined to be the dual space of $H_{\mathrm{cu},\alpha}^{-s,-\tilde{s},r}(M,Y)$, which has been defined above. For $s>0,\tilde{s}<0$, notice that $s=\frac{s}{s+|\tilde{s}|}(s+|\tilde{s}|)+\frac{|\tilde{s}|}{s+|\tilde{s}|}\times 0 , \, \tilde{s} = \frac{s}{s+|\tilde{s}|} \times 0 + \frac{|\tilde{s}|}{s+|\tilde{s}|}(-s-|\tilde{s}|) $, and so $H_{\mathrm{cu},\alpha}^{s,\tilde{s},r}(M,Y)$ can be defined by interpolating $H_{\mathrm{cu},\alpha}^{s+|\tilde{s}|,0,r}(M,Y)$ and $H_{\mathrm{cu},\alpha}^{0,-s-|\tilde{s}|,r}(M,Y)$. Similarly for the case $s<0,\tilde{s}>0$.
We equip these Sobolev spaces with norm 
\begin{align}
||u||_{s,\tilde{s},r}:=||\tau^{-r}D_{s,\tilde{s}}u||_{L_{\mathrm{cu}}^{2}(M)}. \label{wsob_norm}
\end{align}

The $\tau$-weighted versions of $S_{\mathrm{cu},\alpha}^{m,\tilde{m}}(M,Y),\Psi_{\mathrm{cu},\alpha}^{m,\tilde{m}}(M,Y)$ are defined by
\begin{align} \label{eq: weighted smooth new symbol, PsiDO}
\begin{split}
S_{\mathrm{cu},\alpha}^{m,\tilde{m},r}(M,Y):=\tau^{-r}S_{\mathrm{cu},\alpha}^{m,\tilde{m}}(M,Y)=\{ \tau^{-r}a: a \in  S_{\mathrm{cu},\alpha}^{m,\tilde{m}}(M,Y)\},\\
\Psi_{\mathrm{cu},\alpha}^{m,\tilde{m},r}(M,Y):=\tau^{-r}\Psi_{\mathrm{cu},\alpha}^{m,\tilde{m}}(M,Y)=\{\tau^{-r}A: A \in \Psi_{\mathrm{cu},\alpha}^{m,\tilde{m}}(M,Y) \}.
\end{split}
\end{align}
The weighted generalization of Proposition \ref{prop_comp} holds: 
\begin{prop}  \label{prop: composition weighted}
If $A \in \Psi_{\mathrm{cu},\alpha}^{m_1,\tilde{m_1},r_1}(M,Y), \, B \in \Psi_{\mathrm{cu},\alpha}^{m_2,\tilde{m_2},r_2}(M,Y)$, then $A \circ B \in \Psi_{\mathrm{cu},\alpha}^{m_1+m_2,\tilde{m_1}+\tilde{m_2},r_1+r_2}(M,Y)$.  In addition, we still have the asymptotic expansion of the full symbol as in \eqref{sym_composition} and now the term with $l = (l_t,l_1,\ldots,l_{n-1}) \in \N^n$ as derivative index belongs to $S_{\mathrm{cu},\alpha}^{m_1+m_2-|l|,\tilde{m}_1+\tilde{m}_2-|l|(1-\alpha),r_1+r_2}(M,Y)$. 
\end{prop}
This can be proved by writing $A=\tau^{-r_1}\tilde{A},B=\tilde{B}\tau^{-r_1}$ with $\tilde{A} \in \Psi_{\mathrm{cu},\alpha}^{m_1,\tilde{m_1}}(M,Y), \, \tilde{B} \in \Psi_{\mathrm{cu},\alpha}^{m_2,\tilde{m_2}}(M,Y)$ and then apply Proposition \ref{prop_comp}.

\subsection{The principal symbol}
\label{sec_principal_symbol}
\begin{defn}
The \textit{principal symbol} of $A = \mathrm{Op}(a) \in \Psi_{\mathrm{cu},\alpha}^{m,\tilde{m},r}(M,Y)$, denoted by $\sigma_{m,\tilde{m},r}(A)$ is the equivalent class $[a]$ of $a$ in $S_{\mathrm{cu},\alpha}^{m,\tilde{m},r}/S_{\mathrm{cu},\alpha}^{m-1,\tilde{m}-(1-\alpha),r}$.  
\end{defn}
In later arguments, we also call $a$ the principal symbol of $A$ if $a$ is a representative of $A$'s principal symbol. We emphasize that the principal part is defined modulo symbols which are only $(1-\alpha)$ orders weaker in the second index, and the decay order remains the same. An important result we need is the following proposition about the principal symbol of commutators.
\begin{prop}
If $A \in \Psi_{\mathrm{cu},\alpha}^{m,\tilde{m},r}(M,Y),\, B \in \Psi_{\mathrm{cu},\alpha}^{m',\tilde{m}',r'}(M,Y)$, then $[A,B] \in \Psi_{\mathrm{cu},\alpha}^{m+m'-1,\tilde{m}+\tilde{m}'-(1-\alpha),r+r'}(M,Y)$, with principal symbol
\begin{align}
\sigma_{m+m'-1,\tilde{m}+\tilde{m}'-(1-\alpha),r+r'}([A,B]) = -iH_ab, \label{commutator}
\end{align}
where $[a] = \sigma_{m,\tilde{m},r}(a),[b]=\sigma_{m',\tilde{m}',r'}(B)$.
\end{prop}
\begin{proof}
This follows from the symbolic expansion given in the proof of Proposition \ref{prop_comp}. The principal symbols of $AB$ and $BA$ coincide, and going one order further gives (\ref{commutator}).
\end{proof}
\subsection{Wavefront Sets and Ellipticity}
Symbol classes with infinite indices are defined by:
\begin{align*}
&S_{\mathrm{cu},\alpha}^{\infty,\infty,r}(M,Y):=\bigcup_{l_1,l_2\in \Z} S_{\mathrm{cu},\alpha}^{l_1,l_2,r}(M,Y),\\
&S_{\mathrm{cu},\alpha}^{\infty,l_2,r}(M,Y):=\bigcup_{l_1 \in \Z} S_{\mathrm{cu},\alpha}^{l_1,l_2,r}(M,Y),\\
&S_{\mathrm{cu},\alpha}^{l_1,\infty,r}(M,Y):=\bigcup_{l_2 \in \Z} S_{\mathrm{cu},\alpha}^{l_1,l_2,r}(M,Y).
\end{align*} 
And replace union by intersection when we replace $\infty$ by $-\infty$: 
\begin{align*}
&S_{\mathrm{cu},\alpha}^{-\infty,-\infty,r}(M,Y):=\bigcap_{l_1,l_2\in \Z}S_{\mathrm{cu},\alpha}^{l_1,l_2,r}(M,Y),\\
&S_{\mathrm{cu},\alpha}^{-\infty,l_2,r}(M,Y):=\bigcap_{l_1 \in \Z}S_{\mathrm{cu},\alpha}^{l_1,l_2,r}(M,Y),\\
&S_{\mathrm{cu},\alpha}^{l_1,-\infty,r}(M,Y):=\bigcap_{l_2 \in \Z} S_{\mathrm{cu},\alpha}^{l_1,l_2,r}(M,Y).
\end{align*}
Similar notations apply to operator classes with $S_{\mathrm{cu},\alpha}^{\cdot,\cdot,\cdot}(M,Y)$ replaced by $\Psi_{\mathrm{cu},\alpha}^{\cdot,\cdot,\cdot}(M,Y)$, and also for Sobolev spaces with $S$ replaced by $H$
and $\pm \infty$ exchanged. When we use $-\infty$ as an order of the Sobolev norm, we mean this estimate 
holds for $-N$ with arbitrarily large $N \in \R^+$.

For $a\in S_{\mathrm{cu},\alpha}^{m,\tilde{m},r}(M,Y)$, we define its essential support $\text{ess supp}(a)$ by defining its complement: 
\begin{defn}
For $\mk{z} \in Z$, we say $\mk{z} \notin \text{ess supp}(a)$ if there exist $\chi_{\mk{z}} \in C_c^\infty(Z)$ being identically 1 near $\mk{z}$ such that $\chi_{\mk{z}}a \in S_{\mathrm{cu},\alpha}^{-\infty,-\infty,r}(M,Y)$. 

For $A=\mathrm{Op}(a)$, we define its (cusp) wavefront set by $\WF_{\mathrm{cu},\alpha}'(A):=\text{ess supp}(a)$. 
\end{defn}
Next we define the ellipticity of $S_{\mathrm{cu},\alpha}^{m,\tilde{m},r}(M,Y)$ and $\Psi_{\mathrm{cu},\alpha}^{m,\tilde{m},r}(M,Y)$ and give the parametrix construction. Then we prove elliptic estimates after showing the boundedness between Sobolev spaces.
\begin{defn}
For $a \in S_{\mathrm{cu},\alpha}^{m,\tilde{m},r}(M,Y)$, we say that $a$ is elliptic at $\mk{z} \in \partial Z$ if there is a neighborhood of $\mk{z}$ in $Z$ on which $a$ satisfies
\begin{align} \label{eq: elliptic definition}
|a| \geq C \tau^{-r}\hat{\rho}^{-m}\tilde{\rho}^{-\frac{\tilde{m}-m}{\alpha}}. 
\end{align}
$a$ is said to be elliptic on $\mathcal{U}$ if it is elliptic on each point of $\,\mathcal{U}$. $A=\mathrm{Op}(a)$ is said to be elliptic at a point or on an open set if and only if $a$ is elliptic at that point or on that open set. The elliptic set of $A$ (resp. $a$) is denoted by $\Ell(A)$ (resp. $\Ell(a)$).
\end{defn}
The wavefront set of a distribution is defined as follows:
\begin{defn}
For $u \in H_{\mathrm{cu},\alpha}^{-\infty,-\infty,r}(M,Y)$, we say $\mk{z} \notin \WF_{\mathrm{cu},\alpha}^{s,\tilde{s},r}(u)$ if and only if there exists $A \in \Psi_{\mathrm{cu},\alpha}^{s,\tilde{s},r}(M,Y)$ which is elliptic at $\mk{z}$ such that $Au \in L_{\mathrm{cu}}^2(M)$.
\end{defn}
The parametrix construction using a Neumann series in classical microlocal analysis generalizes to our situation directly.
\begin{prop}
Suppose $A \in \Psi_{\mathrm{cu},\alpha}^{m,\tilde{m},r}(M,Y)$ is elliptic at $\mk{z} \in \partial Z$, then there exist $B \in  \Psi_{\mathrm{cu},\alpha}^{-m,-\tilde{m},-r}(M,Y),E \in \Psi_{\mathrm{cu},\alpha}^{0,0,0}(M,Y)$ such that $\mk{z}\notin \WF_{\mathrm{cu},\alpha}'(E)$, and following identity holds:
\begin{align}
 B \circ A = \mathrm{Id}+E.  \label{parametrix}
\end{align}
In particular, when $A$ is elliptic everywhere, then we can find $B$ so that $E \in \Psi_{\mathrm{cu},\alpha}^{-\infty,-\infty,r}(M,Y)$.
\end{prop}

\subsection{Mapping properties}
Next we state mapping properties of our operator class, which are analogous to that in previous calculi. We first give a square root construction and reduce the general boundedness to the $L_{\mathrm{cu}}^2$ boundedness, and then prove the $L_{\mathrm{cu}}^2$ boundedness using this square root construction.
\begin{lmm}
Suppose $A \in \Psi_{\mathrm{cu},\alpha}^{0,0,0}(M,Y)$ is a symmetric elliptic operator whose principal symbol has a representative $a \in S^{0,0,0}_{\mathrm{cu}}(M,Y)$ which is lower bounded by a positive constant, i.e. $a \geq c>0$, then there exists a symmetric operator $B \in \Psi_{\mathrm{cu},\alpha}^{0,0,0}(M,Y)$ such that $A=B^2+E$, $E \in \Psi_{\mathrm{cu},\alpha}^{-\infty,-\infty,0}(M,Y)$.
\end{lmm}
\begin{proof}
The proof is the same as Lemma 5.7 of \cite{vasy2018minicourse}. The only difference is that the gain of the error term in each inductive step is 1 and $1-\alpha$ in the first and second index now. But this does not essentially change the proof as long as we have positive gains in these two indices in each step.
\end{proof}

\begin{prop}
Suppose $A\in \Psi_{\mathrm{cu},\alpha}^{s,\tilde{s},r'}(M,Y)$, then for $m,\tilde{m},r' \in \R$, $A$ is a bounded linear operator from $H_{\mathrm{cu},\alpha}^{m,\tilde{m},r}(M,Y)$ to $H_{\mathrm{cu},\alpha}^{m-s,\tilde{m}-\tilde{s},r-r'}(M,Y)$.
\end{prop}
\begin{proof}

According to (\ref{wsob_norm}), $D_{s,\tilde{s},r} := D_{s,\tilde{s}}\tau^{-r}$ 
is an isometry mapping $H_{\mathrm{cu},\alpha}^{s,\tilde{s},r}(M,Y)$ to $L_{\mathrm{cu}}^2(M)$ 
with right inverse $D'_{-s,-\tilde{s},r}:=\tau^r \la \tilde{D} \ra^{-\tilde{s}}\la D \ra^{-s}$, 
which is an isometry as well. So the claim $A \in \mathcal{L}(H_{\mathrm{cu},\alpha}^{m,\tilde{m},r}(M,Y),H_{\mathrm{cu},\alpha}^{m-s,\tilde{m}-\tilde{s},r-r'}(M,Y))$ can be reduced to the claim that $\tilde{A}:=D_{m-s,\tilde{m}-\tilde{s},-(r-r')} A D'_{-m,-\tilde{m},r} \in \Psi_{\mathrm{cu},\alpha}^{0,0,0}(M,Y)$  is a bounded map from $L_{\mathrm{cu}}^2(M)$ to $L_{\mathrm{cu}}^2(M)$. To be more concrete, we write $A$ as
\begin{align}
A=  D'_{-(m-s),-(\tilde{m-\tilde{s}}),-(r-r')} (D_{m-s,\tilde{m}-\tilde{s},(r-r')} A D'_{-m,-\tilde{m},-r}) D_{m,\tilde{m},r},
\end{align}  
where two operators outside the bracket are isometries between weighted Sobolev spaces with appropriate indices. The graphic illustration of this conjugation process is given below.\\
\begin{center}
\begin{tikzcd}
H_{\mathrm{cu},\alpha}^{m,\tilde{m},r}(M,Y) \arrow[r, "A"] 
& H_{\mathrm{cu},\alpha}^{m-s,\tilde{m}-\tilde{s},r-r'}(M,Y) \arrow[d, "D_{m-s,\tilde{m}-\tilde{s},(r-r')}"] \\
L_{\mathrm{cu}}^2(M) \arrow[r,"\tilde{A}"] \arrow[u,"D'_{-m,-\tilde{m},-r}"]
& L_{\mathrm{cu}}^2(M)
\end{tikzcd}
\end{center}
So we only need to show that for any $\tilde{A} = \mathrm{Op}(\tilde{a}) \in \Psi_{\mathrm{cu},\alpha}^{0,0,0}(M,Y)$, we have $\tilde{A} \in \mathcal{L}(L_{\mathrm{cu}}^2(M,Y),L_{\mathrm{cu}}^2(M,Y))$. We apply the proof of Proposition 5.9 of \cite{vasy2018minicourse} to reduce to the proof of the $L_{\mathrm{cu}}^2(M) \rightarrow L_{\mathrm{cu}}^2(M)$ boundedness of $E =\mathrm{Op}(e) \in \Psi_{\mathrm{cu},\alpha}^{-\infty,-\infty,0}(M,Y)$. The modification needed is replacing $S_{\infty,\delta}^{0,0},\Psi_{\infty,\delta}^{0,0},S_{\infty,\delta}^{-1+2\delta,0}$ there by $S_{\mathrm{cu},\alpha}^{0,0,0},\Psi_{\mathrm{cu},\alpha}^{0,0,0},S_{\mathrm{cu},\alpha}^{-1,-(1-\alpha),0}$ respectively. For a complete statement of the classical Schur's lemma, we refer readers to the Lemma in \cite[Appendix~A.1]{grafakos2014modern}. Our final task is to verify the conditions needed for applying Schur's lemma. Let $K_E(z,z')$ be the Schwartz kernel of $E$, where $z=(t,x),z'=(t',x')$. Then we need to show
\begin{align}
\begin{split}
&\sup_{z} \int |K_E(z,z')|d\nu_{\mathrm{cu}}(z') < \infty,\\
&\sup_{z'} \int |K_E(z,z')|d\nu_{\mathrm{cu}}(z) < \infty.              \label{schur}
\end{split}
\end{align}  
By (\ref{eq: cusp quantization}), we have
\begin{align*}
K_E(z,z')= \mathcal{F}^{-1}_{\tilde{\eta}}(e(z,\tilde{\eta}))(z,z-z').
\end{align*}
Using (\ref{diff_1}) with large (in absolute value) negative $m,\tilde{m}$ and large (in each component) positive $\tilde{\gamma}$, we know that we can find $N_E>n$, constant $C_{N_E}$ such that
\begin{align*}
|K_E(z,z')| \leq C_{N_e}\la z-z' \ra^{-N_E},
\end{align*}
which is sufficient for (\ref{schur}).
\end{proof}

Recalling (\ref{parametrix}) and the boundedness of $B$ there as a map from $H_{\mathrm{cu},\alpha}^{m_1-m,\tilde{m_1}-\tilde{m},r-r'}(M,Y)$ to $H_{\mathrm{cu},\alpha}^{m_1,\tilde{m_1},r}(M,Y)$, we obtain the elliptic estimate:
\begin{prop}\label{prop: elliptic estimate}
Suppose $A \in \Psi_{\mathrm{cu},\alpha}^{m,\tilde{m},r'}(M,Y)$ is elliptic, then $\forall N \in \R$, we have
\begin{align}
||u||_{s,\tilde{s},r} \lesssim   ||Au||_{s-m,\tilde{s}-\tilde{m},r-r'}+||u||_{-N,-N,r}     .        \label{est-ell}
\end{align}
When $A$ is not globally elliptic, for $B \in \Psi_{\mathrm{cu},\alpha}^{0,0,0}(M,Y)$ such that $\WF_{\cu,\alpha}'(B) \subset \Ell(A)$, then (\ref{est-ell}) continues to hold with the left hand side replaced by $||Bu||_{s,\tilde{s},r}$.
\end{prop}

\subsection{G\r{a}rding's inequality} \label{sec: Garding's inequality}
Next we prove G\r{a}rding's inequality, which exploits bounds on symbols, for our symbols and operator classes.
\begin{lmm} \label{lmm1}
Let $B,B' \in \Psi_{\mathrm{cu},\alpha}^{s,\tilde{s},r}(M,Y)$ with $\WF_{\mathrm{cu},\alpha}'(B') \subset \mathrm{Ell}(B)$, and suppose that their rescaled symbols $b = \tau^r\hat{\rho}^s \tilde{\rho}^{\frac{\tilde{s}-s}{\alpha}} \sigma_{s,\tilde{s},r}(B), b' = \tau^r\hat{\rho}^s \tilde{\rho}^{\frac{\tilde{s}-s}{\alpha}} \sigma_{s,\tilde{s},r}(B') $ satisfy $|b'| \leq b$ on $\WF_{\mathrm{cu},\alpha}'(B')$, then for any $\delta>0$, there exists a constant $C$ such that:
\begin{align*}
||B'u||_{L^2_{\cu}} \leq (1+\delta)||Bu||_{L^2_{\cu}} + C ||u||_{s-\frac{1}{2}, \tilde{s} - \frac{1-\alpha}{2},r}.
\end{align*}  
\end{lmm}
\begin{proof}
We only need to prove the case $r=0$, since we can replace $u$ by $\tau^{-r}u$. Consider $(1+\delta)^2B^*B - (B')^*B' \in \Psi_{\mathrm{cu},\alpha}^{2s,2\tilde{s},2r}(M,Y)$, whose principal symbol is always strictly positive near $\WF_{\mathrm{cu},\alpha}'(B')$ , hence it has a smooth real square root $e \in S_{\mathrm{cu},\alpha}^{s,\tilde{s},r}(M,Y)$. (Away from $\WF_{\mathrm{cu},\alpha}'(B')$, we just set it to be $(1+\delta) \sigma(B)$ ). Then $K:=(1+\delta)^2 B^*B - (B')^*B' - E^*E \in \Psi_{\mathrm{cu},\alpha}^{2s-1,2\tilde{s}-(1-\alpha),2r}(X)$. Then we apply $K$ to $u$ and pair with $u$ to obtain the desired inequality. 
\end{proof}

\subsection{The uniform pseudodifferential algebra and cusp-conormal functions}
\label{sec: cusp-conormal symbols and functions}
Next we introduce the cusp-conormal function class characterizing perturbations of Kerr-(de Sitter) spacetimes which are natural for our machinery.
\begin{defn} \label{defn: cusp-conomral}
The cusp-conormal functions are functions remain bounded under iterated application of cusp vector fields:
\begin{align}\label{eq: defn_cusp_conormal}
\mathcal{A}_{\cu} = \{ u \in L^\infty (M): Au \in L^\infty(M), \, \forall A \in \mathrm{Diff}_{\cu}(M) \},
\end{align}
where $\mathrm{Diff}_{\cu}(M) = \bigcup_{m \in \N} \mathrm{Diff}_{\cu}^m(M)$ with $\mathrm{Diff}_{\cu}^m(M)$ being the space of $m$-th order cusp differential operators defined after \eqref{cusp_frame}. 
For $\beta_T>0$, its weighted version:
\begin{align} \label{defn_cusp_conormal_weighted}
\mathcal{A}_{\cu}^{\beta_T}(M):= \tau^{\beta_T} \mathcal{A}_{\cu}(M),
\end{align}
consists of cusp-conormal functions that are decaying to order at least ${\beta_T}$. 
\end{defn}
Let $S^2\,{}^{\rm cu}T^*M$ be the symmetric second tensor power of ${}^{\rm cu}T^*M$. Then we denote the vector bundle with the same local frame as $S^2\,{}^{\rm cu}T^*M$ but with coefficients in $\mathcal{A}_{\cu}^{\beta_T}(M)$ after trivialization by $\mathcal{A}_{\rm cu}^{\beta_T}(M;S^2\,{}^{\rm cu}T^*M)$.

The symbol and operator classes that are closely related to $\mathcal{A}_{\cu}^{\beta_T}(M)$ are the \emph{uniform symbol class} and the \emph{uniform pseudodifferential algebra}. Concretely, suppose $(t,x_1,...,x_{n-1})$ are coordinates on $M^\circ$ and $\tilde{\xi}=(\xi_t,\xi_1,...,\xi_{n-1})$ are corresponding momentum variables on $T^*M^\circ$, then the $m$-th order uniform symbol class, denoted by $S_{\infty}^m(M)$, is the space of functions on $T^*M^\circ$ such that
\begin{align} \label{eq: conormal cusp symbol}
|\partial_t^{\mk{f}} \partial_{x}^\beta   \partial_{\tilde{\xi}}^{\tilde{\gamma}} a(t,x_1,\hat{x}_1,\xi_{t},\xi)| \leq C (1+|\xi_t|^2+|\xi|^2)^{\frac{m-|{\tilde{\gamma}}|}{2}}.          
\end{align}
The difference with \eqref{eq: smooth cusp symbol} is that we use $\partial_t$ instead of $\partial_\tau$. In particular, symbols in $S_{\infty}^m(M)$ are not required to be smooth down to $\tau = 0$. Instead they are required to remain bounded under iterative application of cusp vector fields.
The action of the left quantization of $a \in S_{\infty}^m(M)$ supported in a single coordinate patch, which we still denote by $\mathrm{Op}(a)$, on functions $u$ supported in the same coordinate patch is defined by \eqref{eq: cusp quantization} with $a(t^{-1},x,\xi_t,\xi)$ replaced by $a(t,x,\xi_t,\xi)$. And it is defined using a partition of unity for general $a$ and $u$.
This defines an operator class $\Psi_{\infty}^m(M)$ and by definition we have the inclusion 
\begin{align*}
\Psi_{\mathrm{cu}}^m(M) \subset \Psi_{\infty}^m(M).
\end{align*}
These operator classes form a filtered algebra:
\begin{align*}
\Psi_{\infty}^{m_1} \circ \Psi_{\infty}^{m_2} \subset \Psi_{\infty}^{m_1+m_2}.
\end{align*}
Similar to the extension from $\Psi_{\mathrm{cu}}^m(M)$ to $\Psi_{\infty}^m(M)$, we also define the uniform (or conormal) version of our new pseudodifferential algebra when the defining function of $Y_0$ is $x_1$.
\begin{defn}
Let $M^\circ, Y, Y_0$ be as in Section \ref{sec: new symbol class} and the defining function of $Y_0$ is $x_1$, then we define $S_{\infty,\alpha}^{m,\tilde{m}}(M,Y)$ to be the space of functions on $T^*M^\circ$ such that
\begin{align} \label{eq: conormal new symbol}
|\partial_t^{\mk{f}}\partial_{x_1}^\mk{l} \partial_{\hat{x}_1}^\beta   \partial_{\tilde{\xi}}^{\tilde{\gamma}} a(t,x_1,\hat{x}_1,\xi_{t},\xi)| \leq C \hat{\rho}^{-m+|{\tilde{\gamma}}|}(x_1^2+\hat{\rho}^{2\alpha})^{-\frac{\tilde{m}-m}{\alpha}-\mk{l}},
\end{align}
where $\hat{\rho}=(1+|\xi_t|^2+|\xi|^2)^{-1/2}$ still denotes the defining function of fiber infinity. 

For $a \in S_{\infty,\alpha}^{m,\tilde{m}}(M,Y)$, its left quantization $\mathrm{Op}(a)$ is defined to act as in \eqref{eq: cusp quantization} for $u$ supported in a single coordinate patch and defined using a partition of unity for general $u$. The space of all such $\mathrm{Op}(a)$ is denoted by $\Psi_{\infty,\alpha}^{m,\tilde{m}}(M,Y)$. Their weighted version $S_{\infty,\alpha}^{m,\tilde{m},r}(M,Y),\Psi_{\infty,\alpha}^{m,\tilde{m},r}(M,Y)$ are defined by
\begin{align}
\begin{split}
S_{\infty,\alpha}^{m,\tilde{m},r}(M,Y):=t^{r}S_{\infty,\alpha}^{m,\tilde{m}}(M,Y)=\{ t^{r}a: a \in  S_{\infty,\alpha}^{m,\tilde{m}}(M,Y)\},\\
\Psi_{\infty,\alpha}^{m,\tilde{m},r}(M,Y):=t^{r}\Psi_{\infty,\alpha}^{m,\tilde{m}}(M,Y)=\{t^{r}A: A \in \Psi_{\infty,\alpha}^{m,\tilde{m}}(M,Y) \}.
\end{split}
\end{align}
\end{defn}

We notice that the only subtlety of $S^{m,\tilde{m}}_{\mathrm{cu},\alpha}(M,Y)$ compared with $S^m_{\mathrm{cu}}(M)$ is in terms of the regularity in $x_1$. More concretely, we relax the boundedness under iterative application of $\partial_{x_1}$ to the boundedness under iterative application of $\tilde{\rho}\partial_{x_1}$ (see Remark \ref{remark: alternative-proof, symbol, local}). Both $\partial_t$ and $\partial_\tau$ commute with these two vector fields, thus properties in Section \ref{sec: new symbol class} to Section \ref{sec: Garding's inequality} can be proved in the same way with all $S^{*,*,*}_{\mathrm{cu},\alpha}(M,Y)$ and $\Psi^{*,*,*}_{\mathrm{cu},\alpha}(M,Y)$ (with $*$ being their orders appeared in proofs) replaced by $S^{*,*,*}_{\infty,\alpha}(M,Y)$ and $\Psi^{*,*,*}_{\infty,\alpha}(M,Y)$ respectively.

\section{Statement of the Theorem}
\label{setup}
In this section, we state assumptions that we need for our microlocal estimates. For more detailed background, please refer to \cite{hirsch2006invariant} and \cite{wunsch2011resolvent}. The definitions of the space model $X$ and the spacetime $M$ are the same as in Section \ref{sec_microlocal_analysis}.

\subsection{Assumptions near the trapped set}
\label{sec_assumptions}
We first consider $P \in \Psi_{\mathrm{cu}}^{m}(M)$ with real principal symbol $p= \sigma^m_{\mathrm{cu}}(P)$ and characteristic set $\Sigma:=p^{-1}(0) \subset \prescript{\mathrm{cu}}{}{T}^*M \backslash o$, where $o$ is the zero section of $\prescript{\mathrm{cu}}{}{T}^*M$. We use $p_0:=p|_{\tau=0}$ to denote the restriction of the principal to the boundary. We also use $p_0$ to denote the stationary extension of $p_0$ itself: $p_0(t,x,\xi_t,\xi):=p_0(x,\xi_t,\xi)$. 

We make following assumptions (cf. assumptions (P.1)-(P.6) of \cite{hintz2021normally}, all `assumptions' with indices mentioned in latter sections are assumptions listed here) near the trapped set $\Gamma \subset \Sigma \cap \prescript{\mathrm{cu}}{}{S}^*_{\partial M}M $:
\begin{enumerate}
\item $dp_0 \neq 0$ on $\Sigma$ near $\Gamma$.  \label{assumption1}
\item The defining function of fiber infinity $\hat{\rho}$ is homogeneous of degree -1 and $H_{p_0}\hat{\rho}=0$ near $\Gamma$. \label{assumption2}
\item \label{assumption3} The rescaled Hamilton vector field is $\textbf{H}_{p_0}:=\hat{\rho}^{m-1}H_{p_0} \in \mathcal{V}_{\mathrm{cu}}(\prescript{\mathrm{cu}}{}{S}^*M )$. It is tangent to $\Gamma$, and satisfies:
\begin{align*} 
\inf_\Gamma \textbf{H}_{p_0}t>0.
\end{align*}
\item The stable and unstable manifolds at time infinity $\bar{\Gamma}^{u/s}$ are smooth orientable codimension 1 submanifolds of $\Sigma \cap{}^{\mathrm{cu}}S^*_{\partial M}M$ near $\Gamma$. They intersect transversally at $\Gamma$. $\textbf{H}_{p_0}$ is tangent to $\Sigma \cap {}^{\mathrm{cu}}S_{\partial M}^*M$. \label{assumption4}
\item There exists local defining functions $\bar{\phi}^{u/s} \in \mathcal{C}^\infty({}^{\mathrm{cu}}S^*_{\partial M}M)$ of $\bar{\Gamma}^{u/s}$ in a neighborhood of $\Gamma$ in $\Sigma$ as submanifolds of $\Sigma \cap{}^{\mathrm{cu}}S^*_{\partial M}M$ such that \label{assumption5}
\begin{align} 
\textbf{H}_{p_0}\bar{\phi}^u=-\bar{w}^u\bar{\phi}^u,\quad \textbf{H}_{p_0}\bar{\phi}^s=\bar{w}^s\bar{\phi}^s,\label{eq: barphi u,s assumptions1} \\
 \hat{\rho}^{-1}H_{\bar{\phi}^u}\bar{\phi}^s=\hat{\rho}^{-1}\{\bar{\phi}^u,\bar{\phi}^s\}>0, \label{eq: barphi u,s assumptions2}
\end{align}
where $\bar{\phi}^{u/s}$ are also considered as functions on ${}^{\mathrm{cu}}\bar{T}^*M\backslash o$ by stationary homogeneous degree 0 extension. And we assume that
\begin{align}    \label{eq: nu min definition}
\nu_{\min}:=\min\{\inf_\Gamma \bar{w}^u,\inf_\Gamma \bar{w}^s\}>0.
\end{align}      \label{p5}
\item There exist smooth submanifolds $\Gamma^{u/s}$ of $\Sigma$ such that $\Gamma^{u/s} \cap {}^{\mathrm{cu}}S_{\partial M}^*M=\bar{\Gamma}^{u/s}$ and $\textbf{H}_p:=\hat{\rho}^{m-1}H_p$ is tangent to $\Gamma^{u/s}$. There exist defining functions $\phi^{u/s} \in \mathcal{C}^\infty({}^{\mathrm{cu}}S^*M)$ of $\Gamma^{u/s}$ in $\Sigma$ such that
\begin{align}
\phi^{u/s}-\bar{\phi}^{u/s} \in \tau \mathcal{C}^\infty({}^{\mathrm{cu}}S^*M).
\end{align}   \label{p6}
 \label{assumption6}
 \item \label{assumption7} There is a coordinate system that is valid in a neighborhood of the trapped set $\Gamma$ in ${}^{\mathrm{cu}}\bar{T}^*M\backslash o$ such that $\bar{\phi}^u$ above, the defining function of the stationary extension of $\bar{\Gamma}^u$, can be taken as the first coordinate on $X$:
 $\bar{\phi}^u=x_1$. 
\end{enumerate}
Assumptions \eqref{assumption1}-\eqref{assumption6} encode the nature of the normally hyperbolic trapping in our application to the Kerr(-de Sitter) spacetime. On the other hand, assumption \eqref{assumption7} is made so that our construction of the pseudodifferential algebra in Section \ref{sec_microlocal_analysis} applies. This property is not satisfied by asymptotically Kerr(-de Sitter) spacetimes in general, but we will reduce to this setting in Section \ref{application}. 

Assumption (\ref{p5}) and (\ref{assumption6}) imply that the defining function in (\ref{assumption6}) satisfies:
\begin{align}
\textbf{H}_p\phi^u = -w^u\phi^u, \quad \textbf{H}_p\phi^s = w^s\phi^s  \text{ on } \Sigma,   \label{hpphi}
\end{align}
where $w^{u/s}-\bar{w}^{u/s} \in \tau \mathcal{C}^\infty({}^{\mathrm{cu}}S^*M)$.

We need to relax the assumption on $P$ and $\Gamma^{u/s}$ when we are considering conormal perturbations of Kerr(-de Sitter) spacetimes . In this setting, assumption 
\begin{align} \label{eq: P, with uniform PsiDO perturbation}
P \in \Psi_{\mathrm{cu}}^m(M)+\tau^{\beta_T}\Psi_{\infty}^m(M), \beta_T>0.
\end{align}
Correspondingly, in the same manner as \cite[Assumption~(P.6')]{hintz2021normally}, 
assumption \eqref{assumption6} is relaxed to:
\begin{enumerate}[label=(\ref*{assumption6}'),ref=\ref*{assumption6}']
\item \label{assumption6'}
There exist subsets $\Gamma^{u/s} \subset \Sigma$, whose intersections with $S^*M^\circ$ are smooth submanifolds of $S^*M^\circ$ in $\tau>0$, such that
\begin{align*}
\Gamma^{u/s} \cap {}^{\mathrm{cu}}S^*_{\partial M}M = \bar{\Gamma}^{u/s},
\end{align*}
and so that $\textbf{H}_p=\hat{\rho}^{m-1}H_p$ is tangent to $\Gamma^{u/s}$ when $\tau>0$. There exist functions $\phi^{u/s} \in \mathcal{C}^\infty(S^*M^\circ)$ such that (with $\bar{\phi}^{u/s},\bar{w}^{u/s}$ defined as in assumption \eqref{assumption5})
\begin{align*}
\Gamma^{u/s} = \Sigma \cap (\phi^{u/s})^{-1}(0) \text{ near } \Gamma,  \;
\phi^{u/s} - \bar{\phi}^{u/s} \in \mathcal{A}_{\mathrm{cu}}^{\beta_T}({}^{\mathrm{cu}}S^*M).
\end{align*}
In addition, \eqref{hpphi} still holds, but with $w^{u/s}-\bar{w}^{u/s} \in \mathcal{A}_{\mathrm{cu}}^{\beta_T}({}^{\mathrm{cu}}S^*M)$.
\end{enumerate}

We use 
\begin{align} \label{eq: tilde wu ws definition}
\tilde{w}^{u/s} : = \frac{(\phi^{u/s})^2}{(\phi^{u/s})^2 + |\hat{\rho}|^{2\alpha}}w^{u/s} 
\end{align}
to denote the counterpart of $w^{u/s}$ for $\tilde{\rho} = (\phi_{Y_0}^2 + \rho^2)^{1/2} =(x_1^2+\hat{\rho}^{2\alpha})^{1/2}$, which satisfies
\begin{align*}
\textbf{H}_p \tilde{\rho} = -\tilde{w}^u \tilde{\rho} + \alpha \frac{|\hat{\rho}|^{2\alpha}}{\tilde{\rho}}\hat{\rho}^{-1}\textbf{H}_p\hat{\rho},
\end{align*}
with a similar equation for $\tilde{w}^s$ but without the negative sign.  

We apply the construction in Section \ref{sec_microlocal_analysis} to our current setting with $Y=\Gamma^u$, and all other notations play the same role as in Section \ref{sec_microlocal_analysis}. We still use $\Gamma^{u/s},\Gamma$ to denote their lifts to $Z$.

\begin{figure}[h] 
\centering
\includegraphics[scale=0.3]{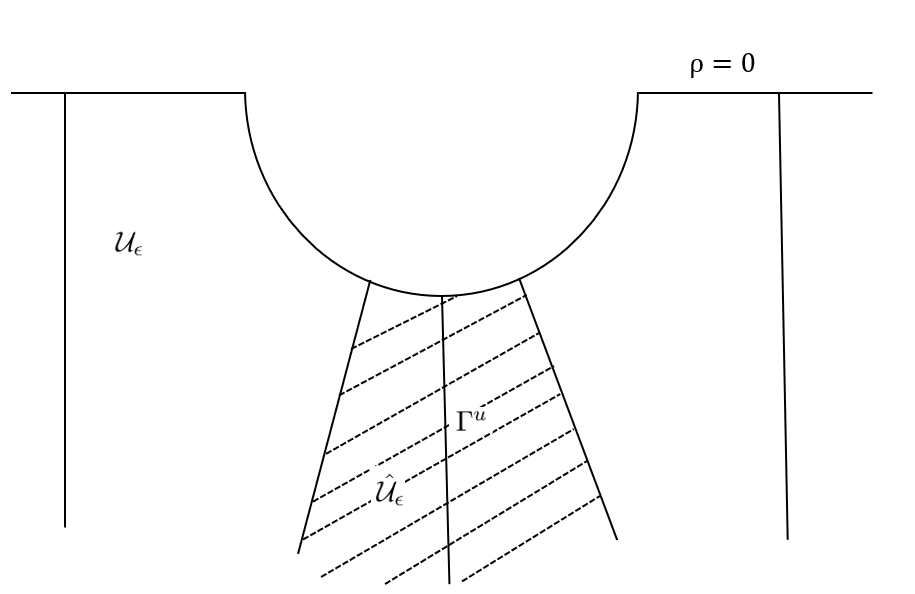}
\caption{Regions defined in \eqref{regions}. Directions parametrized by $\tau,\tb{p},\phi^s$ are suppressed, but remember that we localize to where they are small as well. The shaded region is $\hat{\mathcal{U}}_\epsilon$, and the region between two vertical lines is $\mathcal{U}_\epsilon$.}
\label{fig: regions}
\end{figure}

Let $\epsilon$ be a positive constant, we introduce following regions on $Z$ to facilitate our discussion: 
\begin{align}
\begin{split}
& \mathcal{U}_{\epsilon}:=\{ \tau,|\phi^u|,|\phi^s|,|\textbf{p}| < \epsilon \}, \quad \bar{\mathcal{U}}_{\epsilon}=\mathcal{U}_{\epsilon} \cap \{\tau=0\}, \\
& \hat{\mathcal{U}}_{\epsilon}:=\{ \tau,|\phi^u|,|\frac{\phi^u}{\rho}|,|\phi^s|,|\textbf{p}|< \epsilon \},
\end{split}
\label{regions}
\end{align}
where $\textbf{p}:=\hat{\rho}^mp$ is the normalized symbol of $p$, see Figure \ref{fig: regions} for an illustration. 

The normalized subprincipal symbol of $P$ is:
\begin{align}\label{subps}
{\bf p}_1 :=\hat{\rho}^{m-1} \sigma^{m-1}_{\cu}(\frac{1}{2i}(P-P^*)).    
\end{align}
We assume that
\begin{align}  \label{eq: skew p1 bound}
\sup_{\Gamma} {\bf p}_1 < \frac{1}{2} \nu_{\min},
\end{align}
where $\nu_{\min}$ is defined in \eqref{eq: nu min definition}.

For purposes that will be clear in our positive commutator argument, we take a constant $\lambda \in (0,1)$ such that
\begin{align} \label{eq: lambda condition}
0< \lambda < 1 - \frac{ \sup_{\Gamma} {\bf p}_1}{\nu_{\rm min}},
\end{align}
which exists because of \eqref{eq: skew p1 bound}.
This implies that for $\eta_1$ sufficiently small we have
\begin{align} \label{eq: lambda condition, p1'}
-\textbf{p}_1' - \lambda \tilde{w}^u > 0 \; \text{ on $\mathcal{U}_{2\eta_1}$},
\end{align}
where $\textbf{p}_1'=\textbf{p}_1-w^u$. This is because \eqref{eq: lambda condition, p1'} is equivalent to
\begin{align*} 
-\textbf{p}_1 +w^u - \lambda \tilde{w}^u > 0 \; \text{ on $ \mathcal{U}_{2\eta_1}$},
\end{align*}
Since $0 \leq \tilde{w}^u \leq w^u$ and $\lambda>0$, we only need
\begin{align*} 
-\textbf{p}_1  +(1- \lambda) w^u > 0 \; \text{ on $\mathcal{U}_{2\eta_1}$},
\end{align*}
which holds for $\eta_1$ sufficiently small by the definition of $\nu_{\min}$ and \eqref{eq: lambda condition}. 

\subsection{Improve orders of operators}  \label{sec: improve orders of operators}
Under assumptions above, in particular assumption \eqref{assumption5}-\eqref{assumption7}: the normally hyperbolicity and $\bar{\phi}^u=x_1$ is the first `spatial' coordinate,  
the symbolic expansion of commutators we use in fact have lower order than the order we can count from the general property of the symbol class we constructed in Section \ref{sec_microlocal_analysis}.
Consider $P \in \Psi_{\cu}^{m_1,r_1}(M), A \in \Psi_{\cu,\alpha}^{m_2,\tilde{m}_2,r_2}(M,\Gamma^u)$, 
and view $P$ as an operator lifted to $\Psi_{\cu,\alpha}^{m_1,m_1,r_1}(M,\Gamma^u)$.
Recalling Proposition \ref{prop: composition weighted}, the full symbolic expansion of $[{P},A]$ is:  
\begin{align}
\sum_{\alpha \in \N^n}  \frac{1}{\alpha!} (\partial_\zeta^{\alpha} p D_x^\alpha a -  D_x^\alpha p \partial_\zeta^{\alpha} a),  \label{exp}
\end{align}
where terms with $|\alpha|=k$ lies in $S_{\cu,\alpha}^{m_1+m_2-k,\tilde{m}_1+\tilde{m}_2-k(1-\alpha),r_1+r_2}$. But in fact the proof of Proposition \ref{prop_comp} gave us more concrete characterization of the orders of terms.
Since $P$ is a cusp operator lifted to be an operator in $\Psi_{\cu,\alpha}^{m_1,m_1,r}(M,\Gamma^u)$, the orders of its derivatives are the same as that of cusp symbols.
The only loss in (\ref{exp}) making orders of terms higher
than that of classical results (e.g. the symbolic expansion for cusp pseudodifferential algebra) is the $x_1$-derivatives of $a$.
Writing the local coordinates as $(x,\zeta)$,
consider a typical term $\partial^{k}_{\zeta_1}D_1{p} \partial^k_{x_1}D_2a$, where $D_1$ is an operator differentiating with respect to other momentum variables other than $\zeta_1$,
while $D_2$ is differentiating with respect to spatial variables corresponding to momentum variables in $D_1$.
This term is in $S_{\cu,\alpha}^{m_1+m_2-k-k',m_1+\tilde{m}_2-k(1-\alpha)-k',r_1+r_2}(M,\Gamma^u)$, where $k'$ is the differential 
order of $D_2$ (and $D_1$ as well). 

Suppose $D_1,D_2$ are of order at least $1$, which means $k' \geq 1$, then the term is in $S_{\cu,\alpha}^{m_1+m_2-k-1,m_1+\tilde{m}_2-k(1-\alpha)-1,r_1+r_2}(M,\Gamma^u)$, which, for $k \geq 1$, is contained in $S_{\cu,\alpha}^{m_1+m_2-2,m_1+\tilde{m}_2-2+\alpha,r_1+r_2}(M,\Gamma^u)$.
Then we consider the case where $D_1,D_2$ are vacuum, i.e., $k'=0$.
On the other hand, by the property of normally hyperbolic trapping, we have
\begin{align}  \label{improve_eq1}
\begin{split}
\partial_{\zeta_1}^kp = & - \partial_{\zeta_1}^{k-1}H_{x_1}p\\
= &  \partial_{\zeta_1}^{k-1}H_px_1\\
= & \partial_{\zeta_1}^{k-1}\hat{\rho}^{-(m_1-1)}\textbf{H}_px_1\\
= & \partial_{\zeta_1}^{k-1}(-\hat{\rho}^{-(m_1-1)}w^ux_1+r^u\hat{\rho}p).
\end{split}
\end{align}

When $k=1$, the proof is finished. When $k \geq 2$, we repeat this process for $\partial^{k-1}_{\zeta_1}(\hat{\rho}^{-(m_1-1)}r^up)$. 
Whenever $\partial_{\zeta_1}$ hits $p$, an $x_1$-factor is produced. The only term that remains without an $x_1$-factor is the one where all derivatives fall on other factors, hence it has a $p$-factor.
The sum of all those terms, multiplied by $\partial_{x_1}^ka$, has the form
\begin{align}
\label{expansion_error}
x_1a_1+a_2p,
\end{align}
where $a_1 \in S_{\cu,\alpha}^{m_1+m_2-k,m_1+\tilde{m}_2-k(1-\alpha),r_1+r_2}(M,\Gamma^u), a_2 \in S_{\cu,\alpha}^{m_2-k,\tilde{m}_2-k(1-\alpha),r_2}(M,\Gamma^u)$.
Since $x_1 \in S_{\cu,\alpha}^{0,-\alpha,0}$, we have $x_1a_1 \in S_{\cu,\alpha}^{m_1+m_2-k,m_1+\tilde{m}_2-k(1-\alpha)-\alpha,r_1+r_2}(M,\Gamma^u)$. In particular, when $k=1$, this means that the principal part of the commutator $[P,A]$, given by $H_pa$ is in $S_{\cu,\alpha}^{m_1+m_2-1,m_1+\tilde{m}_2-1,r_1+r_2}(M,\Gamma^u)$, up to an error term $a_2p$. All remaining terms in the expansion in \eqref{exp} have strictly lower orders in the second index since
\begin{align}
-k(1-\alpha)-\alpha<-1 \text{ for } k \geq 2,
\end{align}
which is equivalent to $\alpha<1$. We summarize discussion above as following lemma,
Let $\tilde{a}_1$ be $x_1$ times the sum of $a_1$ above, and $\tilde{a}_2$ be the sum of $a_2$ above, all rescaled by a power of $-i$ as in (\ref{exp}). 
\begin{lmm}
\label{lmm_improve_order}
Suppose $P \in \Psi_{\cu}^{m_1,r_1}(M), A \in \Psi_{\cu,\alpha}^{m_2,\tilde{m}_2,r_2}(M,\Gamma^u)$, and $P$
satisfies assumptions listed in Section \ref{sec_assumptions},
then the full left symbol of $[P,A]$
is given by
\begin{align}
-iH_pa+\tilde{a}_1+\tilde{a}_2p,
\end{align}
where $a_1 \in S_{\cu,\alpha}^{m_1+m_2-2,m_1+\tilde{m}_2-(2-\alpha),r_1+r_2}(M,\Gamma^u), a_2 \in S_{\cu,\alpha}^{m_2-1,\tilde{m}_2-(1-\alpha),r_2}(M,\Gamma^u)$. And the same result holds with all symbol classes and operator classes are replaced by $S^{*,*,*}_{\infty,\alpha}(M,\Gamma^u). \Psi^{*,*,*}_{\infty,\alpha}(M,\Gamma^u)$ with the same orders as above.
\end{lmm}
\begin{rmk} \label{rmk: extend alpha range}
Without the discussion in this section and this Lemma, the main result would hold only for $0<\alpha<\frac{1}{2}$. This condition comes from the requirement that the sub-principal part of commutators, which comes from the sub-sub-principal part in the asymptotic expansion of the full symbol of products $PA,AP$, should have orders lower than that of the principal part of the commutator. But discussion above exploiting the normally hyperbolicity, in particular \eqref{improve_eq1}, extends the result to the full range $0<\alpha<1$.
\end{rmk}

\subsection{The microlocal estimate}
The following theorem is the technical core of our paper. Its proof is the content of Section \ref{sec_commutator} below. 
It does not apply to the wave equation on Kerr(-de Sitter) spacetimes directly since we need assumptions we recall below, but we will reduce that to our current setting in Section \ref{application} via conjugating by a Fourier integral operator.
Throughout this section and Section \ref{sec_commutator}, we assume that the operator $P$ satisfies \eqref{eq: skew p1 bound} and either assumptions \eqref{assumption1}-\eqref{assumption7} made in Section \ref{sec_assumptions} or  
the same set of assumptions with assumption \eqref{assumption6}
replaced by \eqref{assumption6'}. In particular, the stationary extension of $\bar{\Gamma}^u=\Gamma^u \cap \{\tau=0\}$ has defining function $\bar{\phi}^u=x_1$ and the sub-principal symbol of $P$ satisfies:
\begin{align*}
\sup_{\Gamma} {\bf p}_1 < \frac{1}{2} \nu_{\min} .
\end{align*}
We will first prove our estimate under assumptions \eqref{assumption1}-\eqref{assumption7} and indicate changes needed when assumption \eqref{assumption6} is replaced by \eqref{assumption6'}. 

\begin{thm} 
\label{thm: mirolocal estimate}
For $P$ satisfies conditions mentioned above and any $s,N\in \R$, $0<\alpha<1$, $\lambda \in (0,1)$ satisfying (\ref{eq: lambda condition}), there exists $B \in \Psi_{\mathrm{cu},\alpha}^{0,0,r}(M,\Gamma^u)$ which is  elliptic on $\hat{\mathcal{U}}_{\epsilon_0}$ and the front face (in particular it is elliptic on $\Gamma$), and $B_1,G_0 \in \Psi_{\mathrm{cu},\alpha}^{0,0,r}(M,\Gamma^u)$ with $\WF_{\mathrm{cu},\alpha}'(B_1) \cap  \bar{\Gamma}^u = \emptyset,\WF_{\mathrm{cu},\alpha}'(G_0)$ contained in a fixed neighborhood of $\Gamma$ and $G_0$ is elliptic near $\Gamma$ , such that:
\begin{align}
\begin{split}
||B v||_{s,s,0} \lesssim & ||B_1v||_{s+1-\lambda \alpha, s,0} + ||G_0Pv||_{s-m+2-\lambda\alpha,s-m+2-\alpha,0}  + ||v||_{-N,-N,0}. \label{main}
\end{split}  
\end{align}
\end{thm}

\begin{rmk} \label{rmk: lambda, loss explanation}
$\lambda \in (0,1)$ is the relative order on the front face, introduced when we construct the commutator. When $P$ is symmetric up to the sub-leading order, that is $P-P^* \in \Psi^{m-2}_{\mathrm{cu}}(M)$, then $\sup_\Gamma {\bf p}_1= 0$. We then take $\lambda$ satisfying \eqref{eq: lambda condition} to be close to 1. The orders that capture the main feature of this estimate are the second order of $||B v||_{s,s,0}$ and the first order of $||B_1v||_{s+1-\lambda \alpha, s,0}$, hence choosing $\alpha,\lambda$ close to 1 will make the loss of propagation $1-\lambda\alpha$ close to 0.
\end{rmk}

\begin{rmk} \label{rmk: Gamma s case}
Notice that $\Gamma^s$ also satisfies conditions on $Y$ in Section \ref{sec_microlocal_analysis}, thus we can also consider $S_{\cu,\alpha}^{*,*,*}(M,\Gamma^s),\Psi_{\cu,\alpha}^{*,*,*}(M,\Gamma^s)$ (and also their uniform version) and prove a similar estimate with $P$ replaced by $P^*$. Now ${\tb p}_1$ switches sign and we use the backward Hamilton flow instead, thus we still use \eqref{eq: skew p1 bound}\eqref{eq: lambda condition} as bounds on ${\tb p}_1$ and $\lambda$ and \eqref{eq: lambda condition, p1'} becomes $-{\tb p}_1' - \lambda \tilde{w}^s >0$ (with $\tilde{w}^s$ in \eqref{eq: tilde wu ws definition}). The main difference being that, terms arising from differentiating time cut-offs now have favorable sign and can be discarded or combined into the left hand sides of estimates. Thus we don't need a priori control on $\supp (\chi_2^T)'$ and we can assume $\WF_{\cu,\alpha}'(B_1)$ being disjoint from $\Gamma^s$ instead of only $\bar{\Gamma}^s$.
\end{rmk}

\section{Positive commutator argument}
\label{sec_commutator}
We start our positive commutator argument. In proofs below, we first assume that the priori $-N$ order regularity of $v$ is high enough to justify integration by parts and pairing in our proof. Specifically, $-N=s-\frac{1}{2}$ is sufficient. Later, we use a regularization argument to justify the general $N$ case.

We briefly sketch our proof below. Let
\begin{align} \label{eq: Phi u definition}
\Phi^u:= \mathrm{Op}(\phi^u) \in \Psi^{0,-\alpha,0}_{\mathrm{cu},\alpha}(M,\Gamma^u)
\end{align} 
be the quantization of the defining function of $\Gamma^u$. Near $\Gamma$, it is just a multiplication operator acting by multiplying $x_1$ under assumption \eqref{assumption7}.  
\begin{enumerate}
\item In Section \ref{sec_step1}, we use the energy (i.e. a suitable weighted cusp Sobolev norm) of $v$ away from the front face to control the energy of $v$ near the boundary of the front face (i.e. near corner meeting the lift of fiber infinity). 

\item In Section \ref{sec_step2.1}, we propagate this control along the interior of the front face, thus obtaining an estimate for the weak localization $\Phi^u v$ of $v$ away from the unstable trapped set.

\item In Section \ref{sec_step2.2}, we prove a propagation estimate along the unstable direction by means of a real principal type propagation argument for $\Phi^u$ (given the control on $\Phi^u v$ obtained in the previous step).

\item In Section \ref{sec_step2.3}, we prove a propagation estimate for $P$, which exploits the unstable nature of the trapped set.

\item In Section \ref{sec_Combine}, we combine all ingredients to prove Theorem \ref{thm: mirolocal estimate}, assuming that $v$ is sufficiently regular.

\item In Section \ref{sec_regularization}, we apply a regularization argument to remove the assumption that $v$ is sufficiently regular in order to prove Theorem \ref{thm: mirolocal estimate} as stated. 

\item In Section \ref{sec: relaxing regularity}, we describe modifications needed when we relax assumption \eqref{assumption6} to assumption \eqref{assumption6'}.   
\end{enumerate}
The proofs in Section \ref{sec_step2.2} and \ref{sec_step2.3} are very similar to those in \cite[Section~3]{hintz2021normally}. Novelties in our work are the estimates in 4.1 and 4.2, in which we take full advantage of our refined cusp-pseudodifferential algebra.

\subsection{Propagation near the fiber infinity side of the corner}
\label{sec_step1}
In this step, we use the energy on the region away from the front face ($B_1$-term below) to control the energy near the boundary of the front face ($B_0$-term below). 

We first consider a simple estimate. Since $\textbf{H}_p \hat{\rho} = o(\hat{\rho})$ as $\tau \to 0$ by assumption \eqref{assumption2}, terms involving it are negligible. By (\ref{hpphi}), we can find $\textbf{r}^u \in \mathcal{C}^\infty({}^{\mathrm{cu}}S^*M)$ such that 
\begin{align*}
\textbf{H}_p\phi^u = -w^u \phi^u+\textbf{r}^u\textbf{p},
\end{align*}
with $\textbf{H}_p = \hat{\rho}^{m-1}H_p$.  For $[P,\Phi^u]$, we write
\begin{align}
& [P,\Phi^u] = i W^u \Phi^u + R_1P+R_2, \label{[p,phi]} \\
& R_1 \in  \Psi_{\mathrm{cu},\alpha}^{-1,-1,0}(M,\Gamma^u), R_2 \in  \Psi_{\mathrm{cu},\alpha}^{m-2,m-2,0}(M,\Gamma^u),  \label{R1R2}
\end{align}
where $W^u \in \Psi_{\mathrm{cu},\alpha}^{m-1,m-1,0}(M,\Gamma^u)$ is the quantization of $\hat{\rho}^{-(m-1)}w^u$. Although generally for operators in this pseudodifferential algebra, the subprincipal part is $1-\alpha$ order lower on the front face, now all those operators are lifted from the unblown up manifolds, and therefore the asymptotic expansions of their compositions behave the same as in the cusp calculus. Hence the subprincipal symbols are 1 order lower on both boundary faces, which explains the orders of $R_1,R_2$ above. 

Applying both sides of (\ref{[p,phi]}) to $v$, and noticing that $Pv=f$, we have
\begin{align*}
P'v' = f',
\end{align*}
where
\begin{align*}
P': = P-iW^u, v' = \Phi^u v, f' = (\Phi^u + R_1)Pv +R_2v.
\end{align*}
Thus we know for $G \in \Psi_{\mathrm{cu},\alpha}^{0,0,r}(M,\Gamma^u)$, we have
\begin{align*}
||Gf'||_{s-m+2, s-m+2+\lambda \alpha,0} \lesssim & ||G (\Phi^u+R_1) Pv||_{s-m+2,s-m+2+\lambda \alpha,0} 
\\& + ||GR_2v||_{s-m+2,s-m+2+\lambda \alpha,0}.
\end{align*}
Combining this with Proposition \ref{prop: elliptic estimate}, memberships of operators in \eqref{eq: Phi u definition}\eqref{R1R2} and mapping properties of pseudodifferential operators, we have:
\begin{prop}
With $\Phi^u,P,R_1,R_2$ defined as above, then for $G,\tilde{G} \in \Psi_{\mathrm{cu},\alpha}^{0,0,r}(M,\Gamma^u)$ with $\WF_{\mathrm{cu},\alpha}'(G) \subset \mathrm{Ell}(\tilde{G})$, we have:
\begin{align}
\begin{split}
||Gf'||_{s-m+2, s-m+2+\lambda \alpha,0} \lesssim & ||\tilde{G} Pv||_{s-m+2,s-m+2+\lambda\alpha-\alpha,0} + ||\tilde{G}v||_{s,s+\lambda \alpha,0} 
\\ &+ ||v||_{s-1,s+\lambda\alpha-(1-\alpha),r}.
\end{split}
\end{align}
\label{p1}
\end{prop}
There is no restriction on $\lambda$ here, but in latter steps it needs to satisfy (\ref{eq: lambda condition}). Next, we state the main estimate of this step:
\begin{prop}
There exist operators $B_0,\tilde{G} \in \Psi_{\mathrm{cu},\alpha}^{0,0,r}(M,\Gamma^u),B_1 \in \Psi_{\mathrm{cu},\alpha}^{0,-\infty,r}(M,\Gamma^u)$, $B^{ff} \in \Psi_{\mathrm{cu},\alpha}^{-\infty,0,r}(M,\Gamma^u)$ and constants $C_1,\epsilon_1,\eta_1>0$ (in particular, $\eta_1$ is sufficiently small so that \eqref{eq: lambda condition, p1'} holds), 
with $\WF_{\mathrm{cu},\alpha}'(\tilde{G}),\WF_{\mathrm{cu},\alpha}'(B_0),\WF_{\mathrm{cu},\alpha}'(B_1)$ contained in a fixed neighborhood of $\Gamma$, $B_0$ being elliptic on $\{0 \leq |\frac{\rho}{\phi^u}| \leq C_1 \} \cap \mathcal{U}_{\eta_1} $, $\WF_{\mathrm{cu},\alpha}'(B_1)$ being disjoint from both the lift of $\bar{\Gamma}^u$ and the intersection of the front face with $\{\tau=0\}$, $\WF_{\mathrm{cu},\alpha}'(B^{ff}) \subset \{ C_1 \leq \frac{\rho}{\phi^u} \leq 2C_1  \} \cap \mathcal{U}_{2\eta_1}$, $\tilde{G}$ being elliptc on wavefront sets of $B_0,B_1,B^{ff}$, 
such that for any $s, N_1 \in \R$, $0<\alpha<1$, $\lambda$ satisfying (\ref{eq: lambda condition}), we have:
\begin{align}
\begin{split}
 & ||B_0  v||_{s,s+\lambda\alpha-\alpha,0}  + ||B^{ff}v||_{s,s+\lambda\alpha-\alpha,0} \\
& \lesssim || B_1 v ||_{s,-N_1,0} + ||\tilde{G} Pv||_{s-m+1,s-m+1 + \lambda\alpha -\alpha,0} + ||\tilde{G}v||_{s-1,s-1 + \lambda \alpha,0} 
\\& + || v||_{s-\frac{3}{2} , s-\frac{3}{2}+\lambda\alpha+\alpha,r}.
\end{split} \label{est1_prop}
\end{align}
In particular, we can take $-N_1=s$.
\label{step1}
\end{prop}

\begin{proof}
The normalized subprincipal part of $P'$ is
\begin{align}
\textbf{p}_1':= \hat{\rho}^{m-1}\sigma( \frac{1}{2i}(P'-(P')^*) ) = \textbf{p}_1 - w^u.   \label{01}
\end{align}
Our commutant is
\begin{align*} 
\check{a} = \tau^{-r} \tilde{\rho}^{-\lambda} \hat{\rho}^{-s-1+ \frac{m-1}{2}} \chi^{ff}(\frac{\rho}{\phi^u}) \chi^{\inf}(\hat{\rho}\tilde{\rho})\chi^u( (\phi^u)^2 ) \chi^{s} ( (\phi^s)^2 ) \chi_T(\tau) \chi_\Sigma (\textbf{p}),
\end{align*}
where $\chi^{\inf},\chi_T,\chi_\Sigma$ are chosen to be identically 1 on $[-\frac{\eta_1}{2},\eta_1]$,  decreasing on $[0,\infty)$, supported on $[-\eta_1,2\eta_1]$; $\chi^{u},\chi^{s}$ are chosen to be identically 1 on $[-(\frac{\eta_1}{2})^2,\eta_1^2]$,  decreasing on $[0,\infty)$, supported on $[-\eta_1^2,(2\eta_1)^2]$; $\chi^{ff}$ is identically 1 on $[0,C_1]$, decreasing on $[0,\infty)$, supported on $[-C_1,2C_1]$. Since $W^u \in \Psi_{\mathrm{cu},\alpha}^{m-1,m-1,0}(M,\Gamma^u)$, the principal symbols and corresponding Hamilton vector fields of $P'$ and $P$ are the same. Recalling $\rho=\hat\rho^\alpha$, we compute:
\begin{align*}
H_p( \frac{\hat{\rho}^\alpha}{\phi^u} ) & = (\phi^u)^{-1}\alpha \hat{\rho}^{\alpha-1} H_p \hat{\rho} -\hat{\rho}^\alpha \frac{-w^u\phi^u \hat{\rho}^{-(m-1)} + \hat{\rho}^{-(m-1)}\tb{r}^u\tb{p}}{(\phi^u)^2}  \\
  & = \frac{\hat{\rho}^\alpha}{\phi^u}(\alpha\hat{\rho}^{-1}H_p\hat{\rho}+w^u\hat{\rho}^{-(m-1)}+ \hat{\rho}^{-(m-1)} \frac{\tb{r}^u}{\phi^u} \tb{p} ).
\end{align*}
Now take $c$ such that 
\begin{align}  \label{eq: c condition 1}
0< c < \inf_{\mathcal{U}_{2\eta_1}}(-\textbf{p}_1'-\lambda \tilde{w}^u),
\end{align}
which is possible due to \eqref{eq: lambda condition, p1'}. 
Then we have
\begin{align}
\begin{split}
\check{a}H_p \check{a} + \hat{\rho}^{-(m-1)}\textbf{p}_1'\check{a}^2 = & -c \hat{\rho}^{-(m-1)} \check{a}^2 -(\hat{\rho}^{-s-1 } \tilde{\rho}^{-\lambda} b_0)^2  - (  \hat{\rho}^{-s-1}  \tilde{\rho}^{-\lambda} b^s)^2  \\& 
+  (\hat{\rho}^{-s-1}  \tilde{\rho}^{-\lambda}b^u)^2 + (\hat{\rho}^{-s-1}  \tilde{\rho}^{-\lambda}b_T)^2 + hp + e^{\inf} 
\\ &-( \hat{\rho}^{-s-1} \tilde{\rho}^{-\lambda} b^{ff})^2,  \label{sym1}
\end{split}
\end{align}
where $\chi = \chi^{ff} \chi^{\inf} \chi^u \chi^s \chi_T \chi_\Sigma$, and $\hat{\chi}^i$ means the product without $\chi^i, i = ff, \inf ,u,s,T,\Sigma$. $e^{\inf}, b^{ff}$ are terms arising from differentiating $\chi^{\inf},\chi^{ff}$:
\begin{align*}
e^{\inf} & = \tau^{-2r} \tilde{\rho}^{-2\lambda} \hat{\rho}^{-2s-2}(\hat{\chi}^{\inf})^2 \chi^{\inf}(\chi^{\inf})' (\hat{\rho} \textbf{H}_p \tilde{\rho} + \tilde{\rho}\textbf{H}_p\hat{\rho}),\\
b^{ff}  & = \tau^{-r} \hat{\chi}^{ff} (-\chi^{ff} (\chi^{ff})'  (\frac{\hat{\rho}^\alpha}{\phi^u})(\alpha\hat{\rho}^{-1}\textbf{H}_p\hat{\rho}+w^u) )^{1/2}.  
\end{align*}
Since $w^u>0$ near $\Gamma^u$ and $(\chi^{ff})'<0$, we can choose $\chi^{ff}$ so that the square root in $b^{ff}$ is well-defined and smooth. Recall that $\tilde{w}^{u/s} = \frac{(\phi^{u/s})^2}{(\phi^{u/s})^2 + |\hat{\rho}|^{2\alpha}}w^{u/s}$ and $\tilde{\rho}$ satisfies
\begin{align}
\tilde{\rho}^{-1}\textbf{H}_p \tilde{\rho} = -\tilde{w}^u + \alpha \frac{|\hat{\rho}|^{2\alpha}}{\tilde{\rho}^2}\hat{\rho}^{-1}\textbf{H}_p\hat{\rho}. \label{09}
\end{align}
$b_0,b^{u/s},b_T,h$ are defined by
\begin{align}
\label{symbols_step1}
\begin{split}
b_0  = & \tau^{-r} \chi (-\textbf{p}_1'-c + r\tau (\tau^{-2}\textbf{H}_p \tau ) + (s+1- \frac{m-1}{2}) (\hat{\rho}^{-1} \textbf{H}_p \hat{\rho} )  \\& + \lambda \tilde{\rho}^{-1} \textbf{H}_p \tilde{\rho})^{1/2}
\\ = & \tau^{-r} \chi (-\textbf{p}_1'-c + r\tau (\tau^{-2}\textbf{H}_p \tau ) + (s+1- \frac{m-1}{2}+\lambda \alpha (\frac{\rho}{\tilde{\rho}})^2 ) (\hat{\rho}^{-1} \textbf{H}_p \hat{\rho} ) 
\\& - \lambda \tilde{w}^u)^{1/2},\\
b^{u/s} = & \tau^{-r} \hat{\chi}^{u/s} \phi^{u/s} \sqrt{-2w^{u/s}\chi^{u/s}(\chi^{u/s})'}\\
b_T = & \tau^{-r} \hat{\chi}_T \sqrt{\tau \chi_T \chi_T'(\tau^{-2}\textbf{H}_p\tau)},\\
h = & 2\tau^{-2r} \tilde{\rho}^{-2\lambda} \hat{\rho}^{-2s-2+m} \chi ( \hat{\chi}^u (\chi^u)'  \phi^u \textbf{r}^u + \hat{\chi}^s (\chi^s)'\phi^s \textbf{r}^s + m\hat{\chi}_\Sigma \chi_\Sigma' (\hat{\rho}^{-1}\textbf{H}_p \hat{\rho}) ) \\
&+ 2\tau^{-2r} \tilde{\rho}^{-2\lambda} \hat{\rho}^{-2s-2+m-\alpha} \chi \hat{\chi}^{ff} ( \chi^{ff})' (\frac{\hat{\rho}^\alpha}{\phi^u})^2\tb{r}^u.
\end{split}
\end{align}
Recall that $\tau^{-2}\textbf{H}_p\tau = -\textbf{H}_pt < 0$, and it will be bounded if we localize to the region where $\tau$ is small and near $\Gamma$, since $\Gamma$ is compact. So $r\tau(\tau^{-2}\textbf{H}_p \tau ) \rightarrow 0$ as $\tau \rightarrow 0$. Combining \eqref{eq: c condition 1} and the fact that $\hat{\rho}^{-1}\textbf{H}_p\hat{\rho} = o(1)$ as $\tau \rightarrow 0$ due to assumption \eqref{assumption2}, we know that all square roots above are well-defined and smooth if we choose $\chi_T$ to be supported sufficiently close to $\{\tau=0\}$.

Quantize both sides of (\ref{sym1}), and apply it to $v'$, then pair it with $v'$.  We use $T_{j}$ to denote the quantization of $\hat{\rho}^{-j}$ and $F_{j}$ to denote the quantization of $\tilde{\rho}^{-j}$, and for other parts we just use upper case letters to denote the quantization of the corresponding symbol represented by lower case letters. 

 We get, with $A:=\check{A}^*\check{A}, \, \check{A} \in \Psi_{\mathrm{cu},\alpha}^{s+1-\frac{m-1}{2},s+1-\frac{m-1}{2}+\lambda\alpha,r}(M,\Gamma^u)$:
\begin{align}
\begin{split}
\mathrm{Im} \la f' , Av' \ra  &  =   \la (\frac{1}{2}i[P',A] + \frac{P'-(P')^*}{2i}A)v',v' \ra \\
                            &  = \la (- c \la (T_{(m-1)/2}\check{A})^*(T_{(m-1)/2}\check{A}) - (T_{s+1}F_{\lambda}B_0)^*(T_{s+1}F_{\lambda}B_0) \\
                            &  + (T_{s+1}F_{\lambda}B^u)^*(T_{s+1}F_{\lambda}B^u) + (T_{s+1}F_{\lambda}B_T)^*(T_{s+1}F_{\lambda}B_T) +   H^*P' \\
                            & +E + E^{\inf}-(T_{s+1}F_{\lambda}B^{ff})^*(T_{s+1}F_{\lambda}B^{ff}) )      v',v' \ra,
\end{split} \label{10}
\end{align}
where $E \in \Psi_{\mathrm{cu},\alpha}^{2s+1,2s+2\lambda\alpha+1+\alpha,2r}(M,\Gamma^u)$ is the error term introduced because (\ref{sym1}) concerns only principal symbols. 
It collects terms of order lower than the principal part in (\ref{sym1}). They are generated from the full symbol expansion of composition of $P'A,AP',(P')^*A$, hence the highest order term among them is $2$ order lower in the first index and $2-\alpha$ order lower (applying Lemma \ref{lmm_improve_order}, where the $a_2$-term is absorbed into the $h$-term (and $H$-term after quantization) in (\ref{symbols_step1})) in the second index compared with the product $P'A$. Here the order refers to the sum of the orders of two operators in a pairing. 
And $B_0,B^u,B_T,B^{ff} \in \Psi_{\mathrm{cu},\alpha}^{0,0,r}(M,\Gamma^u), H \in \Psi_{\mathrm{cu},\alpha}^{2s+2-m+\alpha,2s+2-m+\alpha+2\lambda \alpha,2r}(M,\Gamma^u)$. 

Next we investigate properties of operators above through analyzing their symbols. $B_0$ is elliptic on $\mathcal{U}_{\eta_1}$. We can choose $\chi^u$ so that $(\chi^u)'$ vanishes identically when $(\phi^u)^2 \in [0,\eta_1^2]$ to ensure that $\WF_{\mathrm{cu},\alpha}'(B^u) \cap \mathcal{U}_{\eta_1} = \emptyset$. Because of the $\chi_T'$ factor, which vanishes identically near $\{\tau=0\}$, we know $\WF_{\mathrm{cu},\alpha}'(B_T) \cap \{\tau=0\} = \emptyset$. By the support condition of $\chi^{ff}$ and its derivative, we know: $\WF_{\mathrm{cu},\alpha}'(B^{ff}) \subset \{ C_1 \leq \frac{\rho}{\phi^u} \leq 2C_1  \} \cap \mathcal{U}_{2\eta_1}$. On the other hand, for the left hand side of (\ref{10}) we have:
\begin{align}
\begin{split}
\mathrm{Im} \la f' , Av' \ra & \leq |\la f',Av' \ra | = |\la P'v',\check{A}^*\check{A}v' \ra| = |\la \check{A}P'v',\check{A}v' \ra| \\
 & \leq  c||\check{A}v'||^2_{ \frac{m-1}{2},\frac{m-1}{2},0} + \frac{1}{2c}||\check{A}P'v'||_{ -\frac{m-1}{2} , -\frac{m-1}{2},0}^2 
 \end{split} \label{pair1}
\end{align}

Combining boundedness of $\hat{\rho}^{-1}\textbf{H}_p\hat{\rho}$ and (\ref{09}), we know that $\tilde{\rho}^{-1}\textbf{H}_p\tilde{\rho}$ is also bounded. Consequently $\hat{\rho} \textbf{H}_p \tilde{\rho} + \tilde{\rho}\textbf{H}_p\hat{\rho} = \hat{\rho}\tilde{\rho}(\tilde{\rho}^{-1}\textbf{H}_p\tilde{\rho} + \hat{\rho}^{-1}\textbf{H}_p\hat{\rho})\in S_{\mathrm{cu},\alpha}^{-1,-1-\alpha,0}(M,\Gamma^u)$. So we have $E^{\inf} \in  \Psi_{\mathrm{cu},\alpha}^{2s+1,2s+1 + 2\lambda \alpha-\alpha,2r}(M,\Gamma^u)$.
Combining these inequalities, we have
\begin{align}  \label{n3.32}
\begin{split}
& c||\check{A} v'||^2_{\frac{m-1}{2},\frac{m-1}{2},0} + ||B_0 v'||^2_{s+1,s+1+\lambda\alpha,0} + ||B^s v'||_{s+1,s+1+\lambda\alpha,0}^2  \\
& +  ||B^{ff}v'||^2_{s+1 ,s+1+\lambda\alpha,0}
\\ & \leq  ||B^u v'||^2_{s+1,s+1+\lambda \alpha,0}  + ||B_Tv'||_{s+1,s+1+\lambda \alpha,0}^2 + |\la P'v',Hv' \ra|\\ &+ (||G_1'v'||^2_{s+ \frac{1}{2} , s+\lambda\alpha+\frac{1+\alpha}{2} ,0}  + C||v'||^2_{s- \frac{1}{2} , s+\lambda \alpha - \frac{1}{2} +\frac{3}{2}\alpha ,r})  + c||\check{A}v'||^2_{ \frac{m-1}{2},\frac{m-1}{2},0} +
\\ & \frac{1}{2c}||\check{A}P'v'||_{ -\frac{m-1}{2} , -\frac{m-1}{2},0},  
\end{split} 
\end{align}
where $G_1' \in \Psi_{\cu,\alpha}^{0,0,0}(M,\Gamma^u)$ is elliptic on $\WF'_{\cu,\alpha}(E)$ and $\WF'_{\cu,\alpha}(E^{\inf})$, hence terms in the bracket controll $\la Ev',v' \ra$ and $\la E^{\inf} v' , v' \ra$.

Use Cauchy-Schwartz inequality to control $|\la P'v',Hv' \ra |$. Because of the $\chi$ factor in $h$, $Hv'$ is microlocalized near the support of $\check{a}$ as well, using $G_1'$ as the microlocalizer again and enlarge its wavefront set if necessary, we obtain
\begin{align}
\begin{split}
|\la P'v',Hv' \ra |  \lesssim & ||G_1'P'v'||^2_{s-m+2,s-m+2+\lambda\alpha,0} + ||Hv'||^2_{-(s-m+2),-(s-m+2+\lambda\alpha),0} \\
& + ||v'||^2_{-N,-N,r} \\
 \lesssim & ||G_1'P'v'||^2_{s-m+2,s-m+2+\lambda\alpha,0} + ||G_1' v'||^2_{s+\alpha,s+\alpha+\lambda\alpha,0} \\ & + ||v'||^2_{-N,-N,r}.
\end{split}  \label{11}
\end{align}
Recalling the discussion following (\ref{10}), we obtain $(\WF_{\mathrm{cu},\alpha}'(B^u) \cup \WF_{\mathrm{cu},\alpha}'(B_T)) \cap \bar{\mathcal{U}}_{\eta_1}=\emptyset$. Combine the first two terms on the right hand side of (\ref{n3.32}) and control them by $B_1\in \Psi_{\mathrm{cu},\alpha}^{0,0,r}(M,\Gamma^u)$ microlocalized in a neighborhood of $\Gamma$ but away from $\Gamma$ itself. They satisfy $\WF_{\mathrm{cu},\alpha}'(B^u) \cup \WF_{\mathrm{cu},\alpha}'(B_T) \subset \Ell(B_1)$ and $\WF_{\mathrm{cu},\alpha}'(B_1) \cap \bar{\mathcal{U}}_{\eta_1} = \emptyset$. In particular, $B_1 \in \Psi_{\mathrm{cu},\alpha}^{0,-\infty,r}(M,\Gamma^u)$. 
Since we can choose $G_1'$ to be elliptic on $\WF'_{\cu,\alpha}(A)$,
 $||\check{A}P'v'||_{ -\frac{m-1}{2} , -\frac{m-1}{2},0}$ is controlled 
by the norm $||G_1'P'v'||_{s-m+2,s-m+2+\lambda\alpha,0}=||G_1'f'||_{s-m+2,s-m+2+\lambda\alpha,0}$ 
using the mapping property of $\check{A}$ and the elliptic estimate. 
Since $0<\alpha<1$, we can combine $||G_1'v'||_{s+ \frac{1}{2} , s + \lambda \alpha + \frac{1+\alpha}{2} ,0}$ and $||G_1' v'||_{s+\alpha,s+\alpha+\lambda\alpha,0} $ to be $||G_1' v'||_{s+\bar{\alpha},s+\lambda\alpha+\frac{1+\alpha}{2},0}$, where $\bar{\alpha}:=\max \{\alpha,\frac{1}{2}\}$. Combining discussion above and \eqref{n3.32}, we obtain
\begin{align}
\begin{split}
& ||B_0 \Phi^u v||_{s+1,s+1+\lambda \alpha,0} +  ||B^{ff}\Phi^uv||_{s+1 ,s+1+\lambda\alpha,0}  \\
\lesssim &  || B_1 \Phi^u v ||_{s+1,s+1+\lambda \alpha,0} +  ||G_1'\Phi^u v||_{s+\bar{\alpha}, s+\lambda\alpha + \frac{1+\alpha}{2},0} \\ 
&+ ||\Phi^u v||_{s-\frac{1}{2},s+\lambda\alpha - \frac{1}{2}+2\alpha,r}    + ||G_1'f'||_{s-m+2,s-m+2+\lambda \alpha,0} ,
\end{split}
\end{align}
where $G_1' \in \Psi_{\mathrm{cu},\alpha}^{0,0,r}(M,\Gamma^u)$ is microlocalized near $\mathrm{supp}\,\check{a}$. Since the support conditions of $B_0$ and $G_1'$ are the same (possibly with different bounds on $\frac{\phi^u}{\rho}$), we can iterate this estimate to control the $||G_1' v'||_{s+\bar{\alpha},s+\lambda\alpha+\frac{1+\alpha}{2},0}$ term at the cost of enlarging the wavefront set of $G_1'$. In each iteration we can improve the first index by $1-\bar{\alpha}$ and the second index by $\frac{1-\alpha}{2}$. So after finite iterations, we can
absorb this term into the $\Phi^uv$-term.
We denote the new operator playing the role of $G_1'$ in the last iteration by $\tilde{G}$. Then we apply Proposition \ref{p1} to control $||\tilde{G}f'||_{s-m+2,s-m+2+\lambda\alpha,0}$ and use the mapping property of $\Phi^u \in \Psi_{\mathrm{cu},\alpha}^{0,-\alpha,r}(M,\Gamma^u)$ to control norms of $\Phi^uv$:
\begin{align}  \label{est1.1}
\begin{split}
& ||B_0 \Phi^u v||_{s+1,s+1+\lambda\alpha,0} +  ||B^{ff}\Phi^uv||_{s+1,s+1+\lambda\alpha,0} \\
\lesssim & || B_1 \Phi^u v ||_{s+1,s+1+\lambda \alpha,0} + ||\tilde{G} Pv||_{s-m+2,s-m+2 + \lambda\alpha -\alpha}  \\
&+ ||\tilde{G}v||_{s,s + \lambda \alpha,0} + ||v||_{s-\frac{1}{2} , s+ \lambda \alpha - \frac{1}{2}+\alpha,r}, 
\end{split}
\end{align}
where $\lambda$ satisfies (\ref{eq: lambda condition}), and $\WF_{\mathrm{cu},\alpha}'(G_1')\subset \Ell(\tilde{G})$. In particular, we can require $\WF_{\mathrm{cu},\alpha}'(\tilde{G}) \subset \{ \frac{\rho}{\phi^u} \leq 3C_1 \} \cap \mathcal{U}_{3\eta_1}$.   (\ref{est1.1}) will be the same as (\ref{est1_prop}) if we can replace $\Phi^uv$ by $v$, which is what we proceed to show next. Recall that $B_1 \in \Psi_{\mathrm{cu},\alpha}^{0,-\infty,r}(M,\Gamma^u)$, hence the second order of the $B_1$-term can be taken to be any real number.
 
Recall that $\phi^u \geq C_1 \rho$ on the wavefront set of $B_0$, hence $\Phi^u$ is elliptic near $\WF_{\mathrm{cu},\alpha}'(B_0)$ as an operator in $\Psi_{\mathrm{cu},\alpha}^{-\alpha,-\alpha,0}(M,\Gamma^u)$, and we can write $B_0 \Phi^u = \Phi^u B_0 + [B_0,\Phi^u]$ with the commutator term having lower order, and similarly for $B_1,B^{ff}$, with commutator terms controlled by $||\tilde{G}v||_{s,s + \lambda \alpha,0}$, 
hence conclude from above estimate:
\begin{align*}
\begin{split}
& ||B_0  v||_{s+1,s+1+\lambda\alpha-\alpha,0} +  ||B^{ff}v||_{s+1,s+1+\lambda\alpha-\alpha,0} 
\\& \lesssim || B_1 v ||_{s+1,s+1+\lambda \alpha -\alpha,0} + ||\tilde{G} Pv||_{s-m+2,s-m+2 + \lambda\alpha -\alpha,0}  \\
&+ ||\tilde{G}v||_{s,s + \lambda \alpha,0} + || v||_{s-\frac{1}{2} , s+ \lambda \alpha - \frac{1}{2}+\alpha,r},
\end{split}
\end{align*}
This implies, if we replace $s+1$ by $s$, and recall that $B_1 \in \Psi_{\mathrm{cu},\alpha}^{0,-\infty,r}(M,\Gamma^u)$:
\begin{align}
\begin{split}
&||B_0  v||_{s,s+\lambda\alpha-\alpha,0}  + ||B^{ff}v||_{s,s+\lambda\alpha-\alpha,0} \\
\lesssim & || B_1 v ||_{s,-N_1,0} + ||\tilde{G} Pv||_{s-m+1,s-m+1 + \lambda\alpha -\alpha,0}  \\
&+ ||\tilde{G}v||_{s-1,s-1 + \lambda \alpha,0} + ||v||_{s-\frac{3}{2} , s-\frac{3}{2}+ \lambda \alpha + \alpha,r} , 
\end{split} \label{est1}
\end{align}
for any $N_1 \in \R$. The constant implicitly included in `$\lesssim$' depends on $N_1$ as well. The effect of $B^{ff}$ term here is that we can extend the region where we have control to the area near the boundary of $\supp (\chi^{ff})$,  
 where $(\chi^{ff})'$ has a lower bound but $\chi^{ff}$ is almost vanishing.
\end{proof}

\begin{rmk}
The constant in (\ref{est1_prop}) depends on how close the wavefront set of $B_0$ is to the lift of $\Gamma^u$. We consider the $B_0$-term first. What affects the constant in the estimate is, as we approach the lift of $\Gamma^u$, the ellipticity of $\Phi^u$ is becoming weaker and weaker and we do not have a uniform lower bound of its principal symbol. The constant in the elliptic estimate is proportional to the reciprocal of the lower bound of the symbol. Recall that $\WF_{\mathrm{cu},\alpha}'(B^0) \subset \{ \phi^u \geq C_1^{-1}\rho \}$, hence the way that we can `push' the estimate near $\Gamma^u$ is letting $C_1 \rightarrow \infty$ and the constant is proportional to $C_1$.  The term $||B^{ff}v'||_{s+1,s+1+\lambda \alpha,0}$ is treated in a similar manner. And a more accurate version of (\ref{est1}) is
\begin{align}
\begin{split}
& ||B_0  v||_{s,s+\lambda\alpha-\alpha,0} + ||B^{ff} v||_{s,s+\lambda \alpha-\alpha,0} \\
\lesssim & C_1 ( || B_1 v ||_{s,-N,0} + ||\tilde{G} Pv||_{s-m+1,s-m+1 + \lambda\alpha -\alpha,0}  \\
&+ ||\tilde{G}v||_{s-1,s-1 + \lambda \alpha,0} + ||v||_{s-\frac{3}{2} , s-\frac{3}{2}+ \lambda \alpha + \alpha,r}),
\end{split} 
\end{align}
where $C_1$ comes from $\WF_{\mathrm{cu},\alpha}'(B^0) \subset \{ \phi^u \geq C_1^{-1}\rho \}$ and the implicit constant in the inequality is independent of $C_1$.
\end{rmk}

\subsection{Propagation near the \texorpdfstring{$\Gamma^u$}{Gamma u} side of the corner}
\label{sec_step2.1}
In this step, our positive commutator argument propagates the control we obtained above along the interior of the front face.
\begin{prop} 
There exist operators $\tilde{B}_0,\tilde{G}_1 \in \Psi_{\mathrm{cu},\alpha}^{0,0,r}(M,\Gamma^u),\tilde{B}_1 \in \Psi_{\mathrm{cu},\alpha}^{0,-\infty,r}(M,\Gamma^u)$ and a small constant $\epsilon_0>0$, with $\WF_{\mathrm{cu},\alpha}'(\tilde{B}_0),\WF_{\mathrm{cu},\alpha}'(\tilde{G}_1),$ $\WF_{\mathrm{cu},\alpha}'(\tilde{B}_1)$ contained in a fixed neighborhood of $\Gamma$, $\tilde{B}_0$ being elliptic on $\hat{\mathcal{U}}_{2\epsilon_0}$ defined in (\ref{regions}), $\WF_{\mathrm{cu},\alpha}'(\tilde{B}_1) \cap \{ \frac{|\phi^u|}{|\hat{\rho}|^\alpha} < \epsilon_0 , \tau=0 \} = \emptyset$, so that:
\begin{align}
\begin{split}
||\tilde{B}_0 \Phi^u v||_{s+1,s+1,0} \lesssim & ||\tilde{B}_1\Phi^u v||_{s+1,s+1,0}+||\tilde{G}_1Pv||_{s-m+2,s-m+2,0} 
\\& + ||\tilde{G}_1v||_{s,s,0} + ||v||_{s-\frac{1}{2},s-\frac{1}{2}+2\alpha,r},
\end{split} \label{est2.1}
\end{align}
\label{step2.1}
\end{prop}
\begin{proof}
We choose the commutant
\begin{align*}
a = \check{a}^2, \,  \check{a} =  \tau^{-r} \hat{\rho}^{-s-1+(m-1)/2} \chi_1^u(\frac{\phi^u}{|\hat{\rho}|^\alpha})\chi_1^s( (\phi^s)^2 ) \chi_1^T(\tau) \chi_1^\Sigma (\textbf{p})  ,    
\end{align*}
where $\textbf{p} = \hat{\rho}^m p$, and $\chi_1^{i}(\cdot),i=u,T,\Sigma$ are identically $1$ on $[-2\epsilon_0,2\epsilon_0]$, monotonically decreasing to 0 when the variable goes from $2\epsilon_0$ to $3\epsilon_0$, identically 0 on $[3\epsilon_0,\infty)$, and extended to $\R$ in an even manner. $\chi_1^s$ is identically $1$ on $[-(2\epsilon_0)^2,(2\epsilon_0)^2]$, monotonically decreasing to 0 when the variable goes from $(2\epsilon_0)^2$ to $(3\epsilon_0)^2$, identically 0 on $[(3\epsilon_0)^2,\infty)$, and extended to $\R$ in an even manner.
Then we consider the pairing
\begin{align*}
\mathrm{Im} \la f' , Av' \ra  &  =   \la (\frac{1}{2}i[P',A] + \frac{P'-(P')^*}{2i}A)v',v' \ra .
\end{align*}

Recalling \eqref{eq: skew p1 bound}, which shows $\inf_{\mathcal{U}_{3\epsilon_0}} -{\tb p}_1'>\frac{1}{2}\nu_{\min}>0$ for $\epsilon_0$ sufficiently small. Then we we choose $c$, which can be different $c$ in the previous subsection, such that
\begin{align}  \label{eq: c condition, 2}
 0< c <  \inf_{\mathcal{U}_{3\epsilon_0}} -{\tb p}_1'.
\end{align}
Defining $\chi_1:=\chi_1^u \chi_1^s \chi_1^T \chi_1^\Sigma$ and defining $\hat{\chi}_1^i$ to be the product without $\chi_1^i$, then the principal symbol of the operator on the right hand side is
\begin{align}
\begin{split}
\check{a}H_p \check{a} + \hat{\rho}^{-m+1}\textbf{p}_1'\check{a}^2 = & -c \hat{\rho}^{-m+1}\check{a}^2- (\hat{\rho}^{-s-1}\tilde{b}_0)^2 - (\hat{\rho}^{-s-1}\tilde{b}^s)^2 \\ 
& + (\hat{\rho}^{-s-1}\tilde{b}^u)^2 + (\hat{\rho}^{-s-1}\tilde{b}_T) + hp,
\end{split}   \label{2.1}
\end{align}
where
\begin{align*}
\tilde{b}_0  & = \tau^{-r} \chi_1 (-\textbf{p}_1'-c + r\tau (\tau^{-2}\textbf{H}_p \tau ) + (s+1- \frac{m-1}{2})\hat{\rho}^{-1} \textbf{H}_p \hat{\rho}  )^{1/2},\\
\tilde{b}^s  & = \tau^{-r} \hat{\chi}_1^s \phi^s \sqrt{ - 2w^s \chi^s (\chi^s)'}.
\end{align*}
We can arrange that terms in the bracket in $\tilde{b}_0$ all together is positive when we use $\chi_T$ to localize to the region on which $\tau$ is small and use the fact $\hat{\rho}^{-1}\textbf{H}_p\hat{\rho} = o(1)$ as $\tau \rightarrow 0$ (again by assumption \eqref{assumption2}). We compute the derivative of $\phi^u/|\hat{\rho}|^\alpha$:
\begin{align*}
\textbf{H}_p(\phi^u/|\hat{\rho}|^\alpha) = -(w^u\phi^u + \tb{r}^u \textbf{p})|\hat{\rho}|^{-\alpha} - \alpha \phi^u |\hat{\rho}|^{-\alpha-1}\textbf{H}_p\hat{\rho}.
\end{align*}
We have
\begin{align*}
& \tilde{b}^u =  \tau^{-r} \hat{\chi}_1^u ( (-w^u\phi^u |\hat{\rho}|^{-\alpha} - \alpha (\phi^u |\hat{\rho}|^{-\alpha}) \hat{\rho}^{-1}\textbf{H}_p\hat{\rho}) \chi^u (\chi^u)' )^{1/2},\\
& \tilde{b}_T = \tau^{-r} \hat{\chi}_1^T ( \tau \chi_1^T (\chi_1^T)'( \tau^{-2}\textbf{H}_p\tau) )^{1/2},\\
& h = 2 \tau^{-2r} \hat{\rho}^{-2s-2+m} \chi_1 ( \hat{\chi}_1^u(\chi_1^u)' \tb{r}^u  \hat{\rho}^{-\alpha}+  \hat{\chi}_1^s(\chi_1^s)' \tb{r}^s   + m \hat{\chi}_1^\Sigma (\chi_1^\Sigma)' (\hat{\rho}^{-1} \textbf{H}_{p} \hat{\rho}) ).
\end{align*}
We can choose $\chi_1^u$ so that $\phi^u \hat{\rho}^{-\alpha}$ is small enough on $\mathrm{supp} \check{a}$, and notice that $|\tilde{\rho}^{-1}\phi^u| \leq 1,\, |\tilde{\rho}^{-1}\hat{\rho}| \leq 1$, and $\textbf{H}_p\hat{\rho} = o(\hat{\rho})$ as $\tau \rightarrow 0$. 
Evaluating $\mathrm{Im} \la f' , Av' \ra $ and using (\ref{2.1}), we know
\begin{align}
\begin{split}
\mathrm{Im} \la f' , Av' \ra  &  =   \la (\frac{1}{2}i[P',A] + \frac{P'-(P')^*}{2i}A)v',v' \ra \\
&  = \la (- c \la (T_{(m-1)/2}\check{A})^*(T_{(m-1)/2}\check{A}) - (T_{s+1}\tilde{B}_0)^*(T_{s+1}\tilde{B}_0) \\
                            &  + (T_{s+1}\tilde{B}^u)^*(T_{s+1}\tilde{B}^u) + (T_{s+1}\tilde{B}_T)^*(T_{s+1}\tilde{B}_T) +   H^*P'+E)      v',v' \ra,   
\end{split}
\label{pair2.1}
\end{align}
with $\check{A} \in \Psi_{\mathrm{cu},\alpha}^{s+1-(m-1)/2,s+1-(m-1)/2,r}(M,\Gamma^u),  \tilde{B}_0, \tilde{B}^u, \tilde{B}_T \in  \Psi_{\mathrm{cu},\alpha}^{0,0,r}(M,\Gamma^u),$ $H^* \in \Psi_{\mathrm{cu},\alpha}^{2s+2-m+\alpha,2s+2-m+\alpha,2r}(M,\Gamma^u), E \in \Psi_{\mathrm{cu},\alpha}^{2s+1,2s+1+\alpha,2r}(M,\Gamma^u)$. $E$ is the error term introduced because (\ref{2.1}) concerns only principal symbols and apply Lemma \ref{lmm_improve_order} when we count orders and the $a_2$-part in that Lemma is absorbed in $h$. We estimate the left hand side by
\begin{align*}
\mathrm{Im} \la f' , Av' \ra & \leq |\la f',Av' \ra | = |\la P'v',\check{A}^*\check{A}v' \ra| = |\la \check{A}P'v',\check{A}v' \ra| \\
 & \leq  c||\check{A}v'||^2_{\frac{m-1}{2},\frac{m-1}{2},0} + \frac{1}{2c}||\check{A}P'v'||_{ -\frac{m-1}{2} , -\frac{m-1}{2},0}^2.
\end{align*}
Combining above equations and inequalities, the estimate we obtain is
\begin{align}
\begin{split} 
& c||\check{A}v'||_{(m-1)/2,(m-1)/2,0}^2+ ||\tilde{B}_0 v'||_{s+1,s+1}^2+||\tilde{B}^sv'||^2_{s+1,s+1,0} \\
 \leq & ||\tilde{B}^u v'||_{s+1,s+1,0}^2 + ||\tilde{B}_T v'||_{s+1,s+1,0} + | \la P'v',Hv' \ra | +  (||G_1'v'||^2_{s+ \frac{1}{2} , s+\lambda\alpha+\frac{1+\alpha}{2} ,0}  
 \\&+ C||v'||^2_{s- \frac{1}{2} , s+\lambda \alpha - \frac{1}{2} +\frac{3}{2}\alpha ,r}) +c||\check{A}v'||_{(m-1)/2,(m-1)/2,0}^2 \\
& +\frac{1}{2c}||\check{A}P'v'||_{-(m-1)/2,-(m-1)/2,0}^2.
\end{split}  \label{n3.32'}
\end{align}
Terms in the bracket are similar to terms in the bracket in (\ref{n3.32}). The Sobolev order of the leading part of terms in the bracket is 2 orders lower in the first index and $2(1-\alpha)$ order lower in the second index compared with the product $P'A$. Applying Cauchy-Schwartz inequality to $| \la P'v',Hv' \ra |$, since $\supp h \subset \supp \check{a}$, we can add a microlocalizer at the cost of introducing a lower order error term:
\begin{align*}\
|\la P'v',Hv' \ra| \lesssim &||\tilde{G}_1P'v'||^2_{s-m+2,s-m+2,0} + ||Hv'||^2_{-(s-m+2),-(s-m+2),-r}
\\ &+||v||_{s+\alpha-\frac{1}{2},s+\alpha-\frac{1}{2},r}.
\end{align*}
Since $H \in \Psi_{\mathrm{cu},\alpha}^{2s+2-m+\alpha,2s+2-m+\alpha,2r}(M,\Gamma^u)$, we know 
\begin{align*}
||Hv'||_{-(s-m+2),-(s-m+2),-r} \lesssim ||\tilde{G}_1v'||_{s+\alpha,s+\alpha,0}
\end{align*}
for $\tilde{G}_1 \in \Psi_{\cu,\alpha}^{0,0,r}(M,\Gamma^u)$ that is elliptic on $\WF'_{\cu,\alpha}(H)$. The terms $c||\check{A}v'||_{(m-1)/2,(m-1)/2,0}^2$ on both sides of (\ref{n3.32'}) cancel out. Dropping the $\tilde{B}^s$-term on the left hand side of (\ref{n3.32'}), we have
\begin{align*}
||\tilde{B}_0 \Phi^u v||_{s+1,s+1,0}  \lesssim & || \tilde{B}_1 \Phi^u v||_{s+1,s+1,0} + ||\tilde{G}_1f'||_{s-m+2,s-m+2,0}  \\& + ||\tilde{G}_1\Phi^u v||_{s+\frac{1}{2},s+\frac{1}{2}+\alpha,0} + C ||\Phi^uv||_{s-\frac{1}{2},s-\frac{1}{2}+2\alpha,r},
\end{align*}
where $\tilde{B}_1 \in \Psi_{\mathrm{cu},\alpha}^{0,0,r}(M,\Gamma^u)$ is elliptic on $\WF_{\mathrm{cu},\alpha}'(B^u)\cup \WF_{\mathrm{cu},\alpha}'(B_T)$. We iterate this estimate to control $||\tilde{G}_1\Phi^u v||_{s+\frac{1}{2},s+\frac{1}{2}+\alpha,0}$. Then apply Proposition \ref{p1} to estimate $\tilde{G}_1f'$-term:
\begin{align*}
||\tilde{B}_0 \Phi^u v||_{s+1,s+1,0} \lesssim  & ||\tilde{B}_1\Phi^u v||_{s+1,s+1,0} + ||\tilde{G}_1Pv||_{s-m+2,s-m+2,0}  \\ &+ ||\tilde{G}_1v||_{s,s,0} + ||v||_{s-\frac{1}{2},s-\frac{1}{2}+2\alpha,r},
\end{align*}
where the wavefront set of $\tilde{G}_1$ is enlarged when we iterate the estimate. Our operators satisfy: $\tilde{B}_0$ is elliptic on $\hat{\mathcal{U}}_{2\epsilon_0}$, $\WF_{\mathrm{cu},\alpha}'(\tilde{B}_1) \cap \{ |\phi^u|/|\hat{\rho}|^\alpha < \epsilon_0 , \tau=0 \} = \emptyset$.
\end{proof}

\subsection{Propagation to the trapped set}
\label{sec_step2.2}
In this step we consider the dynamics of $\tilde{\phi}^u:=\phi^u/|\hat{\rho}|^\alpha$, and denote its quantization by $\tilde{\Phi}^u \in \Psi_{\mathrm{cu},\alpha}^{\alpha,0,0}(M,\Gamma^u)$. The estimate we obtain in this step is to control the energy on $\Gamma$ (i.e., the $B_\delta$-term in \eqref{est6-prop}) by the energy on the region away from $\Gamma^s$ (i.e., the $E_\delta$-term). The regions of microlocalization depend on a parameter $\delta \in (0,\frac{\epsilon_0}{3})$ to be specified later (see the discussion before \eqref{est8}). The reason that we keep track of this $\delta^{1/2}$-factor is that we will absorb this $E_\delta v$-term by a bootstrapping argument in combination with \eqref{est2.3} below.
\begin{figure}[H]
\centering
\includegraphics[scale = 0.4]{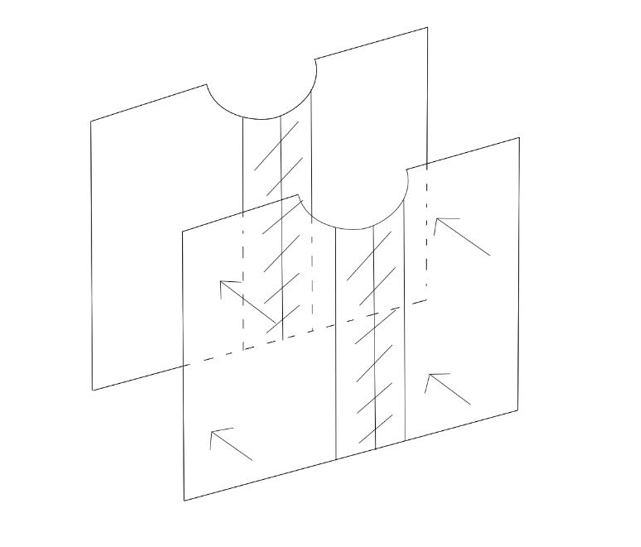}
\caption{The propagation in Proposition \ref{propstep2.2}. 
The plane that is more inward is $\Gamma^s$, the unstable manifold; while the outward plane is its translation, $\{\phi^s=\delta\}$.
Arrows indicate the direction from where we assume a priori regularities to where we conclude regularities.}
\label{fig_step2.2} 
\end{figure}

\begin{prop}
There exist $B_\delta,E_\delta,B_2,\tilde{G}_2 \in \Psi^{0,0,r}(M,\Gamma^u)$ with $B_\delta$ being elliptic on $\hat{\mathcal{U}}_{\epsilon_0}\cap \{|\phi^s| \leq 3\delta \}$ and $\WF'_{\cu,\alpha}(B_\delta) \subset \hat{\mathcal{U}}_{2\epsilon_0} \cap \{|\phi^s|\leq 4\delta\}$, $\WF_{\cu,\alpha}'(E_\delta)$ characterized by \eqref{edelta} below, 
$\WF_{\cu,\alpha}'(\tilde{G}_2) \subset \hat{\mathcal{U}}_{6\epsilon_0} \cap \{|\phi^s| \leq 5\delta\}$,
$\WF_{\cu,\alpha}'(B_2) \subset \hat{\mathcal{U}}_{2\epsilon_0} \backslash \hat{\mathcal{U}}_{\epsilon_0}$ (in particular, $\WF_{\cu,\alpha}'(B_2)$ is away from $\bar{\Gamma}^u$), and constants $C,\tilde{C}$ independent of $\delta$ such that
\begin{align} \label{est6-prop}
\begin{split}
||B_\delta v||_{s,s,0} \leq & C( \delta^{1/2}||E_\delta v||_{s,s,0} + ||B_2 v||_{s+1-\alpha,s+1-\alpha,0}   
\\ &+ ||\tilde{G}_2Pv||_{s-\alpha -m+2,s-\alpha-m+2,0} + \tilde{C} ||v||_{s-\frac{1}{2},s -\frac{1-\alpha}{2},r}).
\end{split}
\end{align}
\label{propstep2.2}
\end{prop}
\begin{proof}
We use a new commutant
\begin{align}
a= \check{a}^2 , \, \check{a} = \tau^{-r} \hat{\rho}^{-s-1/2} \psi_2(\phi^s)\chi_2^u(\phi^u/|\hat{\rho}|^\alpha) \chi_2^T(\tau)\chi_2^\Sigma(\textbf{p}), \label{2.2.1}
\end{align}
where $\chi_2^u,\chi_2^T,\chi_2^\Sigma$ are identically 1 on $[-\epsilon_0,\epsilon_0]$ and supported in $[-2\epsilon_0,2\epsilon_0]$. Notice that
\begin{align*}
H_{\phi^u/|\hat{\rho}|^{\alpha}}  &= \phi^u H_{|\hat{\rho}|^{-\alpha}} + |\hat{\rho}|^{-\alpha} H_{\phi^u}\\
& = -\alpha \phi^u |\hat{\rho}|^{-\alpha-1}H_{\hat{\rho}} + |\hat{\rho}|^{-\alpha} H_{\phi^u},
\end{align*}
we compute the principal symbol of $\frac{i}{2} [\tilde{\Phi}^u,A]$ as two parts corresponding to two terms above:
\begin{align}
\check{a}H_{\tilde{\phi}^u} \check{a} =   -\alpha \tilde{\phi}^u \check{a}  \hat{\rho}^{-1} H_{\hat{\rho}} \check{a} + \hat{\rho}^{-\alpha} \check{a} H_{\phi^u} \check{a} .  \label{sym3}
\end{align}
Recalling our assumptions on $\phi^{u/s}$, we can choose $C_\phi>1$ such that
\begin{align}
C_\phi^{-1} \leq \tb{H}_{\phi^u}\phi^s \leq C_\phi \text{\quad on \quad} \hat{\mathcal{U}}_{3\epsilon_0} , \label{cphi}
\end{align}
where $\tb{H}_{\phi^u}: = \hat{\rho}^{-1}H_{\phi^u}$ is the normalized Hamilton vector field associated with $\phi^u$. We recall the construction of the cutoff with respect to $\phi^s$ in step 2 of the proof of Theorem 3.9 of \cite{hintz2021normally}, i.e. we choose smooth $\psi_2$ such that: $\supp \psi_2 \subset [-4\delta,2\epsilon_0],0\leq \psi_2 \leq 1$, $\psi_2' \geq 0$ on $(-\infty,0]$, $\psi_2(-3\delta) \geq \frac{1}{4}$, $\psi_2'\geq \frac{1}{12\delta}$ on $[-3\delta,3\delta]$, $\psi_2'\geq -\frac{1}{\epsilon_0}$ on $[0,\infty)$, $\psi_2 \leq 12\delta \psi_2'$ on $[-4\delta,3\delta]$. The last inequality automatically holds on $[-3\delta,3\delta]$ by constant bounds on each side respectively, and we can choose appropriate $\psi_2$ to extend it to $[-4\delta,3\delta]$. 
Then decompose $\psi_2$ as
\begin{align*}
\psi_2^2 = \psi_{2-}^2 + \psi_{2+}^2,
\end{align*}
where $\psi_{2-}=\psi_2$ on $(-\infty,3\delta]$ and $\supp \psi_{2-} \subset [-4\delta,4\delta]$, and $\supp \psi_{2+} \subset(3\delta,2\epsilon_0]$.
Define $a_\pm$ to be symbols obtained from $a$ by replacing $\psi_2$ by $\psi_{2+}$ and $\psi_{2-}$ respectively. Hence we have $\check{a}^2 = \check{a}_+^2 + \check{a}_-^2$, and $e_3 = -\alpha \tilde{\phi}^u \hat{\rho}^{-1} H_{\hat{\rho}} \check{a}^2 = -\frac{1}{2} \alpha \tilde{\phi}^u \hat{\rho}^{-1} H_{\hat{\rho}} (\check{a}_+^2 + \check{a}_-^2)$. Consequently, we have
\begin{align}
 \psi_2 \psi_2'  = \frac{1}{24\delta}\psi_{2-}^2 + \frac{1}{24\delta}\tilde{b}_\delta^2 -\tilde{e},  \label{2.3}
\end{align}
where $\supp \tilde{e} \subset (3\delta,2\epsilon_0],|\tilde{e}| \leq \epsilon_0^{-1},\tilde{b}_\delta \geq 0, \supp \tilde{b}_\delta \subset [-4\delta,4\delta]$ and 
\begin{align}
\tilde{b}_\delta \geq \frac{1}{4} \text{ on } [-3\delta,3\delta].  \label{bdelta0}
\end{align}
We consider two terms in (\ref{sym3}) separately. For the second term, using \eqref{2.3}, we have:
\begin{align}
\begin{split}
\check{a}H_{\phi^u}\check{a} = & \frac{1}{48C_\phi\delta} \hat{\rho}\check{a}^2 + (\hat{\rho}^{-s}\tilde{b}_-)^2+ \frac{1}{24C_\phi\delta}(\hat{\rho}^{-s}b_\delta)^2 \\
&+ \hat{\rho}^{-2s}e_1 - \hat{\rho}^{-2s}e_2,
\end{split}
\label{sym2}
\end{align}
where terms introduced by differentiating $\chi_2^u,\chi_2^T,\chi_2^\Sigma$ are included in $e_1$:
\begin{align*}
b_- = & \tau^{-r}\psi_{2-} \chi_2^u \chi_2^T \chi_2^\Sigma ( \frac{1}{24\delta}(\textbf{H}_{\phi^u}\phi^s -\frac{1}{2}C_{\phi}^{-1}) -r\tau (\tau^{-2}\textbf{H}_{\phi^u}\tau) \\
&- (s+1/2)( \hat{\rho}^{-1}H_{\phi^u}\hat{\rho} ))^{1/2},\\
b_\delta = & \tau^{-r} \chi_2^u \chi_2^T \chi_2^\Sigma \tilde{b}_\delta \sqrt{C_\phi \textbf{H}_{\phi^u}\phi^s},\\
e_1 = &  \tau^{-2r}\psi_2^2 (\tau^2 (\chi_2^u)^2\chi_2^T(\chi_2^T)'(\chi_2^\Sigma)^2 \tau^{-2}\textbf{H}_{\phi^u}\tau +   \tau^{-2r}(\chi^u)^2 \chi_T^2 \chi_\Sigma (\chi_2^\Sigma)'(\textbf{H}_{\phi^u}\textbf{p}))\\
& +  \tau^{-2r} (\chi_2^T)^2 (\chi_2^\Sigma)^2 \chi_2^u (\chi_2^u)' (\phi^u/|\hat{\rho}|^\alpha) \hat{\rho}^{-1}\textbf{H}_{\phi^u} \hat{\rho},\\
e_2 = & \tau^{-2r} (\chi_2^u)^2 (\chi_2^T)^2 \chi_\Sigma^2 ( (\textbf{H}_{\phi^u}\phi^s)\tilde{e} +  \tau^{-2r}\psi_{2+}^2 ( r\tau(\tau^{-2}\textbf{H}_{\phi^u}\tau ) + (s+\frac{1}{2}) ( \hat{\rho}^{-1} \textbf{H}_{\phi^u} \hat{\rho}) ) ).
\end{align*} 
 Combining properties of $\tilde{e},\psi_{2+},\chi_2^u,\chi_2^T,\chi_2^\Sigma$, we know that $e_1,e_2$ satisfy
\begin{align}
\begin{split} 
& \supp e_1 \cap \hat{\mathcal{U}}_{\epsilon_0} = \emptyset, \\
& \supp e_2 \subset \{ \phi^s>3\delta \}, \, |e_2| \leq C_\phi\epsilon_0^{-1} + C    \label{e1e2},
\end{split}
\end{align}
where $C$ is independent of $\delta,\epsilon_0$. $b_\delta$ satisfies
\begin{align}
\tau^r b_\delta \geq \frac{1}{4} \text{ \, on }\quad \hat{\mathcal{U}}_{\epsilon_0} \cap \{ |\phi^s|\leq 3\delta \}  \label{bdelta}
\end{align}
by (\ref{cphi}) and (\ref{bdelta0}).

Although $\tilde{\phi}^u\in S_{\mathrm{cu},\alpha}^{\alpha,0,0}(M,\Gamma^u)$, $\chi_2^u$ is supported on the region where $\frac{\phi^u}{\rho}$ is bounded, so $\chi^u_2 \tilde{\phi^u} \in S_{\mathrm{cu},\alpha}^{0,0,0}(M,\Gamma^u)$. Consequently $e_1 \in S_{\mathrm{cu},\alpha}^{0,0,2r}(M,\Gamma^u)$. 

Next we consider the first term in (\ref{sym3}). Define $e_3: =  -\alpha \tilde{\phi}^u \check{a}  \hat{\rho}^{-1} H_{\hat{\rho}} \check{a}$.
Recalling the definition of $\check{a}$ in (\ref{2.2.1}), factors of terms in $e_3$ involving differentiation are: $\hat{\rho}^{-1}H_{\hat{\rho}}\phi^s,\, \hat{\rho}^{-1}H_{\hat{\rho}} \phi^s, \, \hat{\rho}^{-1}H_{\hat{\rho}}\tau, \, \hat{\rho}^{-1}H_{\hat{\rho}}p$, all of which are bounded quantities. In addition, because of the $\chi^u$ factor in $\check{a}$, $\tilde{\phi}^u$ is bounded on $\WF_{\mathrm{cu},\alpha}'(E_3)$, so $e_3$ quantizes to be $E_3 \in \Psi_{\mathrm{cu},\alpha}^{2s,2s,2r}(M,\Gamma^u)$, which introduces another term $|\la E_3 v,v \ra|$.  We define
\begin{align*}
e_{31} = \frac{1}{2}\hat{\rho}^{-\frac{1+\alpha}{2}} H_{\hat{\rho}} \check{a}, \quad e_{32} = \check{a} \hat{\rho}^{-\frac{1-\alpha}{2}},
\end{align*}
and we have $e_3 = e_{31}e_{32}\tilde{\phi}^u$. And define the version truncated by $\psi_{2\pm}$ as:
\begin{align*}
e_{31 \pm} = \frac{1}{2}\hat{\rho}^{-\frac{1+\alpha}{2}} H_{\hat{\rho}} \check{a}_\pm, \quad e_{32\pm} = \check{a}_\pm \hat{\rho}^{-\frac{1-\alpha}{2}}.
\end{align*}
We define $E_{31},E_{32},E_{31\pm},E_{32\pm}$ to be operators with principal symbols denoted by corresponding lower-case letters. In addition, we require them to be symmetric (if not, replace $E_{ij}$ by $\frac{1}{2}(E_{ij}+E_{ij}^*)$). Thus $E_{31\pm} \in \Psi_{\mathrm{cu},\alpha}^{s-1+\frac{\alpha}{2},s-1+\frac{\alpha}{2},r}(M,\Gamma^u), \, E_{32\pm} \in \Psi_{\mathrm{cu},\alpha}^{s+ 1 - \frac{\alpha}{2} , s+ 1 - \frac{\alpha}{2},r}(M,\Gamma^u) $.
 Hence we have $E_3 = (E_{31+}^*E_{32+}+E_{31-}^*E_{32-})\tilde{\Phi}^u$ and:
\begin{align*}
 |\la E_3 v,v \ra | & =|\la (E_{31+}^*E_{32+}+E_{31-}^*E_{32-}) \tilde{\Phi}^u v, v \ra |\\
                    & \leq | \la E_{32+}\tilde{\Phi}^u v , E_{31+}v \ra| + | \la E_{32-}\tilde{\Phi}^u v , E_{31-}v \ra| \\
                    & \leq ||E_{32+}\tilde{\Phi}^uv||_{0,0,0}^2 + ||E_{31+}v||_{0,0,0}^2 + ||E_{32-}\tilde{\Phi}^uv||_{0,0,0}^2 + ||E_{31-}v||_{0,0,0}^2
\end{align*}
We can control $E_{31\pm}$-term up to a constant by $||v||_{s-1+\frac{\alpha}{2},s-1+\frac{\alpha}{2},r}$, and control $E_{32\pm}$-term by $||\check{A}\tilde{\Phi}^u v||_{\frac{1-\alpha}{2},\frac{1-\alpha}{2},0}$ by mapping properties and counting the order of operators, where we have used the fact that $e_{32\pm}= \check{a}_\pm \hat{\rho}^{-\frac{1-\alpha}{2}}$, hence they are just order-shifted version of each other, i.e. $||E_{32\pm}\tilde{\Phi}^uv||_{L_{\mathrm{cu}}^2}$ is equivalent to $||\check{A}\tilde{\Phi}^uv||_{\frac{1-\alpha}{2},\frac{1-\alpha}{2},0}$. To summarize, we have
\begin{align}
\begin{split}
 |\la E_3 v,v \ra| & \leq ||E_{32+}\tilde{\Phi}^uv||_{0,0,0}^2 + ||E_{31+}v||_{0,0,0}^2 + ||E_{32-}\tilde{\Phi}^uv||_{0,0,0}^2 + ||E_{31-}v||_{0,0,0}^2\\
                   & \lesssim  ||\check{A}_+\tilde{\Phi}^u v||_{\frac{1-\alpha}{2},\frac{1-\alpha}{2},0} + ||\check{A}_-\tilde{\Phi}^u v||_{\frac{1-\alpha}{2},\frac{1-\alpha}{2},0} + ||v||_{s-1+\frac{\alpha}{2},s-1+\frac{\alpha}{2},r}
\end{split}
\label{est7}
\end{align}
Now we evaluate the pairing $\mathrm{Im} \la \tilde{\Phi}^u v, Av \ra = \la \frac{i}{2} [\tilde{\Phi}^u,A]v,v \ra$. First we quantize both sides of (\ref{sym3}) and then apply it to $v$ and pair with $v$. We have $\check{A},\check{A}_\pm\in \Psi_{\mathrm{cu},\alpha}^{s+1/2,s+1/2,r}(M,\Gamma^u),\, B_-,B_\delta \in \Psi_{\mathrm{cu},\alpha}^{0,0,r}(M,\Gamma^u)$, $E_1,E_2 \in \Psi_{\mathrm{cu},\alpha}^{0,0,2r}(M,\Gamma^u)$. Setting $\tilde{v}:= \tilde{\Phi}^uv$ to simplify notations, we get:
\begin{align*}
& \frac{1}{48C_\phi \delta}||\check{A}_- v||^2_{\frac{-1+\alpha}{2},\frac{-1+\alpha}{2},0} + ||B_-v||^2_{s+\frac{\alpha}{2},s+\frac{\alpha}{2},0} + \frac{1}{24C_\phi\delta}||B_\delta v||_{s+\frac{\alpha}{2},s+\frac{\alpha}{2},0}^2\\
\leq & |\la E_1 v, v \ra| + | \la E_2v,v \ra| + | + |\la \check{A}_- \tilde{v},\check{A}_-v \ra| + |\la \check{A}_+ \tilde{v},\check{A}_+v \ra|  \\
& + \tilde{C}||v||_{s +\frac{\alpha - 1 }{2} ,s +\frac{\alpha}{2} -\frac{1-\alpha}{2},r}^2 + |\la E_3 v,v\ra| \\
\leq & |\la E_1 v, v \ra| + | \la E_2v,v \ra| + |\la E_3v,v \ra|   + \frac{1}{48C_\phi \delta}||\check{A}_- v||_{\frac{-1+\alpha}{2},\frac{-1+\alpha}{2},0}^2 \\
&+ 24C_\phi\delta||\check{A}_- \tilde{v} ||_{\frac{1-\alpha}{2},\frac{1-\alpha}{2},0}^2  + \epsilon_0 ||\check{A}_+ \tilde{\Phi}^u v||_{\frac{1-\alpha}{2},\frac{1-\alpha}{2},0}^2 + \frac{1}{2\epsilon_0} ||\check{A}_+v ||_{\frac{\alpha-1}{2},\frac{\alpha-1}{2},0}^2 \\& + \tilde{C}||v||_{s +\frac{\alpha - 1 }{2} ,s +\frac{\alpha}{2} -\frac{1-\alpha}{2},r }^2\\
\lesssim  & |\la E_1 v, v \ra| + | \la E_2v,v \ra|  + \frac{1}{48C_\phi \delta}||\check{A}_- v||_{\frac{-1+\alpha}{2},\frac{-1+\alpha}{2},0}^2 + 24C_\phi\delta||\check{A}_- \tilde{v} ||_{\frac{1-\alpha}{2},\frac{1-\alpha}{2},0}^2 \\
& + \epsilon_0 ||\check{A}_+ \tilde{\Phi}^u v||_{\frac{1-\alpha}{2},\frac{1-\alpha}{2},0}^2 + \frac{1}{2\epsilon_0} ||\check{A}_+v ||_{\frac{\alpha-1}{2},\frac{\alpha-1}{2},0}^2 + \tilde{C}||v||_{s +\frac{\alpha - 1 }{2} ,s +\frac{\alpha}{2}-\frac{1-\alpha}{2},r}^2.
\end{align*}
$\tilde{C}||v||_{s +\frac{\alpha-1}{2},s +\frac{\alpha}{2} -\frac{1-\alpha}{2},r}=\tilde{C}||v||_{s-\frac{1}{2}+\frac{\alpha}{2} ,s-\frac{1}{2}+\alpha,r}$ arises because (\ref{sym3}) concerns only principal symbols.
In the last step, we used (\ref{est7}). Multiplying $24C_\phi \delta$ on both sides and then taking square root, applying Lemma \ref{lmm1},  we obtain:
\begin{align}
\begin{split}
||B_\delta v||_{s + \frac{\alpha}{2} ,s + \frac{\alpha}{2},0} \leq & C( \sqrt{\delta/\epsilon_0}||\tilde{E}_\delta v||_{s +\frac{\alpha}{2},s + \frac{\alpha}{2},0} + ||B_2v||_{s + \frac{\alpha}{2},s + \frac{\alpha}{2} ,0} + (\delta + \sqrt{\delta \epsilon_0}) \\&||\tilde{G}_2\tilde{\Phi}^uv||_{s+1-\frac{\alpha}{2},s+1-\frac{\alpha}{2},0}  + \tilde{C}||v||_{s-\frac{1}{2}+\frac{\alpha}{2} ,s-\frac{1}{2}+\alpha,r}),
\end{split}  \label{est4}
\end{align}
where $\tilde{G}_2$ is microlocalized near $\supp \check{a}$, which contains $\Gamma$. The term $\delta ||\tilde{G}_2\tilde{\Phi}^uv||_{s+1-\frac{\alpha}{2},s+1-\frac{\alpha}{2}}$ controls $24C_\phi\delta||\check{A}_- \tilde{v} ||_{\frac{1-\alpha}{2},\frac{1-\alpha}{2},0}$ and the term $\sqrt{\epsilon_0\delta} ||\tilde{G}_2\tilde{\Phi}v||_{s+1-\frac{\alpha}{2},s+1-\frac{\alpha}{2},0}$ controls $\epsilon_0 ||\check{A}_+ \tilde{\Phi}^u v||_{\frac{1-\alpha}{2},\frac{1-\alpha}{2},0}$. 
Let $\tilde{E}_\delta$ be an operaotr in $\Psi_{\mathrm{cu},\alpha}^{0,0,r}(M,\Gamma^u)$ that controls the $E_2v$-term and $\check{A}_+v$-term up to a constant independent of $\delta$ by the elliptic estimate. In particular, its principal symbol $\tilde{e}_\delta$ satisfies
\begin{align} \label{eq: tidlde E condition}
|\tau^r\tilde{e}_\delta| \leq 1 \text{ and } \WF'_{\mathrm{cu}}(\tilde{E}_\delta) \subset \{ \phi^s>3\delta \}.  
\end{align}
$B_2$ is chosen to control $|\la E_1 v, v \ra|$. $B_2$ satisfies
\begin{align}
\WF'_{\cu,\alpha}(B_2) \cap \hat{\mathcal{U}}_{\epsilon_0} = \emptyset,  \label{B1}
\end{align}
which is possible by (\ref{e1e2}).

Next we estimate $(\delta + \sqrt{\delta \epsilon_0}) || \tilde{G}_2\tilde{\Phi}^u v||_{s+1-\frac{\alpha}{2},s+1-\frac{\alpha}{2},0}$, which is equivalent to $(\delta + \sqrt{\delta \epsilon_0}) || \tilde{G}_2\Phi^u v||_{s+1+\frac{\alpha}{2},s+1+\frac{\alpha}{2},0}$. We control this term using (\ref{est2.1}). Enlarging the wavefront set of $\tilde{G}_2$ if necessary, this results in an error term $(\delta+\sqrt{\delta \epsilon_0})||\tilde{G}_2v||_{s +\frac{\alpha}{2} ,s + \frac{\alpha}{2},0}$. 
By Lemma \ref{lmm1} and (\ref{bdelta}), we know
\begin{align}
\begin{split}
||\tilde{G}_2v||_{s+\frac{\alpha}{2},s+\frac{\alpha}{2},0} \leq & 2||B_\delta v||_{s + \frac{\alpha}{2},s + \frac{\alpha}{2},0} + 2||E_\delta v||_{s + \frac{\alpha}{2},s + \frac{\alpha}{2},0} + \tilde{C}||B_2v||_{s + \frac{\alpha}{2} ,s + \frac{\alpha}{2},0}  \\ &+ \tilde{C}||v||_{s-\frac{1}{2}+\frac{\alpha}{2} ,s-\frac{1}{2}+\alpha,0}\\
=  &  2||B_\delta v||_{s+\frac{\alpha}{2},s + \frac{\alpha}{2},0} + 2||E_\delta v||_{s + \frac{\alpha}{2},s + \frac{\alpha}{2},0} + \tilde{C}||B_2v||_{s + \frac{\alpha}{2} ,s + \frac{\alpha}{2},0} \\& +\tilde{C}||v||_{s-\frac{1}{2}+\frac{\alpha}{2} ,s-\frac{1}{2}+\alpha,r},
\end{split} \label{G2}
\end{align}
where $E_\delta$ satisfies
\begin{align}
\begin{split}
& |\tau^r\sigma(E_\delta)| \leq 1, \, \tau^r\sigma(E_\delta) = 1 \text{ on } \{3\delta \leq  \phi^s \leq 4\epsilon_0 , |\tau| \leq \epsilon_0, |\tilde{\phi}^u| \leq \epsilon_0, |\textbf{p}| \leq \epsilon_0  \},\\
& \WF_{\mathrm{cu},\alpha}'(E_\delta) \subset \{ \frac{5}{2}\delta \leq \phi^s \leq 5\epsilon_0, |\tau| \leq 2\epsilon_0, |\tilde{\phi}^u| \leq 2\epsilon_0 , |\textbf{p}| \leq 2\epsilon_0 \}.
\end{split} \label{edelta}
\end{align}
Substitute (\ref{G2}) in (\ref{est4}), and choose $\delta$ small enough so that the $B_\delta$-term on the right hand side can be absorbed by the $B_\delta$-term on the left hand side and fix $\epsilon_0$, we get: 
\begin{align*}
\begin{split}
||B_\delta v||_{s + \frac{\alpha}{2} ,s + \frac{\alpha}{2},0} \leq & C( \delta^{1/2}||E_\delta v||_{s +\frac{\alpha}{2},s+ \frac{\alpha}{2},0} + ||B_2 v||_{s+1-\frac{\alpha}{2},s+1-\frac{\alpha}{2},0}
\\& + ||\tilde{G}_2Pv||_{s-\frac{\alpha}{2}-m+2,s-\frac{\alpha}{2}-m+2,0} +  \tilde{C}||v||_{s-\frac{1}{2}+\frac{\alpha}{2} ,s-\frac{1}{2}+\alpha,r}  ).
\end{split}
\end{align*}
After an overall $\frac{\alpha}{2}$ shift of differential orders, we can rewrite this as:
\begin{align}
\begin{split}
||B_\delta v||_{s,s,0} \leq & C( \delta^{1/2}||E_\delta v||_{s,s,0} + ||B_2 v||_{s+1-\alpha,s+1-\alpha,0} 
\\& + ||\tilde{G}_2Pv||_{s-\alpha -m+2,s-\alpha-m+2,0}  +  \tilde{C} ||v||_{s-\frac{1}{2},s -\frac{1-\alpha}{2},r}).
\end{split}
\label{est6}
\end{align}
\end{proof}

\subsection{Propagation using \texorpdfstring{$H_p$-}{Hp-}flow}
\label{sec_step2.3}
The goal of this part is to control $||E_\delta v||_{s,s,0}$ by  $||B_\delta v||_{s,s,0}$ and get a `reversed' version of (\ref{est6}), using the propagation estimate of $H_p$ again. 
The crutial point is that this control is up to a $\delta^{-\beta}$-factor with $\beta<\frac{1}{2}$, which allows us to absorb this $||B_\delta v||_{s,s,0}$-term when we combine with \eqref{est6-prop}. This strategy is adapted from \cite[Section~3B2]{hintz2021normally}.

Take $\beta$ such that
\begin{align} \label{eq: beta condition}
\max\{0, \frac{\sup_\Gamma {\tb p}_1}{\nu_{\min}} \} < \beta < \frac{1}{2}, 
\end{align}
and this is possible because of \eqref{eq: skew p1 bound}. 

\begin{figure}[H]
\centering
\includegraphics[scale = 0.4]{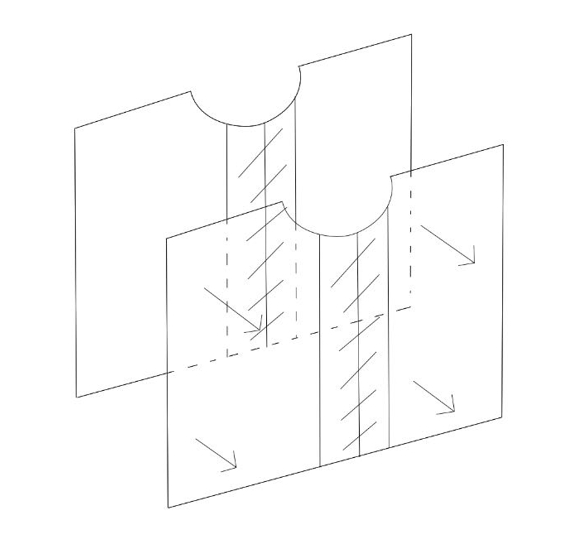}
\caption{The propagation in Proposition \ref{prop: E delta controlled by B delta}. Arrows indicate the direction from where we assume a priori regularities to where we conclude regularities.}
\label{fig_step2.3} 
\end{figure}

\begin{prop}  \label{prop: E delta controlled by B delta}
For $\beta$ satisfying \eqref{eq: beta condition} and $\delta$ sufficiently small, there exist $\tilde{B}_1,\tilde{G}_3 \in \Psi_{\mathrm{cu},\alpha}^{0,0,r}(M,\Gamma^u)$  with  $\WF_{\mathrm{cu},\alpha}'(\tilde{B}_1) \subset \hat{\mathcal{U}}_{7\epsilon_0}\backslash\hat{\mathcal{U}}_{\frac{\epsilon_0}{4}},\tilde{G}_3$ microlocalized near $\Gamma$, such that
\begin{align}
\begin{split}
||E_\delta v||_{s,s,0} \leq &C( (||B_\delta v||_{s,s,0} + || \tilde{B}_1 v||_{s,s,0}+ ||v||_{s-1/2,s-(1-\alpha)/2,r}) \\
&+ \sqrt{\tilde{C}}(||\tilde{G}_3Pv||_{s-m+1,s-m+1,0} + ||v||_{s-1/2,s-(1-\alpha)/2,r})),
\end{split}
\label{est2.3}
\end{align}
where $C$ is independent of $\delta$, $E_\delta,B_\delta$ are the same as in Proposition \ref{propstep2.2}.
\end{prop}
\begin{proof}
Consider the commutant
\begin{align*}
a = \check{a}^2, \quad \check{a} = \tau^{-r} \hat{\rho}^{-s+(m-1)/2}|\phi^s|^{-\beta} \psi_3(\log |\phi^s/\delta| )\chi_{3}^u(\phi^u/|\hat{\rho}|^\alpha) \chi_{3}^T(\tau) \chi_{3}^\Sigma(\textbf{p}),
\end{align*}
where $\delta \in (0,1)$ is typically a small parameter. The same as in the previous step, whenever companied with $\chi_3^u$ or $(\chi_3^u)'$, $\frac{\phi^u}{\hat{\rho}^\alpha}$ is effectively a symbol in $S_{\mathrm{cu},\alpha}^{0,0,r}(M,\Gamma^u)$. For the cut-off $\psi_3$, we can arrange
\begin{align*}
\psi_3 \psi_3' =  -\tilde{b}_l^2 + \tilde{e},
\end{align*}
where $\tilde{e}$ is supported in $\{ \delta \leq  |\phi^s| \leq \frac{9}{4} \delta\}$ and it satisfies $|\tilde{e}| \leq 1$. $\tilde{b}_l \geq 0$, and it is supported in  $\{\frac{3}{2}\delta \leq |\phi^s| < 6\epsilon_0 \delta^{-1}\}$. And according to the support of $\tilde{b}_l$ and $\tilde{e}$, we use a partition $\psi_3^2= (\psi_{3-})^2 + (\psi_{3+})^2$, where $\psi_{3+} = \psi_3$ on $\{\frac{9}{4}\delta \leq |\phi^s| < 6\epsilon_0 \delta^{-1}\}$, $\mathrm{supp} \psi_{3-} \subset [0,\log \frac{9}{4})$, $0 \leq \psi_{3\pm} \leq 1$. Next we consider the pairing
\begin{align} \label{eq: paring, step 2.3}
\mathrm{Im}\la Pv,Av \ra = \la (\frac{i}{2}[P,A]+\frac{P-P^*}{2i}A)v,v\ra.
\end{align}
For the left hand side, we have
\begin{align*}
\mathrm{Im}\la Pv,Av \ra \leq | \mathrm{Im}\la Pv,Av \ra | \leq |\la Pv,Av \ra |.
\end{align*}
For the right hand side, we take $c$ such that
\begin{align}   \label{eq: c condition, step 2.3}
0< c < \beta \nu_{\min} - \sup_{\Gamma}{\tb p}_1,
\end{align}
which exists due to \eqref{eq: beta condition}, and write its principal symbol as
\begin{align}
\begin{split}
\check{a}H_p \check{a} + \hat{\rho}^{-m+1}\textbf{p}_1 \check{a}^2 = & -c \hat{\rho}^{-m+1}\check{a}^2 - ( \hat{\rho}^{-s} b_{l,-})^2 - (\hat{\rho}^{-s}b_{l,+})^2+f_1 \\&+f_2+hp, 
\end{split}
\label{sym4}
\end{align}
where
\begin{align}
\begin{split}
b_{l,\pm}  =& \tau^{-r}|\phi^s|^{-\beta} \psi_3^\pm \chi_{3}^u \chi_{3}^T \chi_{3}^\Sigma (\beta w^s - \textbf{p}_1 - c + r\tau (\tau^{-2}\textbf{H}_p \tau) 
\\& + \tau^{-r}(s- (m-1)/2)(\hat{\rho}^{-1} \textbf{H}_p \hat{\rho}) + w^s \tilde{b}_l^2 )^{1/2}\\
f_ 2  = & \tau^{-2r} \hat{\rho}^{-2s} |\phi^s|^{-2\beta} w^s \tilde{e} (\chi_{3}^u)^2 (\chi_{3}^T)^2 (\chi_{3}^\Sigma)^2,\\
h =& \tau^{-2r} \hat{\rho}^{-2s+m} |\phi^s|^{-2\beta} \psi_3^2 (\chi_{3}^u)^2 (\chi_{3}^T)^2 ( -\beta \textbf{r}^s + \frac{\psi_3'\textbf{r}^s}{\psi_3\phi^s} + m(\hat{\rho}^{-1}\textbf{H}_p\hat{\rho} \chi_{3}^\Sigma (\chi_{3}^\Sigma)') ).\\
f_1 = & \tau^{-2r}\hat{\rho}^{-2s}|\phi_3^s|^{-2\beta}\psi_3^2(\chi_{3}^\Sigma)^2 ( (-(w^u\phi^u + r^u \textbf{p})|\hat{\rho}|^{-\alpha} - \alpha \phi^u |\hat{\rho}|^{-\alpha-1}\textbf{H}_p\hat{\rho} ) 
\\& (\chi_{3}^u)'\chi_{3}^u (\chi_{3}^T)^2 + ( \tau^{-2}\textbf{H}_p \tau)\tau^2 (\chi_{3}^u)^2 \chi_{3}^T(\chi_{3}^T)').
\end{split}
\end{align}
The expression that we are taking square root of in $b_{l,\pm}$ is positive for $\epsilon_0,\delta$ sufficiently small by \eqref{eq: c condition, step 2.3}.
In the calculation of $f_1$, we used
\begin{align*}
\textbf{H}_p(\phi^u/|\hat{\rho}|^\alpha) = -(w^u\phi^u + r^u \textbf{p})|\hat{\rho}|^{-\alpha} - \alpha \phi^u |\hat{\rho}|^{-\alpha-1}\textbf{H}_p\hat{\rho}.
\end{align*}

$b_{l,+}$ is the main term, giving control over the region away from $\Gamma^s$. $f_i$ are `error terms'. On the support of $f_1$, which is away from $\bar{\Gamma}^u$, we have a priori control (because when $\tau$ and $\phi^u/|\hat{\rho}|^\alpha$ are close enough to 0, $(\chi^u)'$ and $\chi_T'$ will vanish, hence $f_1$ is supported away from these regions). Recall that $\supp \tilde{e}(\cdot) \subset [0,\log(\frac{9}{4}))$, hence $\supp f_2 \subset  \{\delta \leq |\phi^s| \leq \frac{9}{4}\delta\}$. $b_{l,+}$ and $f_2$  satisfy bounds:
\begin{align}
\begin{split}
& \tau^rb_{l,+} \geq c' \quad \text{ on } \quad \hat{\mathcal{U}}_{3\epsilon_0} \cap \{ \frac{9}{4}\delta \leq |\phi|^s \leq 5\epsilon_0\},     \\
& |f_2| \lesssim \tau^{-2r} \hat{\rho}^{-2s} \delta^{-2\beta}.
\end{split}
\label{sym5}
\end{align}

Quantizing both sides of (\ref{sym4}), applying them to $v$ and pairing with $v$, we have: 
\begin{align}
\begin{split}
& c||\check{A}v||^2_{(m-1)/2,(m-1)/2,0} + ||B_{l,+}v||_{s,s,0}^2 + ||B_{l,-}||_{s,s,0}^2\\
\leq & |\la F_1 v,v \ra| + |\la F_2 v,v \ra| + |\la Pv, Hv \ra| + |\la \check{A} Pv, \check{A} v \ra| + 
\tilde{C} ||v||_{s-1/2,s-\frac{1-\alpha}{2},r}\\
\leq & |\la F_1 v,v \ra| + |\la F_2 v,v \ra| + ||\tilde{G}_3Pv||_{s-m+1,s-m+1,0}^2 + \tilde{C} ||\tilde{G}_3v||_{s-1,s-1,0}^2
\\&+c||\check{A}v||^2_{(m-1)/2,(m-1)/2,0} + \frac{1}{2c}||\check{A}Pv||_{-(m-1)/2,-(m-1)/2,0}^2 
\\&+ \tilde{C}||v||_{s-1/2,s-\frac{1-\alpha}{2},r}^2,
\end{split}  \label{est2}
\end{align}
where $\tilde{G}_3\in \Psi_{\mathrm{cu},\alpha}^{0,0,r}(M,\Gamma^u)$ is the microlocalizer that is elliptic near $\supp \check{a}$ with $\WF_{\mathrm{cu},\alpha}'(\tilde{G}_3)$ contained in a neighborhood of $\supp\check{a}$. $\tilde{C} ||v||_{s-1/2,s-\frac{1-\alpha}{2},r}$ arises because (\ref{sym4}) concerns only principal symbols. 
We apply Lemma \ref{lmm_improve_order} when we count orders and the $a_2$-term in that Lemma is absorbed into $f_1$-term (and $F_1$-term after quantization) here. Due to the $\psi_3$ factor, $|\phi^s|^{-\beta} \lesssim \delta^{-\beta}$ on $\supp \check{a}$ and this estimate gives an upper bound of $\check{a}$ as well. 

The $\check{A}v$-terms on both sides cancel each other. The $\tilde{G}_3v$-term and $\tilde{G}_3Pv$-term are introduced to control $|\la Pv, Hv \ra|$. Use $\tilde{B}_1 \in \Psi_{\mathrm{cu},\alpha}^{0,0,r}(M,\Gamma^u)$ which is elliptic on $\hat{\mathcal{U}}_{6\epsilon_0}\backslash \hat{\mathcal{U}}_{\frac{\epsilon_0}{2}}$ and $\WF_{\mathrm{cu},\alpha}'(\tilde{B}_1) \subset \hat{\mathcal{U}}_{7\epsilon_0}\backslash\hat{\mathcal{U}}_{\frac{\epsilon_0}{4}}$ to control the errors slightly away from $\bar{\Gamma}^u$. As we have mentioned, the main term being controlled is $B_{l,+}$-term, which is elliptic on a region that is near $\bar{\Gamma}^u$ but away from $\Gamma^s$. By our construction, $\WF_{\mathrm{cu},\alpha}'(E_\delta) \subset \mathrm{Ell}(B_\delta)$. 
For $F_2$-term, since $\hat{\mathcal{U}}_{\epsilon_0}\cap \{|\phi|^s \leq 3\delta\} \subset \mathrm{Ell}(B_\delta)$ and $\hat{\mathcal{U}}_{6\epsilon_0}\backslash \hat{\mathcal{U}}_{\frac{\epsilon_0}{2}} \subset \mathrm{Ell}(\tilde{B}_1)$, we know $\WF_{\mathrm{cu},\alpha}'(F_2) \subset \mathrm{Ell}(B_\delta) \cup \mathrm{Ell}(\tilde{B}_1)$.  Consequently, we can control $|\la F_2 v,v \ra | $ by $\delta^{-2\beta}(||B_\delta v||_{s,s,0}^2+||\tilde{B}_1v||_{s,s,0}+ ||v||_{s-1/2,s-(1-\alpha)/2,r})$. Substitute this into (\ref{est2}) and take square root on both sides, and we have:
\begin{align*}
\begin{split}
||E_\delta v||_{s,s,0} \leq &C( \delta^{-\beta}(||B_\delta v||_{s,s,0} + || \tilde{B}_1 v||_{s,s,0}+ ||v||_{s-1/2,s-(1-\alpha)/2,r}) \\
&+ \sqrt{\tilde{C}}(||\tilde{G}_3Pv||_{s-m+1,s-m+1,0} + ||v||_{s-1/2,s-(1-\alpha)/2,r})),
\end{split}
\end{align*}
where $C$ is independent of $\delta$.
\end{proof}
\subsection{Combining Estimates}
\label{sec_Combine}
Combine (\ref{est2.3}) with (\ref{est6}), we obtain an estimate with leading term $\delta^{1/2-\beta}||B_\delta v||_{s,s}$. Since $\beta<\frac{1}{2}$ (recalling \eqref{eq: beta condition}), by choosing $\delta$ small enough, we can absorb this term into the left hand side and get:
\begin{align}
\begin{split}
||B_\delta v||_{s,s,0} \lesssim &||\tilde{B}_1v||_{s+1-\alpha,s+1-\alpha,0} + ||\tilde{G}_2Pv||_{s-m+2-\alpha, s-m+2-\alpha,0} 
\\& + ||v||_{s-1/2,s-(1-\alpha)/2,r}, 
\end{split}
\label{est8}
\end{align}
with following properties:  $B_\delta$ is elliptic on (the lift of) $\Gamma$, $\WF_{\mathrm{cu},\alpha}'(\tilde{B}_1) \subset \hat{\mathcal{U}}_{7\epsilon_0}\backslash\hat{\mathcal{U}}_{\frac{\epsilon_0}{4}}$. In addition, we enlarge $\Ell(\tilde{G}_2)$ and $\WF_{\mathrm{cu},\alpha}'(\tilde{G}_2)$ to absorb the $\tilde{G}_3$-term. Now combine (\ref{est1_prop}) and (\ref{est8}), and let $\tilde{B}_1$ in (\ref{est8}) play the role of $B_0$ in (\ref{est1_prop}). Since the order now is $s+1-\alpha$ instead of $s+\lambda\alpha-\alpha$, so the order needs to be shifted by $1-\lambda \alpha$. For $\tilde{B}_1$-term, we localize inside the front face by inserting another cutoff, i.e., we use
\begin{align*}
||\tilde{B}_1  v||_{s+1-\lambda\alpha,s+1-\alpha,0}  \lesssim & || B_1 v ||_{s+1-\lambda \alpha,-\infty,0} + ||\tilde{G} Pv||_{s-m+2-\lambda\alpha,s-m+2 -\alpha,0}  \\
&+ ||\tilde{G}v||_{s-\lambda \alpha,s,0} + ||v||_{s-\frac{1}{2} -\lambda \alpha, s-\frac{1}{2} +\alpha,r} ,
\end{align*}
where these operators satisfy: $\mathcal{\hat{\mathcal{U}}}_{6\epsilon_0}\backslash \mathcal{\hat{\mathcal{U}}}_{\epsilon_0/2} \subset \Ell(\tilde{B}_1) \subset \WF_{\mathrm{cu},\alpha}'(\tilde{B_1}) \subset \mathcal{\hat{\mathcal{U}}}_{7\epsilon_0}\backslash \mathcal{\hat{\mathcal{U}}}_{\epsilon_0/4}$, $\WF_{\mathrm{cu},\alpha}'(\tilde{G}) \subset \{ |\frac{\rho}{\phi^u}| \leq 3C_1 \} \cap \mathcal{U}_{3\eta_1} = \{|\frac{\phi^u}{\rho}| \geq (3C_1)^{-1} \} \cap \mathcal{U}_{3\eta_1}$, and $\WF_{\mathrm{cu},\alpha}'(B_1) \cap \bar{\mathcal{U}}_{\eta_1} = \emptyset$. 
In particular, if we choose $C_1,\eta_1,\epsilon_0$ so that $(3C_1)^{-1}>\frac{\epsilon_0}{2},3\eta_1<6\epsilon_0$, then $\WF_{\mathrm{cu},\alpha}'(\tilde{G})\subset \Ell(\tilde{B}_1)$ and we can iterate to improve this error term. Concretely, apply the same estimate to $\tilde{G}v$ with $s$ replaced by $s-(1-\alpha)$, and then repeat. The only cost this iteration might cause is the microlocal error introduced when we apply elliptic estimate to $\tilde{B}_1$, the microlocal error and $\tilde{G}v$-term with one order lower norm can both be absorbed into the last error term $||v||_{s-\frac{1}{2} -\lambda \alpha , s-\frac{1}{2} + \alpha,r}$:
\begin{align*}
||\tilde{B}_1  v||_{s+1-\lambda\alpha,s+1-\alpha,0} \lesssim & || B_1 v ||_{s+1-\lambda \alpha,-\infty,0} + ||\tilde{G} Pv||_{s-m+2-\lambda\alpha,s-m+2 -\alpha,0}  \\
&+ ||v||_{s-\frac{1}{2}-\lambda \alpha, s-\frac{1}{2}+\alpha,r} .
\end{align*}
Substitute this estimate into (\ref{est8}) and we get
\begin{align*}
||B_\delta v||_{s,s,0} \lesssim &||B_1v||_{s+1-\lambda \alpha,-\infty,0} + ||\tilde{G}Pv||_{s-m+2-\lambda\alpha,s-m+2-\alpha,0} \\& + ||\tilde{G}_3Pv||_{s-m+2-\alpha ,s-m+2 - \alpha,0}+||v||_{s-\frac{1}{2}, s-\frac{1}{2}+\alpha,r}.
\end{align*}
$\tilde{G}_3$ is microlocalized in a neighborhood of $\Gamma$ but $\WF_{\mathrm{cu},\alpha}'(\tilde{G}_3) \cap \Gamma^s = \emptyset$. $\tilde{G}_3Pv$ and $\tilde{G}Pv$-terms can be combined together using a $G_0$ obtained by enlarging their wavefront set. And then we iterate to improve the last error term to obtain
\begin{align*}
\begin{split}
||B_\delta v||_{s,s,0} \lesssim &||B_1v||_{s+1-\lambda \alpha, s,0} + ||G_0Pv||_{s-m+2-\lambda\alpha,s-m+2-\alpha,0}  + ||v||_{-N,-N,r},  
\end{split}  
\end{align*}
where $B_\delta$ is elliptic on  $\Gamma$, and $\WF_{\mathrm{cu},\alpha}'(B_1) \cap \bar{\Gamma}^u =\emptyset$, and $\lambda$ satisfies (\ref{eq: lambda condition}). 

\begin{rmk}  \label{rmk: make Ell(B) larger}
Notice that the estimate keeps to hold when we add intermediate terms in the proof to the left hand side. Concretely, in terms of notations in this chapter, $||\tilde{B}_1  v||_{s+1-\lambda\alpha,s+1-\alpha,0}$, $||B_1v||_{s+1-\lambda \alpha, s,0}$. Equivalently we have
\begin{align}
\begin{split}
||B v||_{s,s,0} \lesssim &||B_1v||_{s+1-\lambda \alpha, s,0} + ||G_0Pv||_{s-m+2-\lambda\alpha,s-m+2-\alpha,0}  + ||v||_{-N,-N,r},  \label{est_thm}
\end{split}  
\end{align}
where $B \in \Psi_{\mathrm{cu},\alpha}^{0,0,r}(M,\Gamma^u)$ is elliptic on both $\hat{\mathcal{U}}_{\epsilon_0}$ and the front face.
\end{rmk}

\subsection{Regularization}
\label{sec_regularization}
Only assuming that the right hand side of (\ref{main}) is finite is not sufficient to guarantee that each pairing in our positive commutator is finite and integrations by parts are legal. Potentially some terms in equations (e.g., (\ref{10}) and (\ref{pair1})) are not finite with only $(-N,-N)$ order priori control of $v$. In this section we justify pairing and integration by parts in our positive commutator argument by a regularization argument for $-N=s-\frac{1}{2}$ first and then for general $N$ by induction. Starting with Proposition \ref{step1}, we replace $\check{a}$ by:
\begin{align}
\check{a}_{\eta_r} :=      \mk{t}_{\eta_r}^2 \check{a},            \label{reg_sym_1}
\end{align}
where $\mk{t}_{\eta_r}=(1+\eta_r \hat{\rho}^{-1})^{-1},\eta_r>0$. In (\ref{sym1}), this new $\mk{t}_{\eta_r}^2$ factor introduces an extra term given by the $H_p$-derivative falling on $\mk{t}_{\eta_r}$. Direct computation shows
\begin{align} \label{eq: Hp regularizer}
H_p(\mk{t}_{\eta_r}^4) = 4\mk{t}_{\eta_r}^3H_p\mk{t}_{\eta_r} = 4\frac{\eta_r}{\eta_r+\hat{\rho}}\mk{t}_{\eta_r}^4(\hat{\rho}^{-1}H_p\hat{\rho}).
\end{align}

Now since $|\frac{\eta_r}{\eta_r+\hat{\rho}}|\leq 1$, $\hat{\rho}^{-1}H_p\hat{\rho} \rightarrow 0$ as $\tau \rightarrow 0$, hence this term can be made small when we localize near $\tau=0$. Consequently, this term can be absorbed into the $b_0$-term (notice that all terms have an extra $\mk{t}_{\eta_r}^4$-factor now). For $\eta_r>0$, we know that $a_{\eta_r}=\check{a}_{\eta_r}^2$ is a symbol of 4 order less in both the indices associated with the fiber infinity and the front face compared with $a=\check{a}^2$. To be concrete, $a \in S_{\mathrm{cu},\alpha}^{2s+2-(m-1),2s+2-(m-1)+2\lambda\alpha,2r}(M,\Gamma^u),a_{\eta_r} \in S_{\mathrm{cu},\alpha}^{2s-m-1,2s-m-1+2\lambda\alpha,2r}(M,\Gamma^u)$. Thus assuming $-N=s-\frac{1}{2}$ regularity of $v$ in a priori, all pairing and integration by parts are justified and we obtain an estimate similar to (\ref{est1}), but with $B_0,B^{ff},...$ replaced by $B_{0,\eta_r},B^{ff}_{\eta_r},...$, which are obtained by quantizing the symbol of the same lowercase letter (without $\eta_r$) with an extra $\mk{t}_{\eta_r}^2$ factor (except for, as aforementioned, $b_{0,\eta_r}$ is also used to absorb the term introduced by $H_p(\mk{t}_{\eta_r}^4)$). Finally we let $\eta_r \rightarrow 0$ and apply the weak-* compactness argument, see for example Section 5.4.4 of \cite{vasy2018minicourse}, we obtain the estimate above (\ref{est1}) without $\eta_r$ and $B_0v \in H_{\mathrm{cu},\alpha}^{s+1,s+1+\lambda \alpha-\alpha,r}(M,\Gamma^u)$. The regularization arguments for step Section \ref{sec_step2.1}-\ref{sec_step2.3} are similar. We conduct the argument for in Section \ref{sec_step2.1} individually but regularize the argument in Section \ref{sec_step2.2}, \ref{sec_step2.3} and \ref{sec_Combine} together to obtain (\ref{est_thm}) by sending the regularization parameter to 0. And then the general $N>-(s-\frac{1}{2})$ case follows by induction. For $N<-(s-\frac{1}{2})$, the estimate holds automatically by our initial case $N=-(s-\frac{1}{2})$.

\subsection{Relaxing the regularity requirements}
\label{sec: relaxing regularity}
Now we describe modifications needed in the proof when assumption \eqref{assumption6} is replaced by assumption \eqref{assumption6'}. 
We choose the same commutants in each step, but now since $P$ takes the form \eqref{eq: P, with uniform PsiDO perturbation}, operators (resp symbols) appeared positive commutator arguments above will have an extra error term in $\tau^{\beta_T}\Psi_{\infty,\alpha}^{*,*,*}$ (resp. $\tau^{\beta_T}S_{\infty,\alpha}^{*,*,*}$). For an operator $A = A_0+\tilde{A}$, $A_0 \in \Psi_{\cu,\alpha}^{m,\tilde{m},r},\tilde{A} \in \tau^{\beta_T}\Psi_{\infty,\alpha}^{m,\tilde{m},r}$, then we define the elliptic set of $A$ at $\tau=0$ to be $\Ell(A_0) \cap \{\tau = 0\}$, then the proof of microlocal elliptic estimates (inverting the principal symbol of $A$ up to a lower order error) holds by the same proof, which means operators with positive symbols give microlocal control of weighted Sobolev norms just as in the smooth setting.


\section{Application to Kerr(-de Sitter) spacetimes and its perturbations}
\label{application}
\subsection{Kerr(-de Sitter) spacetime and its perturbations}
\label{kds_intro}
In this section, we consider Kerr(-de Sitter) spacetimes parameterized by the black hole mass $\mathfrak{m}$ and the angular momentum $\mathfrak{a} \in \R$ and its perturbations. We assume the black hole is subextremal in the sense that
\begin{align*}
\Delta(r) = (r^2+\mk{a}^2)(1-\frac{\Lambda r^2}{3})-2\mk{m}r
\end{align*}
has four distinct real roots
\begin{align*}
r_-<r_C<r_e<r_c.
\end{align*}
We point out here that some authors use the condition $|\mathfrak{a}|<\mathfrak{m}$ to define the subextremal property, which is slightly stronger than the distinct root condition. See \cite{sarp2011kds} for more details.

Recall that in the Boyer-Lindquist coordinates, the Kerr(-de Sitter) spacetime is given by (\ref{kds_manifold}). 
We use $\varphi \in \mathbb{S}^1,\theta \in [0,\pi]$ as spherical coordinates on $\mathbb{S}^2$. And $M^\circ$ is equipped with the metric $g_{\mk{m},\mk{a}}$ given by (\ref{kds_metric}).

As discussed after (\ref{Mdef}), we use $M= (M^\circ \sqcup ([0,\infty)_\tau \times X))/\sim$ to denote the spacetime that is compactified at the time infinity. Since (\ref{kds_metric}) is independent of $t$, it naturally extends to a metric on $M$. In order to distinguish variable names from those used on our model $\tilde{M}$ in following sections, variables on $T^*M$ are denoted as $(x_M,\xi_M)$. The singularity of (\ref{kds_metric}) at horizons $\{r=r_e,r_c\}$ can be resolved by a change of coordinates, see \cite{sarp2011kds} for more detailed discussion. In our coordinate system, the dual metric is given by
\begin{align}
\begin{split}
G_{\mk{m},\mk{a}} &= G_r+G_\theta,\\
(r^2+\mk{a}^2\cos^2\theta)G_r& = \Delta(r)\xi_r^2-\frac{(1+\frac{\Lambda \mk{a}^2}{3})^2}{\Delta(r)}( (r^2+\mk{a}^2)\xi_t+\mk{a}\xi_\varphi )^2\\
(r^2+\mk{a}^2\cos^2\theta)G_\theta &=\Delta_\theta \xi_\theta^2+\frac{(1+\frac{\Lambda \mk{a}^2}{3})^2}{\Delta_\theta \sin^2\theta}(\mk{a}\sin^2\theta \xi_t+\xi_\varphi)^2.
\end{split} \label{defn_G}
\end{align}
The singularities of $g$ and $G$ at $\{\theta=0,\pi\}$ can be resolved by coordinate change, we refer to \cite[Section~3.1]{dy15} for more detailed discussion. 
Applying the analytic framework we developed, we take
\begin{align}
p_{\mk{m},\mk{a}}=(r^2+\mk{a}^2\cos^2\theta)G_{\mk{m},\mk{a}}.  \label{psymbol}
\end{align}
And other notations in propagation estimates are inherited as well.

\subsection{Defining functions of the unstable and stable manifolds}
In this section, we characterize the trapping phenomena in exact Kerr(-de Sitter) spacetimes using the combination of \cite[Theorem~3.2]{petersen2024wave} (for Kerr-de Sitter case) and \cite[Proposition~3.5]{dy15} (for Kerr case) restated as Proposition \ref{prop: trappedset}. 
\begin{prop} \label{prop: trappedset}
For $(\xi_t,\xi_\varphi) \in \R^2 \backslash \{(0,0)\}$, define:
\begin{align*}
F_{\xi_t,\xi_\varphi}(r):= \frac{1}{\Delta(r)}( (r^2+\mk{a}^2)\xi_t+\mk{a}\xi_\varphi )^2.
\end{align*}
\begin{enumerate}
\item Then either:
\begin{itemize}
	\item $F_{\xi_t,\xi_\varphi}$ vanishes at $r_e$ or $r_c$ and has no critical point in $(r_e,r_c)$. 
	\item $F_{\xi_t,\xi_\varphi}$ has exactly one critical point $r_{\xi_t,\xi_\varphi} \in (r_e,r_c)$ and $F_{\xi_t,\xi_\varphi}''(r_{\xi_t,\xi_\varphi})>0$.\\ 
\end{itemize}
\item $F_{\xi_t,\xi_\varphi}>0$ on $\Sigma$.\\
\item The trapped set in ${M}$ is:
\begin{align}  \label{eq: Gamma 0}
\Gamma_0 := \bigcup_{ (\xi_t,\xi_\varphi) \in \R^2 \backslash \{(0,0)\} }\Gamma_{\xi_t,\xi_\varphi},
\end{align}
where $\Gamma_{\xi_t,\xi_\varphi}:=\{ \xi_r=r-r_{\xi_t,\xi_\varphi}=p_{\mk{m},\mk{a}}=0 \}$.
\item $\Gamma_0$ is a smooth connected 5-dimensional submanifold of $T^*{M}$ with defining function $\xi_r,r-r_{\xi_t,\xi_\varphi},p_{\mk{m},\mk{a}}$.
\item The trapping of the flow of $\frac{1}{H_{p_{\mk{m},\mk{a}}}t}H_{p_{\mk{m},\mk{a}}}$ in any subextremal Kerr(-de Sitter) spacetime is eventually absolutely r-normally hyperbolic for every r in the sense of \cite{wunsch2011resolvent}. The unstable(u) and stable(s) manifolds are smooth manifold given by:
\begin{align} \label{eq: definition, Gamma u/s 0}
\Gamma^{u/s}_0:= \{{\varphi}^{u/s} =0 \} \cap \Sigma,
\end{align}
where ${\varphi}^{u/s}=\xi_r \mp  \sgn(r-r_{\xi_t,\xi_\varphi})(1+\hat{\alpha})\sqrt{\frac{F_{\xi_t,\xi_\varphi}(r)-F_{\xi_t,\xi_\varphi}(r_{\xi_t,\xi_\varphi})}{\Delta(r)}}$ and $\hat{\alpha}=\frac{\Lambda \mk{a}^2}{3}$.
\end{enumerate}
\begin{rmk}
Coupling with $\sgn(r-r_{\xi_t,\xi_\varphi})$, the square root function in the definition of $\Gamma^{u/s}$ is smooth because at $r_{\xi_t,\xi_\varphi}$, $F_{\xi_t,\xi_\varphi}'(r)=0$ and $F_{\xi_t,\xi_\varphi}(r)-F_{\xi_t,\xi_\varphi}(r_{\xi_t,\xi_\varphi})$ vanishes quadratically at $r_{\xi_t,\xi_\varphi}$.  
\end{rmk}
\label{prop_trap}
\end{prop}
Next we verify that the rescaled version of $\varphi^{u/s}$ above are exactly the defining functions that characterize the normally hyperbolic trapping properties after rescaling them to be homogeneous degree $0$: 
\begin{align}
\label{eq: defn hat varphiu/s}
\hat{\varphi}^u:=(\xi_t^2+\xi_{\varphi}^2)^{-1/2}\varphi^u,\, \hat{\varphi}^s:=(\xi_t^2+\xi_{\varphi}^2)^{-1/2}\varphi^s.
\end{align}
Recalling the expressions in \eqref{defn_G}, $(\xi_t^2+\xi_{\varphi}^2)^{1/2}$ dominates all momentum variables near the characteristic set. 

\begin{prop}
$\hat{\varphi}^{u/s}$ defined in \eqref{eq: defn hat varphiu/s} satisfy assumptions in Section \ref{sec_assumptions} with $\phi^{u/s}$ replaced by $\hat{\varphi}^{u/s}$.
Concretely, they satisfy \eqref{hpphi} and their restrictions to $\tau=0$ satisfy \eqref{eq: barphi u,s assumptions1}\eqref{eq: barphi u,s assumptions2}.
\label{prop_hpphi}
\end{prop}
\begin{proof} We verify the property in \eqref{hpphi} first and this implies the statement about \eqref{eq: barphi u,s assumptions1} directly. Since $\partial_t p_{\mk{m},\mk{a}} = \partial_\varphi p_{\mk{m},\mk{a}} = 0$, we know $\textbf{H}_{p_{\mk{m},\mk{a}}} \xi_t=\textbf{H}_{p_{\mk{m},\mk{a}}} \xi_\varphi =0$, thus any factor as a function of $\xi_t,\xi_\varphi$ commutes with $\textbf{H}_{p_{\mk{m},\mk{a}}}$ and we consider $\varphi^{u/s}$ instead, i.e., ignore the $\xi_t$ power in front of $\varphi^u$. Recalling that
\begin{align*}
p_{\mk{m},\mk{a}}(x,\xi) = & (r^2+\mk{a}^2 \cos^2\theta)G_{\mk{m},\mk{a}}(x,\xi)\\
         = & \Delta(r)\xi_r^2+(1+\hat{\alpha} \cos^2\theta)\xi_\theta^2+\frac{(1+\hat{\alpha})^2}{(1+\hat{\alpha}\cos^2\theta)\sin^2\theta}(\mk{a}\xi_t\sin^2\theta+\xi_{\varphi})^2 \\
         &-\frac{(1+\hat{\alpha})^2}{\Delta(r)}( (r^2+\mk{a}^2)\xi_t+\mk{a}\xi_\varphi)^2,
\end{align*}
thus $\partial_tp_{\mk{m},\mk{a}}=\partial_{\varphi}p_{\mk{m},\mk{a}}=0$ and
\begin{align*}
&H_{p_{\mk{m},\mk{a}}}\xi_t=H_{p_{\mk{m},\mk{a}}}\xi_\varphi=H_{p_{\mk{m},\mk{a}}}r_{\xi_t,\xi_\varphi}=0,\\
& H_{p_{\mk{m},\mk{a}}} \varphi^{u/s} = H_{p_{\mk{m},\mk{a}}}\xi_r \mp \sgn(r-r_{\xi_{\xi_t,\xi_\varphi}})(1+\hat{\alpha})H_{p_{\mk{m},\mk{a}}}(\sqrt{\frac{F_{\xi_t,\xi_\varphi}(r)-F_{\xi_t,\xi_\varphi}(r_{\xi_t,\xi_\varphi})}{\Delta(r)}}),\\
& H_{p_{\mk{m},\mk{a}}}r=-H_rp_{\mk{m},\mk{a}}=\partial_{\xi_r}p_{\mk{m},\mk{a}}=2\Delta(r)\xi_r,\\
& H_{p_{\mk{m},\mk{a}}}\xi_r= -H_{\xi_r}p_{\mk{m},\mk{a}}=-\partial_r p_{\mk{m},\mk{a}} =-\Delta'(r)\xi_r^2+(1+\hat{\alpha})^2F_{\xi_t,\xi_\varphi}'(r),\\
&H_{p_{\mk{m},\mk{a}}}(\sqrt{\frac{F_{\xi_t,\xi_\varphi}(r)-F_{\xi_t,\xi_\varphi}(r_{\xi_t,\xi_\varphi})}{\Delta(r)}}) 
\\ =&  \frac{1}{2} ( F'_{\xi_t,\xi_\varphi}(r)  ( ( F_{\xi_t,\xi_\varphi}(r)-F_{\xi_t,\xi_\varphi}(r_{\xi_t,\xi_\varphi}) )\Delta(r) )^{-1/2} \\
 & - ( F_{\xi_t,\xi_\varphi}(r)-F_{\xi_t,\xi_\varphi}(r_{\xi_t,\xi_\varphi}) )^{1/2} \Delta(r)^{-3/2}\Delta'(r)) H_{p_{\mk{m},\mk{a}}}r,  
\end{align*}

\begin{align*}
H_{p_{\mk{m},\mk{a}}} \varphi^{u/s} =& H_{p_{\mk{m},\mk{a}}}\xi_r \mp \sgn(r-r_{\xi_{\xi_t,\xi_\varphi}})(1+\hat{\alpha})H_{p_{\mk{m},\mk{a}}}(\sqrt{\frac{F_{\xi_t,\xi_\varphi}(r)-F_{\xi_t,\xi_\varphi}(r_{\xi_t,\xi_\varphi})}{\Delta(r)}})\\
=& -\Delta'(r)\xi_r^2+(1+\hat{\alpha})^2F'_{\xi_t,\xi_\varphi}(r)\mp \sgn(r-r_{\xi_{\xi_t,\xi_\varphi}})(1+\hat{\alpha}) 
\\& \times \frac{1}{2} ( F'_{\xi_t,\xi_\varphi}(r)  ( ( F_{\xi_t,\xi_\varphi}(r) -F_{\xi_t,\xi_\varphi}(r_{\xi_t,\xi_\varphi}) ) \Delta(r) )^{-1/2} \\&- ( F_{\xi_t,\xi_\varphi}(r)-F_{\xi_t,\xi_\varphi}(r_{\xi_t,\xi_\varphi}) )^{1/2} \Delta(r)^{-3/2}\Delta'(r)) H_{p_{\mk{m},\mk{a}}}r \\
= &  -\Delta'(r)\xi_r^2+(1+\hat{\alpha})^2F'_{\xi_t,\xi_\varphi}(r)\mp \sgn(r-r_{\xi_{\xi_t,\xi_\varphi}})(1+\hat{\alpha}) 
\\& \times ( F'_{\xi_t,\xi_\varphi}(r)  ( ( F_{\xi_t,\xi_\varphi}(r)-F_{\xi_t,\xi_\varphi}(r_{\xi_t,\xi_\varphi}) ) \Delta(r) )^{-1/2} \\&- ( F_{\xi_t,\xi_\varphi}(r)-F_{\xi_t,\xi_\varphi}(r_{\xi_t,\xi_\varphi}) )^{1/2} \Delta(r)^{-3/2}\Delta'(r)) \Delta(r)\xi_r \\
= &  -\Delta'(r)\xi_r^2+(1+\hat{\alpha})^2F'_{\xi_t,\xi_\varphi}(r)\mp \sgn(r-r_{\xi_{\xi_t,\xi_\varphi}})(1+\hat{\alpha}) 
\\& ( F'_{\xi_t,\xi_\varphi}(r)   ( F_{\xi_t,\xi_\varphi}(r) -F_{\xi_t,\xi_\varphi}(r_{\xi_t,\xi_\varphi}))^{-1/2}
 \Delta^{1/2}(r) 
 \\&- ( F_{\xi_t,\xi_\varphi}(r)-F_{\xi_t,\xi_\varphi}(r_{\xi_t,\xi_\varphi}) )^{1/2} \Delta(r)^{-1/2}\Delta'(r))\xi_r \\
=&(\xi_r\mp (1+\hat{\alpha}) S_{\xi_t,\xi_\varphi}(r)) \times (-\Delta'(r)\xi_r\mp \frac{(1+\hat{\alpha})F'_{\xi_t,\xi_\varphi}(r)}{S_{\xi_t,\xi_\varphi}(r)}).
\end{align*}
where
\begin{align} \label{eq: S ksit, ksiphi definition}
S_{\xi_t,\xi_\varphi}(r) = \sgn(r-r_{\xi_t,\xi_\varphi})\sqrt{\frac{F_{\xi_t,\xi_\varphi}(r)-F_{\xi_t,\xi_\varphi}(r_{\xi_t,\xi_\varphi})}{\Delta(r)}},
\end{align}
which is a monotonically increasing smooth function with inverse function $S_{\xi_t,\xi_\varphi}^{-1}$ when we restrict $r$ close enough to $r_{\xi_t,\xi_\varphi}$. Next we show that
\begin{align*}
\frac{(1+\hat{\alpha})F'_{\xi_t,\xi_\varphi}(r)}{S_{\xi_t,\xi_\varphi}(r)}
\end{align*}
is lower bounded. By the characterization of $F_{\xi_t,\xi_\varphi}$ in Proposition \ref{prop: trappedset}, we can set
\begin{align*}
& F_{\xi_t,\xi_\varphi}'(r) = (r-r_{\xi_t,\xi_\varphi})f_{\xi_t,\xi_\varphi}(r),
\end{align*}
where $c_f(\xi_t^2+\xi_\varphi^2) \leq f_{\xi_t,\xi_\varphi}(r) \leq C_f(\xi_t^2+\xi_\varphi^2)$ with $c_f,C_f>0$ when $r$ is close to $r_{\xi_t,\xi_\varphi}$. The upper bound follows fron the smoothness of $F_{\xi_t,\xi_\varphi}'$. On the other hand, $c_f$ is obtained by considering the homogeneous degree of $F_{\xi_t,\xi_\varphi}$ and hence we can restrict $(\xi_t,\xi_\varphi)$ to a sphere, which is compact. If such $c_f$ does not exist, then by compactness argument we can find $(\xi_t,\xi_\varphi)\in \mathbb{S}^2$ such that $f_{\xi_t,\xi_\varphi}(r_{\xi_t,\xi_\varphi})=0$, contradicting the simplicity of this critical point. Similarly for the argument about $\tilde{f}_{\xi_t,\xi_\varphi}$ below. Since $r_{\xi_t,\xi_\varphi}$ is a critical point of $F_{\xi_t,\xi_\varphi}$ with $F_{\xi_t,\xi_\varphi}''(r)>0$, then we can assume 
\begin{align*}
F_{\xi_t,\xi_\varphi}(r)-F_{\xi_t,\xi_\varphi}(r_{\xi_t,\xi_\varphi})=(r-r_{\xi_t,\xi_\varphi})^2\tilde{f}_{\xi_t,\xi_\varphi}(r),
\end{align*}
with $c_{\tilde{f}}(\xi_t^2+\xi_\varphi^2)\leq \tilde{f}_{\xi_t,\xi_\varphi}(r) \leq C_{\tilde{f}}(\xi_t^2+\xi_\varphi^2)$ near $r_{\xi_t,\xi_\varphi}$. Consequently
\begin{align}
S_{\xi_t,\xi_\varphi}(r)  = (r-r_{\xi_t,\xi_\varphi}) \tilde{f}_{\xi_t,\xi_\varphi}^{1/2}(r) \Delta(r)^{-1/2}.  \label{Sdecomp}
\end{align}
Thus 
\begin{align*}
\frac{(1+\hat{\alpha})F'_{\xi_t,\xi_\varphi}(r)}{S_{\xi_t,\xi_\varphi}(r)} = (1+\hat{\alpha})f_{\xi_t,\xi_\varphi}(r) \tilde{f}_{\xi_t,\xi_\varphi}^{-1/2}\Delta(r)^{1/2}
\end{align*}
is lower bounded by a positive constant multiple of $(\xi_t^2+\xi_\varphi^2)^{1/2}$ near $r=r_{\xi_t,\xi_\varphi}$. In particular, using the characterization of $\Gamma$ in Proposition \ref{prop: trappedset}, we can choose a neighborhood of $\Gamma$ on which $\xi_r$ is small. Thus the sign of $(-\Delta'(r)\xi_r\mp \frac{(1+\hat{\alpha})F'_{\xi_t,\xi_\varphi}(r)}{S_{\xi_t,\xi_\varphi}(r)})$ is $\mp$ for $\varphi^{u/s}$ respectively and  $(-\Delta'(r)\xi_r\mp \frac{(1+\hat{\alpha})F'_{\xi_t,\xi_\varphi}(r)}{S_{\xi_t,\xi_\varphi}(r)})$ is a multiple of $(\xi_t^2+\xi_\varphi^2)^{1/2}$ bounded away from $0$, which verifies (\ref{hpphi}) after rescaling by $(\xi_t^2+\xi_\varphi^2)^{-1/2}$.

Next we verify \eqref{eq: barphi u,s assumptions2}. Since $\varphi^{u/s}=\xi_r \mp (1+\hat{\alpha})S_{\xi_t,\xi_\varphi}(r)$, we only need to verify
\begin{align*}
(\xi_t^2+\xi_\varphi^2)^{-1/2} H_{\xi_r}S_{\xi_t,\xi_{\varphi}}(r) > 0
\end{align*}
near $\Gamma$. This is equivalent to 
\begin{align} \label{eq: partial r S}
\partial_r((\xi_t^2+\xi_\varphi^2)^{-1/2}S_{\xi_t,\xi_{\varphi}}(r))>0. 
\end{align}
This holds because $S_{\xi_t,\xi_{\varphi}}(r)$ vanishes simply (with respect to $r$) by Proposition \ref{prop: trappedset} and it has positive coefficient in front of $(r-r_{\xi_t,\xi_\varphi})$ if we Taylor expand it to the first order, thus \eqref{eq: partial r S} holds.

\end{proof}
\subsection{Dynamics under perturbation}
\label{sec: dynamics under perturbation}
Next we consider perturbations of the Kerr(-de Sitter) metric in the sense of \cite[Section~4]{hintz2021normally} and recall results therein. 
\begin{defn} \label{defn: asymptotic KdS}
Let $\mathcal{A}_{\rm cu}^{\beta_T}(M;S^2\,{}^{\rm cu}T^*M)$ be the class of cusp-conomral symmetric 2-tensors with decay order $\beta_T>0$ defined in Section \ref{sec: cusp-conormal symbols and functions}, 
then a metric $g$ is called an $\mathcal{A}_{\cu}^{\beta_T}$-asymptotically  Kerr(-de Sitter) metric if
\begin{align}
\label{asymptotic_kds_metric}
g= g_{\mk{m},\mk{a}}+\tilde{g},
\end{align}
where $g_{\mk{m},\mk{a}}$ is given by (\ref{kds_metric}) and $\tilde{g} \in \mathcal{A}_{\rm cu}^{\beta_T}(M;S^2\,{}^{\rm cu}T^*M)$. And we denote its dual metric function by $G$. We say $g$ is subextremal if $g_{\mk{m},\mk{a}}$ is so. This is well-defined since $g_{\mk{m},\mk{a}}$ in
(\ref{asymptotic_kds_metric}) is unique when $\beta_T>0$.
\end{defn}

Let $g$ be an $\mathcal{A}_{\cu}^{\beta_T}$-asymptotically Kerr(-de Sitter) metric defined in  Definition \ref{defn: asymptotic KdS}, we consider
\begin{align}
p = (r^2+\mk{a}^2\cos^2\theta)G    \label{psymbol_perturbed}
\end{align}
and $P \in \Psi_{\cu}^2(M)+\tau^{\beta_T}\Psi_{\infty}^2(M)$ with principal symbol $p$. In particular, $(r^2+\mk{a}^2\cos^2\theta)\Box_g$, which is a smooth (and uniformly lower and upper bounded) multiple of $\Box_g$ satisfies this.

The goal of the rest of this subsection is to recall the structure of the unstable/stable manifolds of the flow of $\mathsf{H}=(H_pt)^{-1}H_p$ given in \cite[Section~4]{hintz2021normally}. 
Notice that our $p$ differs with the dual metric function $G$ by a factor $(r^2+\mk{a}^2\cos \theta^2)$, which rescales the Hamilton flow by this factor when restricted to the characteristic set. Since this factor is smooth and bounded away from $0$ and from above, hence does not affect conclusions about the dynamics.

The characteristic set $p^{-1}(0) \subset {}^{\cu}T^*M$ is denoted by $\Sigma$ and that of the exact Kerr-(de Sitter) metric is denoted by $\Sigma_{\mk{m},\mk{a}}$.  Then we denote their restriction to the sphere bundle in the interior by 
\begin{align}
\mk{M}:=\Sigma \cap {}^{\cu}S^*M, \quad \mk{M}_0 =\Sigma_{\mk{m},\mk{a}} \cap {}^{\cu}S^*M.
\end{align}
We also denote $\mk{X}:=\Sigma_{\mk{m},\mk{a}} \cap {}^{\cu}S^*_{\partial M}M$, and we have $\mk{M}_0=[0,\infty)_\tau\times\mk{X}$. 
Then we define $\mathcal{A}_{\cu}(\mk{M}_0)$, the space of conormal functions over $\mk{M}_0$ to be 
\begin{align}\label{eq: defn, M0 cusp-conormal}
\mathcal{A}_{\cu}(\mk{M}_0) = \{ u \in L^\infty (\mk{M}_0): Au \in L^\infty(\mk{M}_0), \, \forall A \in \mathrm{Diff}_{\cu}(\mk{M}_0) \},
\end{align}
where $\mathrm{Diff}_{\cu}(\mk{M}_0)$ is the space of products of vector fields that are spanned over $\mathcal{C}^\infty(\mk{M}_0)$ by stationary extensions of smooth vector fields on $\mk{X}$ and $\tau^2\partial_\tau$.
Then we define the space of weighted cusp-conormal functions on $\mk{M}_0$ to be:
\begin{align} 
\mathcal{A}_{\cu}^{\beta_T}(\mk{M}_0) = \tau^{\beta_T} \mathcal{A}_{\cu}(\mk{M}_0).
\end{align}
For $\Gamma^{u/s}_0$ defined in \eqref{eq: definition, Gamma u/s 0}, we define $\mathcal{A}_{\cu}^{\beta_T}(\Gamma^{u/s}_0)$ in the same manner. 

For $\mk{q} \in \mk{X}$, let $\mk{U} \subset \mk{X}$ be an open neighborhood of it in $\mk{X}$, then let $\mk{U} \times (-1,1)$ be a tubular neighborhood of $\mk{U}$  in ${}^{\cu}S^*_{\partial M}M$. Extend this (by taking product with $[0,\infty)_\tau$) to a tubular neighborhood $([0,\infty)_\tau \times \mk{U}) \times (-1,1)$ of $[0,\infty)_\tau \times \mk{U} \subset \mk{M}_0$ in ${}^{\cu}S^*M$.

Finally we define graphs of functions in $\mathcal{A}_{\cu}^{\beta_T}(\mk{M}_0)$ (resp. $\mathcal{A}_{\cu}^{\beta_T}(\Gamma^{u/s}_0)$) as the set 
(only take the part near $\tau=0$, so that the function value is in $(-1,1)$) with the last component in the tubular neighborhood of $\mk{M}_0$ (resp. $\Gamma^{u/s}_0$) being the function value. Now we are ready to restate results in \cite[Section~4]{hintz2021normally}:
\begin{thm} \label{thm: structure of perturbed dynamics}
\cite[Theorem~4.3, Lemma~4.6]{hintz2021normally} 
Let $\Gamma = \Gamma_0 \cap \{ \tau=0 \}$ (with $\Gamma_0$ defined in \eqref{eq: Gamma 0}) be the trapped set and time infinity and $\mathsf{H}=(H_pt)^{-1}H_p$ be the rescaled Hamilton vector field that has unit speed in $t$. 
Then in a neighborhood of $\Gamma$, we have:
\begin{enumerate}
\item $\mk{M}$ is the graph of a function in $\mathcal{A}_{\cu}^{\beta_T}(\mk{M}_0)$ over $\mk{M}_0$.

\item There exists sets $\Gamma^{u/s}$, to which $\mathsf{H}$ is tangent, that are graphs of functions in $\mathcal{A}_{\cu}^{\beta_T}(\Gamma^{u/s}_0)$ over $\Gamma^{u/s}_0$. In particular, $\Gamma^{u/s} \cap \{\tau = 0\} = \Gamma^{u/s}_0 \cap \{\tau = 0\}$ and $\Gamma^{u/s}_0$ are the stationary extensions of $\Gamma^{u/s} \cap \{\tau = 0\}$.

\item With $\mathcal{A}_{\cu}^{\beta_T}({}^{\cu}S^*M)$ defined similar to Definition \ref{eq: defn, M0 cusp-conormal}, we have defining functions $\phi^{u/s} \in (\mathcal{C}^\infty+\mathcal{A}_{\cu}^{\beta_T})({}^{\cu}S^*M)$ of $\Gamma^{u/s}$ in $\Sigma \cap {}^{\cu}S^*M$ that satisfies assumption \eqref{assumption1}-\eqref{assumption5},\eqref{assumption6'}.
\end{enumerate}

\end{thm}

\begin{rmk}
As pointed out in the proof of \cite[Theorem~4.3]{hintz2021normally}, $\Gamma^s$ is unique while $\Gamma^u$ is not necessarily unique.
The stated results above also used \cite[Proposition~4.1]{hintz2021normally} for the exact Kerr metric, but that and the exact Kerr-de Sitter case have been included in Proposition \ref{prop: trappedset}.
\end{rmk}

\subsection{Construction of the symplectomorphism}
\label{sec: SP construction}
The pseudodifferential algebra we constructed exploits the fact that the defining functions of the stationary extension of the unstable manifold is `$x_1$'.

In this section we construct a homogeneous symplectomorphism mapping $\mathcal{U}$, a conic neighborhood of $\Gamma$, to a conic set in ${}^{\mathrm{cu}}T^*\tilde{M}$. The key property of it is that it sends $\Gamma^{u}_0$, the unstable manifold on the exact Kerr(-de Sitter) spacetime, which is also the stationary extension of $\Gamma^u \cap \{\tau=0\}$, to $\{x_1=0\}$, and we will use this to reduce estimates on asymptotically Kerr(-de Sitter) spacetimes to the model case satisfying this property. 

Define our `model manifold' to be
\begin{align}
\tilde{M}^\circ:=\R_{x_1} \times \mathbb{S}^2 \times \R_{x_4}. \label{model_defn}
\end{align}
Set $\tilde{X}:=\R_{x_1} \times \mathbb{S}^2$ and use $ x_2 \in [0,\pi], x_3 \in \mathbb{S}^1$ as spherical coordinates.
Similar to how we obtain $M$ from $M^\circ$, we define $\tilde{M}$ to be $\tilde{M}^\circ$ compactified by attaching the hypersurface $\{x_4=\infty\}$. Formally, using $x=(x_1,x_2,x_3,x_4)$ as coordinates on $\tilde{M}^\circ$, we set
\begin{align*}
\tilde{M}:= (\tilde{M}^\circ \sqcup (\tilde{X} \times [0,\infty)_{\tilde{\tau}}))/\sim,
\end{align*}
where $\sim$ is the identification: $(x_1,x_2,x_3,x_4) \sim (x_1,x_2,x_3,\tilde{\tau}=x_4^{-1})$. And ${}^{\mathrm{cu}}T^*\tilde{M}$ is equipped with the natural symplectic structure, extending the one on $T^*\tilde{M}^\circ$.



\begin{thm}
There exists a homogeneous symplectomorphism $\Sp$ from $\mathcal{U}$ to ${}^{\mathrm{cu}}T^*\tilde{M}$, such that $\Sp(\Gamma_0^u \cap \mathcal{U}) = \{x_1=0\} \cap \Sp(\mathcal{U}),\, \Sp(\Gamma_0^s \cap \mathcal{U}) = \{\xi_1=0\}\cap\Sp(\mathcal{U})$.
In addition, $\Sp$ preserves the time infinity in the sense that $\Sp(\{\tau=0\})\subset \{\tilde{\tau}=0\}$.
\label{thm_Sp}
\end{thm}
\begin{proof}
We define $\Sp$ on the coordinate patch $\R_t \times X$ and then show that $|t-x_4|$ is bounded, so that the sets $\{t=\infty\}$ and $\{x_4=\infty\}$ coincide. Taking $t\to\infty$ gives the extension of $\Sp$ down to $\{\tau=0\} \subset{}^{\mathrm{cu}}T^*M$.
For $z_0 = (\tau=t^{-1},r,\varphi,\theta) \in \mathcal{U}$, we define $\Sp(z_0)$ by defining its components $(x_1,x_2,x_3,x_4,\xi_1,\xi_2,\xi_3,\xi_4)$. $\Sp$ is a homogeneous symplectomorphism if and only if 
\begin{align}
\begin{split}
& \{\Sp^*(x_i),\Sp^*(x_i)\}=\delta_{ij},\, \{\Sp^*(\xi_i),\Sp^*(\xi_j)\}=\delta_{ij},\\
& \{\Sp^*(x_i),\Sp^*(\xi_j)\}=0,\, \tilde{G}_{t_d} \circ \Sp =\Sp \circ G_{t_d},\label{symp_condition}
\end{split}
\end{align}
where the Poisson brackets are with respect to the natural symplectic structure on ${}^{\mathrm{cu}}T^*M$. $G_{t_d}$, resp. $\tilde{G}_{t_d}$ are dilating fiber parts by $t_d \in \R$ on ${}^{\mathrm{cu}}T^*M$, resp. ${}^{\mathrm{cu}}T^*\tilde{M}$. In the following discussion, we use $x_i,\xi_j$ to donote their pull-back by $\Sp$ when there is no confusion. Define $x_1$ by
\begin{align*}
x_1:=\hat{\varphi}^u = (\xi_t^2+\xi_\varphi^2)^{-1/2}(\xi_r - \sgn(r-r_{\xi_t,\xi_\varphi})(1+\hat{\alpha})\sqrt{ \frac{F_{\xi_t,\xi_\varphi}(r)-F_{\xi_t,\xi_\varphi}(r_{\xi_t,\xi_\varphi})}{\Delta(r)} } ),
\end{align*}
which is the normalized defining function of $\Gamma_0^u$. 

Since $H_{\hat{\varphi}^u}\hat{\varphi}^s>0$ in a neighborhood $\mathcal{U}$ of $\Gamma$ by Proposition \ref{prop_hpphi}, we know that starting from each point $q=(t,r,\theta,\varphi,\xi_t,\xi_r,\xi_\theta,\xi_\varphi)$ in $\mathcal{U}$, the $H_{\hat{\varphi}^u}$-flow  has exactly one intersection with $\{\hat{\varphi}^s=0\}$ and denote it by $q_0=(t_0,r_0,\theta,\varphi_0,\xi_{t},\xi_{r0},\xi_{\theta},\xi_{\varphi})$, where we used the fact that $\xi_t,\xi_\theta,\xi_\varphi,\theta$ are constants along $H_{\hat{\varphi}^u}$-flow.
Define $\xi_1$ to be the time needed for traveling from $p$ to $p_0$ along the $H_{\hat{\varphi}^u}$-flow:
\begin{align*}
\xi_1:=T^s(q)=\frac{r-r_0}{H_{\hat{\varphi}^u}r}=(\xi_t^2+\xi_\varphi^2)^{1/2}(r-r_0).
\end{align*}
Since $\xi_1$ is defined as travel time to $\Gamma_0^s$, hence $\Gamma_0^s$ is
sent to $\{\xi_1=0\}$.
Since $\hat{\varphi}^u$ is preserved by $H_{\hat{\varphi}^u}$-flow, we know
\begin{align}
\xi_{r0}-\xi_r =(1+\hat{\alpha})S_{\xi_t,\xi_\varphi}(r_0)-(1+\hat{\alpha})S_{\xi_t,\xi_\varphi}(r). \label{xi1}
\end{align}
Since $p_0\in \Gamma_0^s$, we know
\begin{align*}
\xi_{r0}=-(1+\hat{\alpha})S_{\xi_t,\xi_\varphi}(r_0),
\end{align*}
where $S_{\xi_t,\xi_\varphi}$ is defined in \eqref{eq: S ksit, ksiphi definition}. 
Thus 
\begin{align}
2\xi_{r_0} = \xi_r - (1+\hat{\alpha})S_{\xi_t,\xi_\varphi}(r) = {\varphi}^u(q). \label{xi2}
\end{align}
Combining (\ref{xi1}) and (\ref{xi2}), we obtain
\begin{align*}
S_{\xi_t,\xi_\varphi}(r_0) & =\frac{\xi_{r0}-\xi_r}{1+\hat{\alpha}} + S_{\xi_t,\xi_\varphi}(r)\\
                  & =\frac{\xi_r - (1+\hat{\alpha})S_{\xi_t,\xi_\varphi}(r)}{2(1+\hat{\alpha})}-\frac{\xi_r}{1+\hat{\alpha}}+S_{\xi_t,\xi_\varphi}(r)\\
                  & = -\frac{\xi_r - (1+\hat{\alpha})S_{\xi_t,\xi_\varphi}(r)}{2(1+\hat{\alpha})}.
\end{align*}
Thus
\begin{align}
r_0 = S_{\xi_t,\xi_\varphi}^{-1}(-\frac{{\varphi}^u(q)}{2(1+\hat{\alpha})}).            \label{r0}
\end{align}
Consequently
\begin{align}
T^s(q)=(\xi_t^2+\xi_\varphi^2)^{1/2}(r-S_{\xi_t,\xi_\varphi}^{-1}(-\frac{{\varphi}^u(q)}{2(1+\hat{\alpha})}) ).  \label{Ts}
\end{align}
By considering what coordinates $x_1=\hat{\varphi}^{u}$ and $\xi_1=T^s(q)$ depend on (namely, $r,\xi_t,\xi_\varphi,\xi_r$), we can choose
\begin{align*}
x_2=\theta,\xi_2=\xi_\theta, \xi_3=\xi_\varphi,\xi_4=\xi_t.
\end{align*}
So $\theta=0,\pi$ are sent to $x_2=0,\pi$ respectively. Those $\mathbb{S}^1_\varphi$ degenerate to a single point in $M$ correspond to those $\mathbb{S}^1_{x_3}$ degenerate to a single point in $\tilde{M}$, thus $\Sp$ is smooth and well defined on $M$. Next we try to find $X_3$ so that
\begin{align*}
x_3=\varphi+X_3(x,\xi)
\end{align*}
satisfies (\ref{symp_condition}), where we are interpreting $\mathbb{S}^1_\varphi$ as $\R_\varphi/2\pi \Z$ and $x_3$ is also parametrizing $\mathbb{S}^1_{x_3}$, i.e. $\R_{x_3}/2\pi\Z$. If two points have the same $x_3$-coordinates, but with different $\varphi,X_3$ respectively, then their $x_1=\hat{\varphi}^u,\xi_1=T^s$ coordinates are different as well because of (\ref{X3}) below. On the other hand, when we choose two different representatives of $\varphi$ for the same point, the $x_3$ components of their image, which is given by $\varphi+X_3$, differ by a multiple of $2\pi$, hence correspond to the same point in $\R_{x_3}/2\pi\Z$. Thus this map is well-defined and injective.
Using (\ref{symp_condition}), $\{\xi_\theta,X_3\}=\{\theta,X_3\}=\{\xi_t,X_3\}=0$ implies that $X_3$ is independent of $\theta,\xi_\theta,t$. $\{\xi_\varphi,\varphi+X_3\}=1$ implies $\{\xi_\varphi,X_3\}$, i.e. $X_3$ is independent of $\varphi$. By remaining equations: $\{\hat{\varphi}^u,\varphi+X_3\}=0,\{T^s,\varphi+X_3\}=0$, the ODE system that $X_3$ satisfies is
\begin{align}
\begin{cases}
H_{T^s}X_3=-\partial_{\xi_\varphi}T^s(q) \\
H_{\hat{\varphi}^u}X_3=-\partial_{\xi_\varphi}(\hat{\varphi}^u)\\
V_DX_3=0,  \label{X3}
\end{cases}
\end{align}
where $V_D=\xi_r\partial_{\xi_r}+\xi_\varphi\partial_{\xi_\varphi}+\xi_t\partial_{\xi_t}$, in which we neglected $\xi_\theta\partial_{\xi_\theta}$ component since it is not involved in the defining function of $\Gamma_0^{u/s}$ and hence not involved in our discussion.

Next we apply \cite[Corollary~C.1.2]{hormander2007analysis}. 
Applying equation $(21.1.6)'''$ of \cite{hormander2007analysis}, we obtain
\begin{align*}
& [H_{\hat{\varphi}^u},V_D]=H_{\hat{\varphi}^u},\\
& [H_{T^s},V_D]=0.
\end{align*}
Combining with the fact that $H_\theta,H_{\xi_\theta},H_{\hat{\varphi}^u},H_{T^s},H_{\xi_\varphi}$ commute with each other, the condition (C.1.2) is satisfied. Next we verify (C.1.4) there:
\begin{align*}
H_{T^s}(\partial_{\xi_\varphi}\hat{\varphi}^u)-H_{\hat{\varphi}^u}(\partial_{\xi_\varphi}T^s)=0,
\end{align*}
which is equivalent to, using $H_\varphi=-\partial_{\xi_\varphi}$,
\begin{align*}
\partial_{\xi_\varphi}H_{T^s}\hat{\varphi}^u+[H_{T^s},-H_\varphi]\hat{\varphi}^u=\partial_{\xi_\varphi}H_{\hat{\varphi}^u}T^s+[H_{\hat{\varphi}^u},-H_\varphi]T^s.
\end{align*}
Since $H_{\hat{\varphi}^u}T^s=1$ by the definition of $T^s$, the first terms on both sides vanish, and this is equivalent to
\begin{align*}
H_{\{T^s,\varphi\}}\hat{\varphi}^u+H_{\{\varphi,\hat{\varphi}^u\}}T^s =0.
\end{align*}
Using Poisson brackets, this is
\begin{align*}
\{\{T^s,\varphi\},\hat{\varphi}^u\}+\{\{\varphi,\hat{\varphi}^u\},T^s\}=0.
\end{align*}
Since $\{ \{\hat{\varphi}^u,T^s \}   ,\varphi\} = \{1,\varphi\}=0$, above equation holds by Jacobi's identity.
Equations in (C.1.4) involving $V_D$ are 
\begin{align*}
& H_{T^s}0-V_D(-\partial_{\xi_\varphi}T^s)=0,\\
& H_{\hat{\varphi}^u}0-V_D(-\partial_{\xi_\varphi}\hat{\varphi}^u)=-\partial_{\xi_\varphi}\hat{\varphi}^u,
\end{align*}
both of which follow from the fact that $\partial_{\xi_\varphi}T^s$ and $\partial_{\xi_\varphi}\varphi$ are homogeneous function of degree 0 and -1 respectively and Euler's Theorem on homogeneous functions. Thus, we can solve the ODE of $X_3$ with prescribed initial condition on $\Theta_1:=\mathcal{U}\cap\{T^s=\hat{\varphi}^u=0, t=t_1,\theta=\theta_1,\varphi=\varphi_1,\xi_\theta=\xi_{\theta 1},|\xi|=1  \}$, where $t_1,\xi_{\theta 1},\varphi_1,\theta_1$ are fixed constants and $|\xi_{\theta 1}|<1$. By Proposition \ref{prop: trappedset}, $T^s=\hat{\varphi}^u=0$ is equivalent to $r=r_{\xi_t,\xi_\varphi},\xi_r=0$, thus $\Theta_1$ is parametrized by $\xi_\varphi$. We set $X_3=0$ on this codimension 7 smooth submanifold.

Using the same method, we construct $X_4$ in $x_4=t+X_4$ such that
\begin{align*}
\begin{cases}
H_{T^s}X_4=-\partial_{\xi_t}T^s(q) \\
H_{\hat{\varphi}^u}X_4=-\partial_{\xi_t}(\hat{\varphi}^u)\\
H_{x_3}X_4=-\partial_{\xi_t}x_3,\\
V_DX_4=0,
\end{cases}
\end{align*}
where other commutation relations are already satisfied. Condition (C.1.2) and (C.1.4) in \cite{hormander2007analysis} are verified in the same manner. We take the first and the third equations as examples. We need to verify
\begin{align*}
H_{x_3}\partial_{\xi_t}T^s(q)-H_{T^s}\partial_{\xi_t}x_3=0.
\end{align*}
This is equivalent to 
\begin{align*}
\{ \{T^s,t\},x_3\}+\{ \{t,x_3\},T^s\}=0,
\end{align*}
which follows from $\{ \{x_3,T^s\},t\}=\{0,T^s\}=0$ and Jacobi's identity. The rest conditions involving $V_D$ are verified in the same manner as when we solve $X_3$. And we can assign initial value of $X_4$ at $\Theta_2:=\mathcal{U}\cap\{\xi_r=0,r=r_{\xi_t,\xi_\varphi}, t=t_1,\theta=\theta_1,\varphi=\varphi_1,\xi_\theta=\xi_{\theta 1},\xi_\varphi=\xi_{\varphi 1},|\xi|=1  \}$ with $|\xi_{\varphi 1}|^2+|\xi_{\theta 1}|^2<1$, which is a single point. We set $X_4=0$ at $\Theta_2$.
Finally, we show that $X_4=x_4-t$ is bounded. Since $X_4$ is homogeneous of degree 0, we consider its value on the unit sphere bundle. Then the variables in the ODE from which we obtain $X_4$ are varying over a bounded region and the right hand sides of those ODEs are bounded as well, thus $X_4$ is bounded.
\end{proof}
Then as discussed before, for asymptotically Kerr(-de Sittter) spacetimes, we have:
\begin{coro}
The homogeneous symplectomorphism $\Sp$ constructed in Theorem \ref{thm_Sp} sends the stationary extension of $\Gamma^{u} \cap \{\tau=0\}$ to $\{x_1=0\}$, and the stationary extension of $\Gamma^{s} \cap \{\tau = 0\}$ to $\{\xi_1=0\}$.
\label{coro_Sp}
\end{coro}

\subsection{Quantize \texorpdfstring{$\Sp$}{Sp}}
Next we apply the Egorov's theorem of conjugating pseudodifferential operators by multi-valued Fourier integral operators to reduce our estimates to the model case. 
The key observation is that, although the Fourier integral operator $T$ itself, which is locally defined, has 
an obstruction to be glued to a global Fourier integral operator, but this obstruction is a transition factor of the form $e^{i\alpha_{ij}}$,
and hence always cancels out if we conjugate by $T$, i.e., multiplying $T,T^*$ simultaneously.
We use a finite open cover $\{\mathcal{U}_j\}_{j\in J}$ of $\mathcal{U}$ so that $\mathcal{U}_i \cap \mathcal{U}_j$ is contractible for $i,j \in J$ and $\{(\pi_M(\mathcal{U}_j),\varphi_j)\}$ is an atlas of $M$. We use $J=\{1,2\}$, which is possible for our $M$. Here we allow half spaces in the model of charts, so that the compactified $\R_t$ only need one chart to cover and the 2 charts are needed for the $\mathbb{S}^2$ component.

We denote the coordinates on $\mathcal{U}_j$ by $(y,\eta)=(t_y,r_y,\varphi_y,\theta_y,\eta_t,\eta_r,\eta_\varphi,\eta_\theta)$ and the coordinates on $\Sp(\mathcal{U}_j)$ by $(x,\zeta)=(x_1,x_2,x_3,x_4=\tilde{\tau}^{-1},\zeta_1,\zeta_2,\zeta_3,\zeta_4)$, which is valid down to $\tau=0$ when we consider oscillatory integral expression of $T_j$.
We apply \cite[Theorem~10.1]{grigis1994microlocal}. The original proof given is valid for conjugating classical pseudodifferential operators, we will verify that it is valid for cusp calculus in our setting as well after verifying the condition needed to quantize $\Sp|_{\mathcal{U}_j}$ in that theorem.

\begin{prop}  \label{prop: conjugate operator}
Let $\{\chi_j\}_{j\in J}$ be a partition of unity subordinate to $\{\pi_M(\mathcal{U}_j)\}_{j\in J}$, shrinking $\mathcal{U}_j$ if necessary, we can choose a family of Fourier integral operators $T_j$ associated to $\Sp|_{\mathcal{U}_j}$ such that
\begin{enumerate}  \label{Tj_property}
	\item $T_j^*T_j-\Id_M \in \Psi_{\mathrm{cu}}^{-\infty,0}(M)\text{ microlocally on } \mathcal{U}_j,\, T_jT_j^*-\Id_{\tilde{M}} \in \Psi_{\mathrm{cu}}^{-\infty,0}(\tilde{M})$ microlocally on $\Sp(\mathcal{U}_j)$. $T_j$ is elliptic on $\mathcal{U}_j$. \label{property1}
	\item \label{property2} For $\tb{A} \in \Psi_{\mathrm{cu}}^{m,r}(M)$ with principal symbol $\tb{a}$ and $\WF(a) \subset \mathcal{U}_j$, $\tb{B}:=T_j\tb{A}T_j^* \in \Psi_{\mathrm{cu}}^{m,r}(\tilde{M})$ has principal symbol
\begin{align}
\tb{b}=(\Sp^{-1})^*(\tb{a})=\tb{a}\circ(\Sp^{-1}).  \label{b-symbol}
\end{align}
If instead $\tb{A} \in \Psi_{\infty}^{m,r}(M)$, then $\tb{B}:=T_j\tb{A}T_j^* \in \Psi_{\infty}^{m,r}(\tilde{M})$ and its principal symbol is still \eqref{b-symbol}.

    \item \label{property3} $T_k^*T_j-c_{kj}e^{i\alpha_{kj}}\in \Psi_{cu}^{-\infty,0}(M)$ on $\mathcal{U}_j\cap \mathcal{U}_k$, where $c_{kj},\alpha_{kj}\in \R$.
\end{enumerate}
Define the global Fourier integral operator $T$ by 
\begin{align*}
Tu = T_1\chi_1u+c_{21}e^{i\alpha_{21}}T_2\chi_2u ,\quad u \in D'(M).
\end{align*}
Then for $P \in \Psi_{\cu}^{m,0}(M)$, define the glued conjugated $P$ by:
\begin{align}
\tilde{P}:=\sum_{j \in J} T_j(\chi_jP)T_j^*,  \label{Ptilde}
\end{align}
we have $\tilde{P} \in \Psi_{\cu}^{m,0}(\tilde{M})$ and
\begin{align}
\tilde{P}T=TP+R,  \label{PT_eq}
\end{align}
where $R$ is smoothing in the sense that it is a sum of smooth multiples of $T_jR_j$ with $R_j \in \Psi_{\rm cu}^{-\infty,0}(M)$. 

If instead we have $P \in \Psi_{\cu}^{m,0}(M)+\tau^{\beta_T}\Psi_{\infty}^{m,0}(M)$, then with $\tilde{P}$ defined in the same manner above but now
$\tilde{P} \in \Psi_{\cu}^{m}(\tilde{M})+\tilde{\tau}^{\beta_T}\Psi_{\infty}^{m,0}(\tilde{M})$, and \eqref{PT_eq} still holds with $R$ being a a sum of smooth multiples of $T_jR_j$ with 
$R_j \in \Psi_{\cu}^{-\infty,0}(M)+\tau^{\beta_T}\Psi_{\infty}^{-\infty,0}(M)$.
\end{prop}

\begin{proof}
We first prove the theorem for $\tb{A} \in \Psi_{\cu}^{m,r}(M)$ and $P \in \Psi^{m,0}_{\cu}(M)$, and describe modifications needed for $\tb{A} \in \Psi_{\infty}^{m,r}(M)$ and $P$ has a perturbation term in $\tau^{\beta_T}\Psi^{m,0}_\infty(M)$.

The condition we need to verify in order to apply \cite[Theorem~10.1]{grigis1994microlocal} is that $G(\Sp) \subset{}^{\mathrm{cu}}T^*M \times {}^{\mathrm{cu}}T^*\tilde{M}$, the graph of $\Sp$, has a generating function 
in the sense of \cite{grigis1994microlocal}, which means there exists $\varphi(z,{\eta})$ such that $\Sp$
has the form
\begin{align} \label{defn_kds_generating_function}
(\varphi'_{{\eta}}(x,{\eta}),{\eta})  \mapsto  (x,\varphi'_x(x,{\eta})) .
\end{align} 
By \cite[Theorem~5.3]{grigis1994microlocal}, we need to verify that, the projection $\pi_{\Sp}$ from $G(\Sp)$ to $M \times \R^4$, i.e., the $(\tau,r,\theta,\varphi,\zeta_1,\zeta_2,\zeta_3,\zeta_4)$ part is a diffeomorphism. Although the result of Theorem 5.3 there is local, but the proof only relies on the fact that closed differential forms are locally exact, which is true on any contractible region, thus its proof gives a generating function that is `global' on $\mathcal{U} \times {}^{\mathrm{cu}}T^*\tilde{M}$. Since $\zeta_2=\zeta_\theta,\zeta_3=\zeta_\varphi,\zeta_4=\zeta_t=-\zeta_\tau$, and the projection from $G(\Sp)$ to ${}^{\mathrm{cu}}T^*M$, i.e., the $(\tau,r,\theta,\varphi,\zeta_t,\zeta_r,\zeta_\theta,\zeta_\varphi)$ part is a diffeomorphism, we need to verify
\begin{align*}
\frac{\partial \zeta_1}{\partial \zeta_r} \neq 0,
\end{align*}
which is sufficient for $\pi_{\Sp}$ to have full rank. This is straightforward by $\zeta_1=(\zeta_t^2+\zeta_\varphi^2)^{1/2}(r-S_{\zeta_t,\zeta_\varphi}^{-1}(-\frac{{\varphi}^u(q)}{2(1+\hat{\alpha})}))$, ${\varphi}^u=\zeta_r-S_{\zeta_t,\zeta_\varphi}(r)$ and $(S_{\zeta_t,\zeta_\varphi}^{-1})' \neq 0$, where the last condition holds because by (\ref{Sdecomp}), when $r$ is close to $r_{\zeta_t,\zeta_\varphi}$, $S'_{\zeta_t,\zeta_\varphi}(r)$ is continuous and positive, thus $(S_{\zeta_t,\zeta_\varphi}^{-1})'$ is continuous and positive by the inverse function theorem. 

Consequently, we can apply \cite[Theorem~10.1]{grigis1994microlocal} (assuming that it holds for the cusp calculus for the moment), 
which is stated for a small conic region, but its proof goes through on the region where the generating function is valid, hence on $G(\Sp) \subset \mathcal{U} \times {}^{\mathrm{cu}}T^*\tilde{M}$. We obtain a Fourier integral operator $T_j:\, D'(\pi_M(\mathcal{U}_j)) \rightarrow D'(\pi_{\tilde{M}}(\Sp(\mathcal{U}_j)))$ quantizing $\Sp|_{\mathcal{U}_j}$, being elliptic on and microlocally unitary on $\mathcal{U}_j$ in the sense that
\begin{align*}
T_j^*T_j-\Id_M \in \Psi_{\mathrm{cu}}^{-\infty,0}(M)\text{ on } \mathcal{U}_j,\, T_jT_j^*-\Id_{\tilde{M}} \in \Psi_{\mathrm{cu}}^{-\infty,0}(\tilde{M}) \text{ on } \Sp(\mathcal{U}_j).
\end{align*}
The unitary property is not included in \cite{grigis1994microlocal}, but as stated before (2.7) of \cite{hitrik2004non}, we can achieve this by adding a smooth factor to the amplitude. In addition, for $\tb{A} \in \Psi_{\mathrm{cu}}^{m,r}(M)$ with principal symbol $\tb{a}$ and $\WF'(\tb{A}) \subset \mathcal{U}_j$, we know $\tb{B}:=T_j\tb{A}T_j^*\in \Psi_{\mathrm{cu}}^{m,r}(\tilde{M})$. And it has the same symbol modulo $S_{\mathrm{cu}}^{m-2,r}(M)$:
\begin{align}
\tb{b}=(\Sp^{-1})^*(\tb{a})=\tb{a}\circ(\Sp^{-1}) \, \mathrm{mod} \, S_{\mathrm{cu}}^{m-2,r}(M),  \label{improved_egorov}
\end{align}
which is the improved Egorov's property, i.e., their symbols coincide up to the sub-principal level 
and this comes from the argument in \cite[Section~2]{hitrik2004non}. Although the argument there is for the semiclassical case, but the part starting from equation (2.6) there goes through for classical case as well if one understands $O(h^\mu)$ terms as being $\mu$ orders lower in the cusp differential order sense, and if one drops all factors of $h^{-1}$ in the phases. The fact we use is that differences between oscillatory integrals in different patches are purely imaginary. Thus, when we go back to the same point following a closed loop, the product of all these factors, which is the obstruction to compatibility, is a purely imaginary factor. This factor is cancelled when we apply both the Fourier integral operator and its adjoint.

We mention that this property is also discussed in \cite{silva2007accuracy}, and the global version is mentioned in remark i) in Section 1.2. 

Now we verify that \cite[Theorem~10.1]{grigis1994microlocal} applies to the cusp calculus as well. 
In each $\mathcal{U}_j$, we choose $T_j$ of the form
\begin{align} \label{T_integral_form}
T_ju(x) = \int e^{i(\varphi(x,{\eta})- {y} \cdot {\eta})}
t_j(x,y,{\eta})u(z)dyd{\eta} +Ku,
\end{align}
where $K$ is a smoothing operator, and $T_j^*$ has a similar expression. 

Under composition, for $\tb{A} \in \Psi_{\mathrm{cu}}^{m,r}(M)$ with left symbol $\tb{a}$ and $\WF'(\tb{A}) \subset \mathcal{U}_j$, $T\tb{A}T^*$ has an oscillatory integral representation (see \cite[page~109]{grigis1994microlocal})
\begin{align} \label{TAT_integral_form} 
\begin{split}
(T_j\tb{A}T_j^*u)(x) = \int & e^{i(\varphi(x,\eta)-x' \cdot \eta+ x'\cdot \theta-\varphi(y,\theta))} t_j(x,x',\eta)
\tb{a}(y,\varphi'_x(x,\eta))\\
& t^*(x',y,\theta)u(y)dx'd\eta dyd\theta +Ru, \quad R \in \Psi^{-\infty,0}_{\cu,\alpha}(\tilde{M}),
\end{split}
\end{align}
for $u \in \mathcal{C}^\infty(\mathcal{U}_j)$ and we used
$({y}',{\eta}')$ to denote a point on ${}^{\mathrm{cu}}T^*\tilde{M}$.
Applying the stationary phase lemma, its leading contribution 
is from the non-degenerate critical point
$\eta = \theta, x'=\varphi'_\eta(x,\eta)$, and we have
\begin{align*}
(T_j\tb{A}T_j^*u)(x) =  \int e^{i(\varphi(x,\theta) - \varphi(y,\theta))}
c(x,y,\theta) u(y)dyd\theta+Ru,
\end{align*}
where $R \in \Psi^{-\infty,0}_{\cu,\alpha}(\tilde{M})$ and $c(x,y,\theta)$ is a sum of derivatives of $t_j(x,\varphi'_\eta(x,\eta),\theta) 
\tb{a}(y,\varphi'_x(x,\eta))t_j^*(\varphi'_\eta(x,\eta),y,\theta)$
with respect to $(x',\eta)$, multiplied by factors involving derivatives of $\varphi(x,\eta)-x' \cdot \eta+ x'\cdot \theta-\varphi(y,\theta)$
with respect to $(x',\eta)$.
Notice that,
\begin{align*}
\varphi(x,\theta) - \varphi(y,\theta)=(x-y)\Xi(x,y,\theta),
\end{align*}
where $\Xi$ is smooth and homogeneous of degree 1 with respect to $\theta$, and so does its inverse. 
After a change of variable, we have
\begin{align} \label{eq: TjATj, after reduction}
(T_j\tb{A}T_j^*u)(x) = \int e^{i(x-y) \cdot \xi}
\tilde{c}(x,y,\xi)
u(y) |\det \frac{\partial \Xi}{\partial \theta}|^{-1} dyd\xi + Ru,
\end{align}
where $R \in \Psi^{-\infty,0}_{\cu,\alpha}(\tilde{M})$ and $\tilde{c}$ is determined by $\tilde{c}(x,y,\Xi(x,y,\theta))=c(x,y,\theta)$. This has the form of a cusp pseudodifferential operator, as long as we verify that $c(x,{y},\eta)|_{x={y}}$ is a cusp symbol.
This is because  $c$ is a sum of derivatives of $t(x,\varphi'_\eta(x,\eta),\theta) \tb{a}(y,\varphi'_x(x,\eta))t^*(\varphi'_\eta(x,\eta),y,\theta)$ with respect to $(x',\eta)$, which is a cusp symbol if we choose $t_j$ to be smooth in $\tau=t^{-1}$, multiplied by factors involving derivatives of $\varphi(x,\eta)-x' \cdot \eta+ x'\cdot \theta-\varphi(y,\theta)$ with respect to $(x',\eta)$.
By the remark at the end of \cite[Section~9]{grigis1994microlocal}, the generating function is
\begin{align*}
\varphi(x,\eta) = y\cdot \eta|_{G(\Sp)}=\pi_y(\pi_{x,\eta}^{-1}(x,\eta))\cdot \eta.
\end{align*}
Notice that the only part in $y$ that involves $t$ is $x_4=t+X_4$, with $X_4$ being independent of $t$, thus all these extra factors arising from applying the method of stationary phase are bounded under iterated application of $t^2\partial_t$, which is the desired property for a cusp symbol (the boundedness when taking other derivatives is clear).

Let $\{\chi_j\}_{j\in J}$ be a partition of unity subordinate to $\{\pi_M(\mathcal{U}_j)\}_{j\in J}$, shrinking $\mathcal{U}_j$ if necessary, we can choose a family of Fourier integral operators $T_j$ associated to $\Sp|_{\mathcal{U}_j}$ satisfying properties listed in Proposition \ref{Tj_property}.
Define the global Fourier integral operator $T$ by 
\begin{align*}
Tu = T_1\chi_1u+c_{21}e^{i\alpha_{21}}T_2\chi_2u ,\quad u \in D'(M).
\end{align*}
By Property (\ref{property3}) in the statement of the proposition, the two terms here are microlocally equal on $\mathcal{U}_j\cap\mathcal{U}_k$, and since $T_j$ is elliptic on $\mathcal{U}_j$, $T$ is elliptic on $\mathcal{U}$. Notice that $T_1^*T_2T_2^*T_1 = \Id_M$ microlocally on $\mathcal{U}_1\cap\mathcal{U}_2$, thus we know
\begin{align*}
c_{12}=c_{21}^{-1},\alpha_{12}=-\alpha_{21}.
\end{align*}
Direct computation shows
\begin{align*}
T^*T-\Id_M \in \Psi_{\mathrm{cu}}^{-\infty,0}(M)\text{ on } \mathcal{U},\, TT^*-\Id_{\tilde{M}} \in \Psi_{\mathrm{cu}}^{-\infty,0}(\tilde{M}) \text{ on } \Sp(\mathcal{U}).
\end{align*}

Next we prove \eqref{PT_eq} when $P \in \Psi_{\cu}^{m,0}(M)$. For $\tilde{P}=\sum_{j \in J} T_j(\chi_jP)T_j^*$, we have
\begin{align*}
\tilde{P}Tu = & (T_1\chi_1PT_1^*+T_2\chi_2PT_2^*)(T_1\chi_1u+c_{21}e^{i\alpha_{21}}T_2\chi_2)u\\
            = & (T_1\chi_1P\chi_1+T_2\chi_2Pc_{21}e^{i\alpha_{21}}\chi_1+T_1\chi_1PT_1^*c_{21}e^{i\alpha_{21}}T_2\chi_2+ \\
            & c_{21}e^{i\alpha_{21}}T_2\chi_2P\chi_2)u+Ru.
\end{align*}
Since the third term has both $\chi_1,\chi_2$ factor, and on $\mathcal{U}_1\cap\mathcal{U}_2$ we have $T_1^*T_2=c_{12}e^{i\alpha_{12}}+R$. So 
\begin{align*}
T_1\chi_1PT_1^*c_{21}e^{i\alpha_{21}}T_2\chi_2u=T_1\chi_1Pc_{21}c{12}e^{i(\alpha_{21}+\alpha_{12})}\chi_2u=T_1\chi_1P\chi_2u.
\end{align*}
We have 
\begin{align}
\begin{split}
\tilde{P}Tu = & (T_1\chi_1P\chi_1+T_2\chi_2Pc_{21}e^{i\alpha_{21}}\chi_1+T_1\chi_1P\chi_2+c_{21}e^{i\alpha_{21}}T_2\chi_2P\chi_2)u \\
&+Ru\\
            = & (T_1\chi_1+c_{21}e^{i\alpha_{21}}T_2\chi_2)Pu+Ru\\
            = & TPu+Ru,
\end{split}
\label{PTcommute}
\end{align}
where $R$ is a sum of smooth multiples of $T_jR_j$, $R_j \in \Psi_{\cu}^{-\infty,0}(M)$. $R$ has this form because this error is introduced when we replace $T_j^*T_j$ by $\Id_M$ or $T_k^*T_j$ by $c_{kj}e^{i\alpha_{kj}}$ and then compose with factors in the front.
The fact $\tilde{P} \in \Psi_{\cu}^{m,0}(\tilde{M})$ follows from part \eqref{property2} of this proposition. $R$ representing the microlocal errors in different steps may represent different operators, but all of them are infinitely smoothing operators and do not affect estimates.

Next, we consider property \eqref{property2} for $\tb{A} \in \Psi_{\infty}^{m,r}(M)$. The proof is essentially the same, the assumption and conclusion on variables other than $t$ remain the same. For $t$-direction, the only difference being that $\tb{a}$ now is only assumed to be uniformly bounded in $t$ and bounded under iterative of $\partial_t$ instead of $t^2\partial_t$. But we also only need the same property for $\tilde{c}$ in \eqref{eq: TjATj, after reduction}. Since properties of other factors remain the same, thus it is bounded under iterative application of $\partial_t$.

When $P \in \Psi_{\cu}^{m,0}(M)+\tau^{\beta_T}\Psi_{\infty}^{m,0}(M)$, we write it as $P_0+P_1$ with $P_0 \in \Psi_{\cu}^{m,0}(M), P_1 \in \tau^{\beta_T}\Psi_{\infty}^{m,0}(M)$, and then we apply part \eqref{property2} of this proposition to $P_0,\tilde{P}$ respectively to obtain $\tilde{P} \in \Psi_{\cu}^{m,0}(\tilde{M})+ \tilde{\tau}^{\beta_T}\Psi_{\infty}^{m,0}(\tilde{M})$, and similarly for the error term.
\end{proof}

\subsection{Properties of the conjugated operator}
In this part we verify that $\tilde{P}$ in Proposition \ref{prop: conjugate operator} fits into the framework in Section \ref{sec_microlocal_analysis} to Section \ref{sec_commutator}. In particular, the microlocal estimate in Section \ref{sec_commutator} holds with $P$ replaced by $\tilde{P}$ and $v$ replaced by $Tv$. 

First, we prove that weighted Sobolev spaces are preserved under $T$. When an inequality involves norms on both $M$ and $\tilde{M}$, we 
write $H_{\mathrm{cu}}^{s}(\tilde{M}), H_{\mathrm{cu}}^{s}(M)$ to index the norm instead of only writing the orders.
\begin{prop} \label{prop_decay}
Suppose $u \in \tau^\mu H_{\mathrm{cu}}^{s}(M)$ and $\supp u \subset \pi_M(\mathcal{U})$, then $Tu \in \tilde{\tau}^\mu H_{\mathrm{cu}}^s(\tilde{M})$. And conversely if $Tu \in \tilde{\tau}^\mu H_{\mathrm{cu}}^s(\tilde{M})$, then  $u \in \tau^\mu H_{\mathrm{cu}}^{s}(M)$ and 
\begin{align}
\begin{split}
& ||\tilde{\tau}^{-\mu}Tu||_{H_{\mathrm{cu}}^{s}(\tilde{M})} \lesssim ||\tau^{-\mu}u||_{H_{\mathrm{cu}}^s(M)}, \\
& ||\tau^{-\mu}u||_{H_{\mathrm{cu}}^{s}(M)} \lesssim ||\tilde{\tau}^{-\mu}Tu||_{H_{\mathrm{cu}}^s(\tilde{M})}+||\tau^{-\mu}u||_{H_{\mathrm{cu}}^{-N}(M)}. 
\end{split}
\label{est_decay}
\end{align}
Also, we have 
\begin{align}
||\tilde{\tau}^{-\mu}T_ju||_{H_{\mathrm{cu}}^{s}(\tilde{M})} \lesssim ||\tau^{-\mu}u||_{H_{\mathrm{cu}}^s(M)},
\end{align}
for $j=1,2$.
\end{prop}
\begin{proof}
For the inequality that $Tu$ is on the left hand side, by the definition of $T$, we only need to prove the same estimates for $T_j,j=1,2$ respectively. In this proof we denote the coordinates on $\mathcal{U}_j$ by $(y,\eta)=(t_y,r_y,\varphi_y,\theta_y,\eta_t,\eta_r,\eta_\varphi,\eta_\theta)$ and the coordinates on $\Sp(\mathcal{U}_j)$ by $(x,\xi)=(x_1,x_2,x_3,x_4=\tilde{\tau}^{-1},\xi_1,\xi_2,\xi_3,\xi_4)$, which is valid down to $\tau=0$ when we consider oscillatory integral expression of $T_j$. First suppose that $u \in \tau^\mu H_{\mathrm{cu}}^{s}(M)$ and $\supp u \subset \pi_M(\mathcal{U}_j)$ and we need to show $T_ju \in \tilde{\tau}^\mu H_{\mathrm{cu}}^s(\tilde{M})$. By the remark at the end of Section 9 of \cite{grigis1994microlocal}, the generating function is
\begin{align*}
\varphi(x,\eta) = y\cdot \eta|_{G(\Sp)}=\pi_y(\pi_{x,\eta}^{-1}(x,\eta))\cdot \eta.
\end{align*}
By the argument in \cite{grigis1994microlocal}, this is valid on the entire $G(\Sp)$ since we have shown that the projection from $G(\Sp)$ to $(x,\eta)$ is a diffeomorphism. $\pi_y(\pi_{x,\eta}^{-1}(x,\eta))$ is the $y$-component of the point on $G(\Sp)$ parametrized by $(x,\eta)$, which is different, actually independent of, the $y$ written in the oscillatory integral. $T_j$ can be written as
\begin{align*}
T_ju(x)=\int \int e^{i(\pi_y(\pi_{x,\eta}^{-1}(x,\eta))-y)\cdot \eta}a(x,y,\eta)u(y)dyd\eta.    
\end{align*}
We choose $a(x,y,\eta)$ to be smooth and uniformly bounded. Denote $\pi_y(\pi_{x,\eta}^{-1}(x,\eta))$ by $z(x,\eta)=(t_z,r_z,\varphi_z,\theta_z)$ and denote components of $y$ by $(t_y,r_y,\varphi_y,\theta_y)$, then
\begin{align}
T_ju(x)=\int \int e^{i(z(x,\eta)-y)\cdot \eta} a(x,y,\eta) u(y)
dt_ydr_yd\varphi_yd\theta_y d\eta.   \label{T_int}
\end{align}
In discussion below, subindex $X$ means function spaces or variables corresponds to variables on $X$ and their dual variables. For example, $y_X=(r_y,\varphi_y,\theta_y)$. Define the Sobolev space $H_{\mathrm{cu},t}^{s_1}H_{\mathrm{cu},X}^{s_2}(M)$ as: $u \in H_{\mathrm{cu},t}^{s_1}H_{\mathrm{cu},X}^{s_2}(M)$ if and only if
\begin{align} \label{defn_decomposed_Sobolev}
\la\eta_t\ra^{s_1} \la \eta_X \ra^{s_2} \hat{u}(\eta) \in L^2(\R^4),
\end{align}
where $\eta_X=(\eta_\theta,\eta_\varphi,\eta_r)$. 
Now suppose $t^\mu u(y) \in H_{\cu}^{s}(M)$, which is equivalent to (since we assume $\supp u \subset \{t \geq 1\}$ for all functions $u$ in our discussion)
\begin{align*}
|t|^\mu  D_t ^{\alpha_1} D_X^{\alpha_2}u \in L_{\mathrm{cu}}^2(M),
\end{align*}
for all $\alpha_1\in \R,\alpha_2 \in \R^{n-1}$ such that $\alpha_1$ and components of $\alpha_2$ have the same sign as $s$ and sum to be $s$. $D_t ^{\alpha_1},D_X^{\alpha_2}$ are understood as the left quantizations of $ \eta_t^{\alpha_1},\eta_X^{\alpha_2}$. Thus, we can take partial Fourier transform in $t$ and obtain
\begin{align}
|\eta_t|^{\alpha_1}D_X^{\alpha_2}\mathcal{F}_tu \in H^\mu_{\eta_t}L_{\mathrm{cu},X}^2(\R_{\eta_t}\times X),
\end{align}
where $H^\mu_{\eta_t}L_{\mathrm{cu},X}^2(\R_{\eta_t}\times X)$ is defined in the same manner as (\ref{defn_decomposed_Sobolev}).
Integrate against $t_y$ in (\ref{T_int}) first, we obtain
\begin{align*}
T_ju(x) = & \int \int e^{i(z(x,\eta)\cdot \eta-y_X\cdot \eta_X)} \mathcal{F}_t(au)(x,\eta_t,y_X,\eta)d\eta_tdXd\eta_X\\
        = & \int \int e^{it_x\eta_t}e^{-iX_4(z(x,\eta))\eta_t}e^{i(z_X-y_X)\cdot\eta_X} \mathcal{F}_t(au)(x,\eta_t,y_X,\eta)d\eta_tdXd\eta_X.
\end{align*}
By the construction process of $X_4$, it is independent of $t_x$, hence we can interchange $e^{-iX_4(z(x,\eta))\eta_t}e^{i(z_X-y_X)\cdot\eta_X}$ and $\mathcal{F}_t$, and multiplying $e^{-iX_4(z(x,\eta))\eta_t}e^{i(z_X-y_X)\cdot\eta_X}$  does not affect both smoothness and integrability of $au$, thus 
\begin{align}
|\eta_t|^{\alpha_1}D_X^{\alpha_2}\mathcal{F}_t(e^{-iX_4(z(x,\eta))\eta_t}e^{i(z_X-y_X)\cdot\eta_X}u) \in H^\mu_{\eta_t}L_X^2(\R_{\eta_t}\times X).
\label{T1_2}
\end{align}
The integration against $\eta_t$ is an inverse Fourier transform evaluated at $t_x$, and by (\ref{T1_2}) we know
\begin{align*}
\la t_x \ra^{\mu} T_ju(x) \in  H_{\mathrm{cu}}^s(\tilde{M}),
\end{align*}
which completes the proof of the first claim. Notice that our procedure above is using the equivalences of norms between different Sobolev spaces, hence we have the first inequality of (\ref{est_decay}):
\begin{align*}
\begin{split}
||\tilde{\tau}^{-\mu}Tu||_{H_{\mathrm{cu}}^{s}(\tilde{M})} \lesssim ||\tau^{-\mu}u||_{H_{\mathrm{cu}}^s(M)}.
\end{split}
\end{align*}
Conversely, the same argument applies to $T^*$, which quantizes $\Sp^{-1}$, and notice that $T^*T=\Id_{M}+R$, with $R \in \Psi_{\mathrm{cu}}^{-\infty,0}(M)$, we obtain the second inequality of (\ref{est_decay}).
\end{proof}
 
Since the symplectomorphism preserves the symplectic form, it preserves equations using Hamiltonian vector fields, hence the normally hyperbolic trapping (more concretely, \eqref{hpphi}) is preserved under the symplectomorphism. Define $\tilde{\Gamma}^u_0=\Sp(\Gamma^u_0), \tilde{\Gamma}^u = \Sp(\Gamma^u),\tilde{\Gamma}^s_0=\Sp(\Gamma^s_0),
\tilde{\Gamma}^s = \Sp(\Gamma^s), \tilde{p}= (\Sp^{-1})^*p.$
Since $\Sp$ is a symplectomorphism, we know $H_{\tilde{p}}=\Sp_*H_p$. 
In addition, the minimal contract/expansion rate $\nu_{\min}$ is unchanged under this symplectomorphism. Applying Proposition \ref{prop: conjugate operator} to $P-P^*$, we know that the bound \eqref{eq: skew p1 bound} on the subprincipal symbol holds for $\tilde{P}$ if $P$ satisfy the same bound.
Summarizing discussion above, we have:
\begin{prop} \label{prop: dynamics after conjugation}
The trapping of the flow of $H_{\tilde{p}}$ is eventually absolutely $r$-normally hyperbolic for every $r$ in the sense of \cite{wunsch2011resolvent}. The unstable and stable manifolds are $\tilde{\Gamma}^u,\tilde{\Gamma}^s$ with defining functions $(\Sp^{-1})^*\phi^{u}, (\Sp^{-1})^*\phi^{u}$, where $\phi^{u/s}$ are given in Theorem \ref{thm: structure of perturbed dynamics}. The stationary extension of their intersection with the time infinity $\{\tau=0\}$ are $\tilde{\Gamma}^{u/s}_0$, which have defining functions $x_1,(\Sp^{-1})^*\hat{\varphi}^s$ respectively, and they satisfy dynamical assumptions \eqref{assumption1}-\eqref{assumption5}, \eqref{assumption6'}, \eqref{assumption7} in Section \ref{sec_assumptions} with $\bar{\phi}^u,\bar{\phi}^s$ replaced by $x_1,(\Sp^{-1})^*\hat{\varphi}^s$ respectively
and $\phi^u,\phi^s$ replaced by $(\Sp^{-1})^*\phi^{u}, (\Sp^{-1})^*\phi^{u}$ respectively. In addition, when ${\tb p}_1$ satisfies \eqref{eq: skew p1 bound}, it also holds for $\tilde{\tb{p}}_1$ (with $\Gamma$ replaced by $\tilde{\Gamma}=\tilde{\Gamma}^u \cap \tilde{\Gamma}^s \cap \{\tau = 0\}$), the subprincipal symbol of $\tilde{P}$.
\end{prop}
Next we prove that wavefront sets of functions and operators are preserved under the action of $\Sp$ and its quantization.
\begin{prop}
\begin{align}
\WF_{\cu}^{s,r}(Tv) = \Sp(\WF_{\cu}^{s,r}(v))              \label{wf_transform}
\end{align}
\label{prop_wf_transform}
\end{prop}
\begin{proof}
We consider the complement of $\WF_{\cu}^{s,r}(v)$. Suppose $x_0 \notin \WF_{\cu}^{s,r}(v)$, then there is an $A \in \Psi_{\mathrm{cu}}^{0,0}(M)$ such that $Au \in H^{s,r}_{\cu}(M)$ and $A$ is elliptic at $x_0$. Since $T^*T =\Id_M$ microlocally, this is equivalent to the existence of $B \in \Psi_{\mathrm{cu}}^{0,0}(\tilde{M})$ such that $B Tv \in H^{s,r}_{\cu}(\tilde{M})$ with $B=TAT^*$, which is elliptic at $\Sp(x_0)$. This shows $\Sp((\WF_{\cu}^{s,r}(v))^c) \subset (\WF^{s,r}(Tv))^c$. The converse also holds, which implies $\WF^{s,r}(Tv) = \Sp(\WF^{s,r}(v))$. 
\end{proof}

Proposition \ref{prop_decay}, Proposition \ref{prop: dynamics after conjugation} and \eqref{PTcommute} imply that, the estimate (\ref{main}) in which $\tilde{P}$ and $Tv$ playing the role of $P$ and $v$ implies the same estimate for $P$ and $v$:
\begin{thm}
\label{thm_main_app}
For $\lambda$ satisfying \eqref{eq: lambda condition}, $0<\alpha<1$ and $s,N\in \R$, there exists $B \in \Psi_{\mathrm{cu},\alpha}^{0,0,r}(\tilde{M},\tilde{\Gamma}^u)$ which is  elliptic on $\tilde{\Gamma}$ and the front face in blown-up ${}^{\mathrm{cu}}\bar{T}^*\tilde{M}$, and $B_1,G_0 \in \Psi_{\mathrm{cu},\alpha}^{0,0,r}(\tilde{M},\tilde{\Gamma}^u)$ with $\WF_{\mathrm{cu},\alpha}'(B_1) \cap  \tilde{\Gamma}^u \cap \{\tilde{\tau}=0\} = \emptyset, \WF_{\mathrm{cu},\alpha}'(G_0)$ contained in a fixed neighborhood of $\tilde{\Gamma}$ and $G_0$ is elliptic near $\tilde{\Gamma}$, such that:
\begin{align}
\begin{split}
||B Tv||_{s,s,0} \lesssim & ||B_1Tv||_{s+1-\lambda \alpha, s,0} + ||G_0TPv||_{s-m+2-\lambda\alpha,s-m+2-\alpha,0}  \\
& + ||Tv||_{-N,-N,0}. 
\end{split}  
\label{main_app}
\end{align}
\end{thm}
\begin{proof} Let $\alpha, s, N, \lambda$ be as stated above, 
then by Proposition \ref{prop: dynamics after conjugation}, we can apply (\ref{main}) with $\tilde{P},Tv$ being $P,v$ there, we obtain
\begin{align*}
\begin{split}
||B Tv||_{s,s,r} \lesssim & ||B_1Tv||_{s+1-\lambda \alpha, s,r} + ||G_0\tilde{P}Tv||_{s-m+2-\lambda\alpha,s-m+2-\alpha,r}  
\\& + ||Tv||_{-N,-N,r},  
\end{split}  
\end{align*}
where $B \in \Psi_{\mathrm{cu},\alpha}^{0,0,r}(\tilde{M},\tilde{\Gamma}^u)$ which is elliptic near $\tilde{\Gamma}=\tilde{\Gamma}^u\cap \tilde{\Gamma}^s \cap \{\tilde{\tau}=0\}$ and the front face. $B_1,G_0 \in \Psi_{\mathrm{cu},\alpha}^{0,0,r}(\tilde{M},\tilde{\Gamma}^u)$ and $\WF_{\mathrm{cu},\alpha}'(B_1) \cap  \tilde{\Gamma}^u \cap \{\tilde{\tau}=0\} = \emptyset,\WF_{\mathrm{cu},\alpha}'(G_0) \cap {\Gamma}^u = \emptyset$, 
Then we apply (\ref{PTcommute}) to obtain
\begin{align}
\begin{split}
||B Tv||_{s,s,r} \lesssim & ||B_1Tv||_{s+1-\lambda \alpha, s,r} + ||G_0TPv||_{s-m+2-\lambda\alpha,s-m+2-\alpha,r}  
\\&+ ||Tv||_{-N,-N,r}, \label{main2}
\end{split}  
\end{align}
\end{proof}

Combining Proposition \ref{prop_wf_transform} and the microlocal estimate in Theorem \ref{thm_main_app}, we obtain Theorem \ref{thm_propagation_M}.

\section*{Acknowledgements}
The author would like to thank Andr\'as Vasy, who suggested this project and offered helpful discussions and plenty of advice throughout this project.
The author would also like to thank the referee, whose suggestions are very helpful and considerably improved the quality of this paper.

\bibliographystyle{plain}
\bibliography{bib_KdS}

\end{document}